\documentclass[letterpaper,11pt,oneside,reqno]{amsart}

\usepackage{amsfonts,amsmath, amssymb,amsthm,amscd}
\usepackage{bm}
\usepackage{mathrsfs}
\usepackage{yfonts}
\usepackage{verbatim}
\usepackage{xcolor}
\usepackage{lmodern}
\usepackage[colorlinks=false,linkbordercolor=red]{hyperref}

\usepackage{graphicx}
\usepackage[utf8]{inputenc}
\usepackage{latexsym}
\usepackage{upgreek}
\usepackage{relsize}
\usepackage{mathdots}
\usepackage{mathtools}
\usepackage{dsfont}
\usepackage[left=2.5cm, right=2.5cm, top=2.5cm, bottom=3cm]{geometry}
\usepackage{tikz}
\usepackage{enumitem}

\numberwithin{equation}{section}
\renewcommand{\emph}[1]{\textsf{\textit{#1}}}

\usepackage[width=.9\textwidth]{caption}

\let\oldtocsection=\tocsection
\let\oldtocsubsection=\tocsubsection
\renewcommand{\tocsection}[2]{\hspace{0em}\oldtocsection{#1}{#2}}
\renewcommand{\tocsubsection}[2]{\hspace{2em}\oldtocsubsection{#1}{#2}}

\begin{document}
\fontdimen8\textfont3=0.5pt  


\def \Ai {{\rm Ai}}
\def \Pf {{\rm Pf}}
\def \sgn {{\rm sgn}}
\def \SS {\mathcal{S}}
\renewcommand{\P}{\mathbb{P}}
\newcommand{\E}{\mathbb{E}}
\renewcommand{\L}{\mathbb{L}}
\newcommand{\EE}{\ensuremath{\mathbb{E}}}
\newcommand{\Var}{\mathrm{Var}}
\newcommand{\qq}[1]{(q;q)_{#1}}
\newcommand{\PP}{\ensuremath{\mathbb{P}}}
\newcommand{\R}{\ensuremath{\mathbb{R}}}
\newcommand{\C}{\ensuremath{\mathbb{C}}}
\newcommand{\Z}{\ensuremath{\mathbb{Z}}}
\newcommand{\N}{\ensuremath{\mathbb{N}}}
\newcommand{\Q}{\ensuremath{\mathbb{Q}}}
\newcommand{\T}{\ensuremath{\mathbb{T}}}
\newcommand{\Y}{\ensuremath{\mathbb{Y}}}
\newcommand{\I}{\ensuremath{\mathbf{i}}}
\newcommand{\BD}{\mathbb{BD}}
\newcommand{\Real}{\ensuremath{\mathfrak{Re}}}
\newcommand{\Imag}{\ensuremath{\mathfrak{Im}}}
\newcommand{\subs}{\ensuremath{\mathbf{Subs}}}
\newcommand{\Sym}{\ensuremath{\mathsf{Sym}}}
\newcommand{\dist}{\textrm{dist}}
\newcommand{\Res}[1]{\underset{{#1}}{\mathbf{Res}}}
\newcommand{\Resfrac}[1]{\mathbf{Res}_{{#1}}}
\newcommand{\Sub}[1]{\underset{{#1}}{\mathbf{Sub}}}
\newcommand{\la}{\lambda}
\newcommand{\ta}{\theta}
\newcommand{\labold}{\boldsymbol{\uplambda}}
\def\note#1{\textup{\textsf{\color{blue}(#1)}}}
\newcommand{\oldrho}{\rho}
\renewcommand{\rho}{\varrho}
\renewcommand{\le}{\leqslant}
\renewcommand{\leq}{\leqslant}
\renewcommand{\ge}{\geqslant}
\renewcommand{\geq}{\geqslant}
\newcommand{\msf}{\mathsf}
\newcommand{\fullspace}{\mathbb{Z}}
\newcommand{\vspan}{\mathrm{span}}
\newcommand{\ov}{\overline}


\newcommand{\Zhalf}{\Z^2_{\rm half}}

\newcommand{\icond}{\hyperlink{icond}{(i)}}
\newcommand{\iicond}{\hyperlink{iicond}{(ii)}}
\newcommand{\iiicond}{\hyperlink{iiicond}{(iii)}}
\newcommand{\Icond}{\hyperlink{Icond}{(\mathrm{I})}}
\newcommand{\IIcond}{\hyperlink{IIcond}{(\mathrm{II})}}
\newcommand{\IIIcond}{\hyperlink{IIIcond}{(\mathrm{III})}}

\newcommand{\Acond}{\hyperlink{Acond}{(\mathrm A)}}
\newcommand{\Bcond}{\hyperlink{Bcond}{(\mathrm B)}}
\newcommand{\BBcond}{\hyperlink{BBcond}{(\mathrm B')}}
\newcommand{\Ccond}{\hyperlink{Ccond}{(\mathrm C)}}
\newcommand{\CCcond}{\hyperlink{CCcond}{(\mathrm C')}}


\usetikzlibrary{patterns}
\usetikzlibrary{shapes.multipart}
\usetikzlibrary{arrows}

\tikzstyle{axis}=[->, >=stealth', thick, gray]
\tikzstyle{grille}=[dotted, gray]
\tikzstyle{path}=[->, >=stealth', thick]


\newtheorem{theorem}{Theorem}[section]
\newtheorem{conjecture}[theorem]{Conjecture}
\newtheorem{lemma}[theorem]{Lemma}
\newtheorem{proposition}[theorem]{Proposition}
\newtheorem{corollary}[theorem]{Corollary}

\newtheorem{theoremintro}{Theorem}
\renewcommand*{\thetheoremintro}{\Alph{theoremintro}}

\theoremstyle{definition}
\newtheorem{remark}[theorem]{Remark}

\theoremstyle{definition}
\newtheorem{example}[theorem]{Example}

\theoremstyle{definition}
\newtheorem{definition}[theorem]{Definition}

\theoremstyle{definition}
\newtheorem{definitions}[theorem]{Definitions}


\title{\large Random walk on nonnegative integers in beta distributed random environment}

\author[G. Barraquand]{Guillaume Barraquand}
\address{G. Barraquand,
Laboratoire de physique de l’école normale supérieure, ENS, Université PSL, CNRS, Sorbonne Université, Université de Paris, Paris, France}
\email{guillaume.barraquand@ens.fr}
\author[M. Rychnovsky]{Mark Rychnovsky}
\address{M. Rychnovsky, Department of Mathematics, USC, 3620  S. Vermont Ave. CA 90089}
\email{mrychnov@gmail.com}

\maketitle

\begin{abstract} 
We consider  random walks on the nonnegative integers in a space-time dependent random environment. We assume that transition probabilities are given by independent $\mathrm{Beta}(\mu,\mu)$  distributed random variables, with a specific behaviour at the boundary, controlled by an extra parameter $\eta$. We show that this model is exactly solvable and prove a formula for the mixed moments of the random heat kernel. We then provide two formulas that allow us to study the large-scale behaviour. The first involves a Fredholm Pfaffian, which we use to prove a local limit theorem describing how the boundary parameter $\eta$ affects the return probabilities. The second is an explicit series of integrals, and we show that non-rigorous critical point asymptotics suggest that the large deviation behaviour of this half-space random walk in random environment is the same as for the analogous random walk on $\mathbb{Z}$.
\end{abstract}

\setcounter{tocdepth}{1}
\tableofcontents


\section{Introduction and main results}
\label{sec:introduction}

Consider a random walk $X(t)$ on the integers, such that if at time $t$, $X(t)=x$, then at time $t+1$, $X(t+1)=x+1$ with probability $p_{t,x}$ and   $X(t+1)=x-1$ with probability $1-p_{t,x}$. In the homogeneous case where the parameters $p_{t,x}$ are all equal to some fixed $p\in (0,1)$, this is just a simple random walk, and the model is very well-understood. If the $p_{t,x}$ are disordered, for instance if the $p_{t,x}$ form an i.i.d. family of random variables in $(0,1)$, the model is more complicated, but it is known \cite{rassoul2005almost} that for almost every family of parameters $( p_{t,x})_{t,x\in \mathbb Z}$, the random walk $X(t)$ behaves at large scale as a Brownian motion, so that the disorder seems to have no influence on the model, at least at large scale. 

\medskip 
The disorder has an effect, however, on the large deviation principle satisfied by the random walk. For instance, the large deviation rate function increases in presence of disorder \cite{yilmaz2010differing, rassoul2014quenched}. Further, \cite{barraquand2017random} studied second order corrections to the large deviation principle. As a consequence, \cite{barraquand2017random} showed  that when the $p_{t,x}$ are i.i.d. uniform, if one considers $N$ random walks $X_1, \dots, X_N$ in the same random environment starting from $X_i(0)=0$, and if one scales the number of walkers as $N=e^{ct}$ for some constant $c>0$, then, at time $t$,  the maximum of the random walk positions  $X_i(t)$ has asymptotically Tracy-Widom GUE distributed fluctuations on the scale $t^{1/3}$. This is radically different from the behaviour of the maximum of the same number of independent simple random walks (the fluctuations would be of order $1$), and more similar to interface growth models in the Kardar-Parisi-Zhang universality class \cite{kardar1986dynamic}. To arrive at this result, \cite{barraquand2017random} introduced an exactly solvable model of random walk in random environment (RWRE), called the beta RWRE (defined more precisely in Definition \ref{def:fullspacebetaRWRE} below). 

\medskip 
In this paper, we introduce an integrable RWRE on the non-negative integers, which is the half-space counterpart of the beta RWRE from \cite{barraquand2017random}, and we call it the half-space beta RWRE (see Definition \ref{def:halfspacebetaRWRE}). Finding half-space variants of stochastic integrable models is a difficult task in general. This was accomplished for models related to Schur measures \cite{baik2001algebraic, baik2001asymptotics, borodin2005eynard, baik2018pfaffian, betea2018free}, for the log-gamma polymer \cite{o2014geometric},  for the stochastic six-vertex model \cite{barraquand2018stochastic} (only in a specific case), and a general approach for models related to Macdonald measures \cite{borodin2014macdonald} was developed in \cite{barraquand2018half}. The beta RWRE model, however, does not fit in the realm of Macdonald measures. Its solvability is related to the  higher-spin stochastic six-vertex model \cite{povolotsky2013integrability, borodin2017family, corwin2015stochastic, borodin2016higher}, for which no half-space version is known (despite some progress in \cite{mangazeev2019boundary} on the algebraic side, see more details in Remark \ref{rem:YangBaxter} below). 

\medskip 
The integrability of the half-space beta RWRE is manifest in the existence of a simple integral formula for joint moments of the random heat kernel (Theorem \ref{theo:moments}). This is our main result.  We show that the information contained in this moment formula can be repackaged in two ways in order to deduce asymptotics. 
\begin{enumerate}[leftmargin=1.6em]
    \item A Fredholm Pfaffian formula, which is convenient to analyze the typical behaviour of the random walk. We demonstrate how it can be used to prove a local central limit theorem for return probabilities at $0$. The analogous question has been studied for  RWRE on the whole integers in several works (see the discussion in Section \ref{sec:quenchedlocallimits}), but we are not aware of any previous result in the literature about half-space RWRE. We demonstrate in Corollary \ref{cor:returnproba} that the result depends non trivially on the boundary condition.
    \item  We also prove another formula, an explicit convergent series of integrals, which characterises the probability for the RWRE  to start from some $x>0$ and to arrive at the boundary. This formula is more suitable to taking asymptotics in the large deviation regime, that is  when $x$ is proportional to time $t$. Under these scalings, we conjecture, based on non-rigorous critical point asymptotics, that the half-space beta RWRE follows a large deviation principle with Tracy-Widom GUE distributed second order corrections, similar to the full-space beta RWRE.
\end{enumerate}

\subsection{An exactly solvable RWRE on the positive integers}

\begin{definition}[Half-space beta RWRE] \label{def:halfspacebetaRWRE} Let $\eta, \mu>0$. Let $\left(W_{t,x}\right)_{x \in \mathbb{Z}_{\geq 1}, t \in \mathbb{Z}}$ be a collection of independent beta distributed random variables. Recall that the beta distribution with parameters $(\alpha, \beta)$, denoted $\mathrm{Beta}(\alpha, \beta)$, is a continuous probability distribution with density 
$$\mathds{1}_{x \in [0,1]} \frac{ x^{\alpha-1} (1-x)^{\beta-1}}{B(\alpha,\beta)},$$
where $B(\alpha,\beta)=\frac{\Gamma(\alpha)\Gamma(\beta)}{\Gamma(\alpha+\beta)}$. 
We assume that 
\begin{equation} 
W_{t,x} \sim \begin{cases} \mathrm{Beta}(\mu, \mu) \mbox{ if }x\geq 2,\\
\mathrm{Beta}(\mu, \eta)  \mbox{ if }x=1. \end{cases}
\label{eq:weightshalfspace}
\end{equation}
The half-space beta random walk in random environment is a discrete random walk $t\mapsto X(t)$  in $\mathbb{Z}_{\geq 0}$, such that for any $x\geq 1$,  
$$\msf{P}(X(t+1)=x+1|X(t)=x)=W_{t,x}, \qquad \msf{P}(X(t+1)=x-1|X(t)=x)=1-W_{t,x}.$$ 
We impose further the condition that if at time $t$, $X(t)=0$, then  almost-surely $X(t+1)=1$, so that all increments are given by $\pm1$ and $X_t$ remains in $\mathbb{Z}_{\geq 0}$. We use the letter $\msf P$ to denote the probability measure on random walk paths, conditionally on the environment,  that is the  collection of random variables  $W=\left( W_{t,x}\right)_{x \in \mathbb{Z}_{\geq 1}, t \in \mathbb{Z}}$. 
The letters $\mathbb P, \mathbb E$ will denote the probability measure and expectation over the environment.  The choice of weight distribution \eqref{eq:weightshalfspace} is not arbitrary, but it will become clear only after the proof of our main result (see  Remark \ref{rem:origin} below).  
\end{definition}
  
\begin{remark}
	Averaging over the environment, the measure on paths becomes a simple random walk $\overline X(t)$  on $\mathbb{Z}_{\geq 0}$, with steps equal to $+1$ and $-1$ chosen with equal probability when $\overline X(t)\geq 2$, and such that when $\overline X(t)=1$, it jumps to $\overline X(t+1)=0$ with probability $\frac{\eta}{\eta+\mu}$, and jumps to $\overline X(t+1)=2$ with probability $\frac{\mu}{\eta+\mu}$. Hence, the parameter $\eta$ controls the attractiveness of the boundary at $0$, while the parameter $\mu$ controls the variance of the environment noise, and can be seen as the  temperature in this statistical mechanical model. In the diffusive limit, $\overline X(t)$ converges to the reflected Brownian motion (i.e. the absolute value of a standard Brownian motion). 
	\label{rem:averagerandomwalk}
\end{remark}

Let us also define the random discrete heat kernel
\begin{equation}
    \label{eq:defP}
    \mathsf P_{s,t}(x,y) = \mathsf P\left( X_t=y \vert X_s=x \right).
\end{equation}
We may write  $\mathsf P_{s,t}(x,y)$ for all $s,t\in \mathbb Z$ and $x,y\in \mathbb Z_{\ge 0}$, but follow the convention that it is nonzero only if $s \le t$ and if $s+x$ and $y+t$ have the same parity. 
Our main result is the following mixed moment formula for transition probabilities of the half-space beta RWRE. 
\begin{theorem} Let $\mu, \eta>0$. For $1\leq x_1 \leq \dots  \leq x_k$ and $t>0$, with $t+x_i$ odd, we have
\begin{multline}\label{eq:momentformula}
\E[ \msf{P}_{0,t}(x_1,1)\cdots \msf{P}_{0,t}(x_k,1)] =(-2)^k (\mu+\eta)_k   \int_{\I \R} \frac{dz_1}{2 \pi \I}\dots \int_{\I \R} \frac{dz_k}{2 \pi \I} \prod_{1 \leq a< b \leq k} \frac{z_a-z_b}{z_a-z_b-1} \frac{z_a+z_b}{z_a+z_b+1} 
\\ \times \prod_{i=1}^k \left( \frac{z_i^2}{z_i^2-\mu^2} \right)^{t/2+1} \left( \frac{z_i-\mu}{z_i+\mu} \right)^{(x_i-1)/2} \frac{1}{z_i(z_i+\eta)},
\end{multline}
where throughout the paper, for any $a\in \mathbb C$, the Pochhammer symbol $(a)_k$ denotes the raising factorial $(a)_k= a(a+1) \dots (a+k-1)$.
\label{theo:moments}
\end{theorem}
Theorem \ref{theo:moments} is proved in Section \ref{sec:betheansatz}. The starting point of the proof is that random variables $\mathsf P_{0,t}(x,1)$ have the same distribution as $\mathsf P_{-t,0}(x,1)$, which satisfy the recurrence relation 
\begin{equation}\begin{cases}  \mathsf P_{-t,0}(x,1) = W_{-t,x} \mathsf P_{-(t-1),0}(x+1,1)+(1- W_{-t,x}) \mathsf P_{-(t-1),0}(x-1,1) \mbox{ if }x\ge 1 \\
\mathsf P_{-t,0}(0,1) = \mathsf P_{-(t+1),0}(1,1).
\end{cases}
\label{eq:recurrencerelation}
\end{equation} 
The random recurrence relation \eqref{eq:recurrencerelation} implies a recurrence relation for mixed moments in \eqref{eq:momentformula}, which we solve explicitly via Bethe ansatz, leading to Theorem \ref{theo:moments}. The approach is similar to the one employed for the full-space beta RWRE \cite{barraquand2017random} (see also \cite{corwin2015q, povolotsky2013integrability}), except that an extra boundary condition at $x=1$ must be satisfied. Verifying that \eqref{eq:momentformula} satisfies this boundary condition turned out to be challenging -- the discrete setting is much more involved than analogous arguments for other half-space continuous models \cite{borodin2016directed}. Our proof reduces to showing that some polynomial of $k$ variables may be decomposed as a sum of polynomials satisfying adequate symmetry properties. We could not exhibit an explicit decomposition for the  specific polynomial at hand, but we prove that it exists. Section \ref{sec:boundarycondition} is devoted to characterizing the space of all polynomials admitting such decomposition.   
\begin{remark}
Theorem \ref{theo:moments} provides a moment formula for transition probabilities $\msf{P}_{0,t}(x,1)$ from $x$ to $1$. We may also characterize the distribution of the probabilities $\msf{P}_{0,t}(x,0)$ from $x$ to  $0$,  since 
\begin{equation}
    \msf{P}_{0,t}(x,0) \overset{(d)}{=} W \msf P_{0,t-1}(x,1) ,
    \label{eq:probato1and0}
\end{equation}
where $W$ is a $\mathrm{Beta}(\eta, \mu)$ distributed random variable independent from $\msf P_{0,t-1}(x,1)$. 
\end{remark}
\begin{remark}\label{rem:YangBaxter}
Theorem \ref{theo:moments} shows that the half-space beta RWRE that we have defined in Definition \ref{def:halfspacebetaRWRE} is the appropriate half-space analogue of the exactly solvable beta RWRE from \cite{barraquand2017random}. The solvability of the beta RWRE can be seen as a degeneration of the integrability of the stochastic higher-spin six-vertex model \cite{borodin2017family, corwin2015stochastic, borodin2016higher}, solvable via algebraic Bethe ansatz. However, in general, finding half-space variants of integrable models is a difficult task. In the context of algebraic Bethe ansatz solvable models, finding a half-space analogue often boils down to finding  a so-called $K$ matrix, which, together with the $R$ matrix solution of the Yang-Baxter equation, satisfies some $KRKR=RKRK$ type reflection relation. This question was studied in \cite{mangazeev2019boundary} for the stochastic higher-spin six-vertex model, and explicit solutions for the $K$ matrix were found. In the case that should be related to the $q$-Hahn vertex model from \cite{povolotsky2013integrability} and the beta RWRE after a $q\to 1$ limit,  a solution of the reflection equation is found  in \cite{mangazeev2019boundary} modulo  conjectural combinatorial identities which, to the best of our knowledge, are still unproven (see \cite[Eq. (5.10)]{mangazeev2019boundary}). It would be interesting to understand better the relation between the half-space beta RWRE and these solutions to the reflection equation, but we leave this for future consideration. 
In a different direction, in the special case of the stochastic spin $1/2$ six-vertex model, an exactly solvable analogue model was studied in \cite{barraquand2018stochastic}, though the model considered there does not have a free boundary parameter. 
\end{remark}

\subsection{Hankel transforms} 
\label{sec:introHankel}
The moment formula from Theorem \ref{theo:moments} fully determines the distribution of the random variables $\msf P_{0,t}(x,0)$ (since these random variables belong to $(0,1)$). However, to analyze asymptotically their distribution, it is convenient to obtain a compact formula for some moment generating series. Due to the Pochhammer symbol $(\mu+\eta)_k$ in front of the integral in \eqref{eq:momentformula}, it turns out that it is not very convenient to consider the usual moment generating series, i.e. the Laplace transform, instead we use a deformed  moment generating series called the Hankel transform. 
\begin{definition}\label{def:Hankeltransform}
Let 
\begin{equation}
F_{\nu}(\zeta)=\sum_{k=0}^{\infty} \frac{\zeta^k}{k! (\nu)_k}.
\label{eq:defFnu}
\end{equation}
We define the Hankel transform of order $\nu$ of  a function $f(x)$ by 
$$\mathcal{F}_{\nu} f(\zeta) = \int_0^{\infty}{F}_{\nu}(\zeta x) f(x) dx,$$
and for a nonnegative random variable $X$, $\E[F_{\nu}(\zeta X)]$ will be called the Hankel transform of $X$. 
\end{definition}
The Hankel transform is closely connected to the Laplace transform, and like the Laplace transform, the Hankel transform of a nonnegative random variable completely determines its distribution (Lemma \ref{lem:hankelcharacterize}). The Hankel transform also satisfies a variant of L\'evy's continuity theorem: the limit of the transforms of a sequence of random variables is the transform of the limiting distribution (Proposition \ref{prop:Levycontinuity}). Even better, in some asymptotic  regimes, it  turns out that the limit of the Hankel transform of a sequence of random variables directly yields the cumulative distribution function of the limiting distribution (Lemma \ref{lem:hankellimit}). More generally, the Hankel transform could be inverted explicitly, but we will eventually not need to use any inversion formula. 

The Hankel transform of the random variable $\msf{P}_{0, 2t}(1,1)$ takes a particularly simple form and can be written as a Fredholm Pfaffian. The following proposition is proved in Section \ref{sec:Cauchyformula}.
\begin{proposition} \label{prop:CauchytypeFredholm}
Let $\mu, \eta>0$. For any $\zeta\in \mathbb C$, we have 
\begin{equation}
\mathbb E\left[F_{\mu+\eta}(-\zeta  \msf{P}_{0, 2t}(1,1)) \right] = \mathrm{Pf}(J+\zeta K)_{\mathbb L^2(\I\R\times \lbrace 1,2\rbrace)},
\label{eq:CauchytypeFredholm}
\end{equation}
where $K$ acts on $L^2(\I\R\times \lbrace 1,2\rbrace)$ via the $2\times 2$ skew-symmetric matrix kernel
\begin{subequations}
	\begin{align}
	K_{11}(z,w) &= \frac{z-w}{z+w+1}\left( \frac{z^2}{z^2-\mu^2} \right)^t  \frac{z}{(z^2-\mu^2)(z^2-\eta^2)} ,\\
	K_{12}(z,w) &= -K_{2,1}(w,z) = \frac{z+w}{z-w+1}\left( \frac{z^2}{z^2-\mu^2} \right)^{t}  \frac{z}{(z^2-\mu^2)(z^2-\eta^2)},\\
	K_{22}(z,w) &= \frac{w-z}{-z-w+1}\left( \frac{w^2}{w^2-\mu^2}\right)^{t}\frac{w}{(w^2-\mu^2)(w^2-\eta^2)}.
	\end{align}
	\label{eq:defK}
\end{subequations}
The definition of a Fredholm Pfaffian as in \eqref{eq:CauchytypeFredholm} is given in Appendix \ref{sec:backgroundPfaffian}, along with some additional background on Fredholm determinants and Pfaffians. 
\end{proposition}
\begin{remark} The way Proposition \ref{prop:CauchytypeFredholm} is deduced from Theorem  \ref{theo:moments} is inspired by proofs of so-called ``large contour formulas'' in the context of ASEP \cite{tracy2008fredholm}  and  Macdonald processes \cite[Section 3.2.3]{borodin2014macdonald}, also termed ``Cauchy-type Fredholm determinants'' in \cite[Section 3.2]{borodin2014duality}. The occurrence of determinants in full-space models  comes from symmetrizing the integrand of moment formulas. In the half-space setting, we symmetrize the integrand in \eqref{eq:momentformula} with respect to the action of the hyperoctahedral group, and this yields a Pfaffian. For the full-space Beta RWRE, Cauchy-type formulas have been derived in \cite{thiery2016exact}. 
\end{remark}
We may also derive explicit formulas for the Hankel transform of $\msf{P}_{0,t}(x,1)$ when $x$ is arbitrary. The following proposition is proved in Section \ref{sec:MellinBarnesformula}. 
\begin{proposition}[Mellin-Barnes type Hankel transform] \label{prop:mellinbarnes}  Let $\mu, \eta>0$ and let $x,t$ be positive integers so that $x+t$ is odd. Then, for any $\zeta \in \mathbb{C}\setminus \mathbb R_{<0}$, 
\begin{multline}
\label{eq:MellinBarnesFinal}
\E\left[F_{\mu+\eta} \left( -\zeta \mathsf P_{0,t}(x,1) \right) \right]=
1+\sum_{\ell=1}^{\infty} \frac{1}{\ell!} \int_{\mathcal C_{\mu+ 1/2}^{\pi/3}}\frac{d z_1}{2 \pi \I} \dots \int_{\mathcal C_{\mu + 1/2}^{\pi/3}} \frac{d z_{\ell}}{2 \pi \I}  \oint_{\gamma} \frac{dw_1}{2 \pi \I} \dots \oint_{\gamma} \frac{d w_{\ell}}{2 \pi \I} \\ 
\det  \left[\frac{1}{z_i-w_j} \right]_{i,j=1}^{\ell} \prod_{i=1}^{\ell} \frac{\pi}{\sin(\pi (w_i-z_i))} 
\prod_{1 \leq i <j \leq \ell} \frac{\Gamma(z_i+w_j) \Gamma(w_i+z_j)}{\Gamma(w_i+w_j) \Gamma(z_i+z_j)} \\
\times \prod_{i=1}^{\ell} \zeta^{z_i-w_i} \frac{ \Gamma(z_i+w_i)}{\Gamma(2w_i)} 
\frac{ \Gamma(w_i+\eta) \Gamma(w_i+\mu) \Gamma(w_i-\mu)}{ \Gamma(z_i+\eta) \Gamma(z_i+\mu) \Gamma(z_i-\mu)}. \\
\times \left( \frac{\Gamma(z_i-\mu) \Gamma(w_i+\mu)}{\Gamma(w_i-\mu)\Gamma(z_i+\mu)} \right)^{(x-1)/2} 
 \left[ \left( \frac{\Gamma(z_i)}{\Gamma(w_i)} \right)^2 \frac{\Gamma(w_i-\mu) \Gamma(w_i+\mu)}{\Gamma(z_i-\mu) \Gamma(z_i+\mu)} \right]^{t/2}. 
\end{multline}
where $\gamma$ is a positively oriented circle around $\mu$ with radius smaller than $1/4$ (containing neither $-\mu$ nor $-\eta$) and the contour $\mathcal C_{\mu+1/2}^{\pi/3}$ is defined as follows. For any $a\in \mathbb R$ and $\theta\in (0, \pi)$, we define the contour $\mathcal C_{a}^{\theta}$ to be an infinite curve formed by the union of two semi-infinite rays between $a$ and $\infty e^{\pm\I\pi/3}$, which leave $a$ in directions $\pm\I\pi/3$. The contour is oriented from bottom to top. 
\end{proposition}

\subsection{Return probability asymptotics} 
The Fredholm Pfaffian formula from Proposition \ref{prop:CauchytypeFredholm}  characterizes the distribution of the return probability at the point $1$. We can also characterize return probabilities at $0$ -- which may appear as a more natural quantity -- since we have the equality in distribution between return probabilities
\begin{equation}
\msf{P}_{0,t}(0,0) \overset{(d)}{=}  W \msf{P}_{0,t-2}(1,1),
\label{eq:returnproba1and0}
\end{equation}
where $W$ is a $\mathrm{Beta}(\eta, \mu)$ distributed random variable independent from $\msf P_{0,t-2}(1,1)$. Our main asymptotic result is the following local limit theorem for the return probability at the boundary. We explain in Section \ref{sec:returnprobaproof} how this result is  deduced from Proposition \ref{prop:CauchytypeFredholm} after an asymptotic analysis of the Fredholm Pfaffian and using general properties of the Hankel transform discussed in Section \ref{sec:hankelgeneral}.
\begin{corollary} \label{cor:returnproba}
Assume $\mu+\eta>\frac{1}{2}$. Then we have the weak convergence 
\begin{equation}
\sqrt{t} \msf P_{0,t}(1,1) \xRightarrow[t\to\infty]{} \frac{2}{\sqrt{2\pi}}\frac{\mathrm{Gamma}(\eta+\mu)}{2\mu} ,
\label{eq:returnproba1}
\end{equation} 
where $\mathrm{Gamma}(\theta)$ denotes the gamma distribution with shape parameter $\theta$, that is the continuous probability distribution with density 
$$ \mathds{1}_{x>0} x^{\theta-1}e^{-x}dx.$$
In view of \eqref{eq:returnproba1and0}, we also have that 
\begin{equation} \sqrt{ t} \msf P_{0,t}(0,0) \xRightarrow[t\to\infty]{} \frac{2}{\sqrt{2\pi}}\frac{\mathrm{Gamma}(\eta)}{2\mu} .
\label{eq:returnproba0}
\end{equation}
\end{corollary} 
In \eqref{eq:returnproba1} and \eqref{eq:returnproba0}, we have intentionally not simplified the factors of $2$, because the prefactor $\frac{2}{\sqrt{2\pi}}$ should be interpreted as the density at $0$ of a reflected Brownian motion. We refer to Section  \ref{sec:quenchedlocallimits}, and in particular Conjecture \ref{conj:half-spacefunctional} below for  a more detailed explanation. 
\begin{remark} 
The statement of Corollary \ref{cor:returnproba} remains valid for any starting point $x> 1$ instead of $0$ or $1$.  Indeed, we show in Section \ref{sec:arbitrarystratingpoint} that Corollary \ref{cor:returnproba} also implies that for any fixed $x\geq 0$, 
\begin{equation*}
\sqrt{t} \msf P_{0,t}(x,1) \xRightarrow[t\to\infty]{} \frac{2}{\sqrt{2\pi}}\frac{\mathrm{Gamma}(\eta+\mu)}{2\mu}
\label{eq:returnprobax1}
\end{equation*} 
and 
\begin{equation*} \sqrt{ t} \msf P_{0,t}(x,0) \xRightarrow[t\to\infty]{} \frac{2}{\sqrt{2\pi}}\frac{\mathrm{Gamma}(\eta)}{2\mu}.
\label{eq:returnprobax0}
\end{equation*}
\label{rem:arbitrarystartingpoint}
\end{remark}
\subsection{About quenched local limit theorems and some conjectures} 
\label{sec:quenchedlocallimits}
In this Section, we will compare Corollary \ref{cor:returnproba} to an analogous result for RWRE on $\mathbb Z$, and discuss some open questions. Let us first recall the precise definition of the full-space beta RWRE, introduced in  \cite{barraquand2017random}. 
\begin{definition}\label{def:fullspacebetaRWRE}
Let $\alpha, \beta>0$. Let $\left(W_{t,x}\right)_{x, t \in \mathbb{Z}}$ be a collection of independent beta distributed random variables.
We assume that for all $t,x \in \mathbb Z$
$$
W_{t,x} \sim  \mathrm{Beta}(\alpha, \beta).$$
The (full-space) beta RWRE is a discrete time random walk $t\mapsto X(t)$  in $\mathbb{Z}$, such that for any $x\in \mathbb Z$,  
$$\msf{P}(X(t+1)=x+1|X(t)=x)=W_{t,x}, \qquad \msf{P}(X(t+1)=x-1|X(t)=x)=1-W_{t,x}.$$ 
 We use again the letter $\msf P$ to denote the probability measure on random walk paths, conditionally on the environment; and  the letters $\mathbb P, \mathbb E$ to denote the probability measure and expectation over the environment. We will also consider the random heat kernel 
 \begin{equation*}
    \label{eq:defPfullspace}
    \mathsf P^{\fullspace}_{s,t}(x,y) = \mathsf P\left( X_t=y \vert X_s=x \right).
\end{equation*}
We use the superscript $\fullspace$ for observables of full-space models, such as $\msf P^{\fullspace}$, or $\mathsf Z_t^{\fullspace}$ below, in order to distinguish them from their half-space counterpart which have no superscript. 
 \end{definition}
The following analogue of Corollary \ref{cor:returnproba} was first derived in the physics paper  \cite{thiery2016exact}. We provide a rigorous proof in Appendix \ref{sec:localasymptoticsfullspace}. 
\begin{proposition}
Let $\alpha, \beta>0$  satisfying $\alpha+\beta>1/2$. Then we have the weak convergence 
$$ \sqrt{t} \msf P^{\fullspace}_{0,t}\left(\left(\tfrac{\beta-\alpha}{\alpha+\beta} \right)t-x\sqrt{t},\;0\right) \xRightarrow[t\to\infty]{} g_{\sigma}(x) \frac{\mathrm{Gamma}(\alpha+\beta)}{\alpha+\beta} ,$$
where $g_{\sigma}$ is the density of the centered Gaussian distribution with mean $0$ and variance  $\sigma^2=\frac{4\alpha\beta}{(\alpha+\beta)^2}$. 
\label{prop:localasymptoticsfull-space}
\end{proposition}
 Local limit theorems analogous to Proposition \ref{prop:localasymptoticsfull-space}  have been proved in a number of references for full-space models: for a general class of random walks in space time i.i.d. environments, it is proved in  \cite{boldrighini1997almost, boldrighini1999central} (see also \cite{boldrighini2007random, deuschel2017quenched}) that point to point probabilities, when rescaled by $\sqrt{t}$, converge to some functional of the noise seen from the arrival point. More precisely, assume that the variables  $W_{t,x}$ have mean $1/2$ (so that the random walk is centered, for simplicity of the exposition), but may follow some arbitrary distribution satisfying mild hypotheses. Then, \cite[Theorem 2]{boldrighini1999central} shows that 
\begin{equation}
    \frac{\sqrt{t} \mathsf P^{\fullspace}_{-t,0}(-x\sqrt{t},0)}{g_{\sigma}(x)} - \mathfrak O(W),
\end{equation}
converges to $0$ in $L^2([0,1]^{\mathbb Z\times\mathbb Z},\mathbb P)$ as $t$ goes to infinity. Here $\mathfrak O$ is a deterministic functional of the environment $W =\left(W_{t,x}\right)_{x , t \in \mathbb{Z}} $, almost surely positive, and such that $\mathbb E[\mathfrak O(W)]=1$.  
From the result of  \cite[Eq. (5)]{thiery2016exact}, that is Proposition \ref{prop:localasymptoticsfull-space} just above, we know that that $\mathfrak O(W)$ is gamma distributed. An analogous problem was also considered in \cite{dunlap2021quenched} for general class of continuous stochastic flows, and in \cite{brockington2021bethe} for uniform Howitt-Warren stochastic flows \cite{Howitt2009consistent, stochasticflowsinthebrownianwebandnet, barraquand2020large}, which is a continuous limit of the (full-space) beta RWRE \cite{LeJan2004products, stochasticflowsinthebrownianwebandnet}.  We may actually guess a very explicit description of the  functional $\mathfrak O(W)$.
For fixed $\varepsilon>0, x\in \mathbb R$ and large $t$, let us decompose the trajectory from time $-t$ to $0$ into a first part from time $-t$ to $-\varepsilon \sqrt{t}$ which will yield the Gaussian density factor, and a second part from time $-\varepsilon \sqrt{t}$ to $0$ which will yield the functional of the environment. We have 
\begin{align} \sqrt{t} \msf P^{\fullspace}_{-t,0}(-x\sqrt{t},0) &=  \sum_{y=-\varepsilon\sqrt{t}}^{\varepsilon\sqrt{t}} \sqrt{t}\msf P^{\fullspace}_{-t,-\varepsilon \sqrt{t}}(-x\sqrt{t},y) \msf P^{\fullspace}_{-\varepsilon \sqrt{t}, 0}(y,0) \label{eq:exactdecomposition}\\
&\approx \left( \frac{1}{2\varepsilon \sqrt{t} } \sum_{y=-\varepsilon\sqrt{t}}^{\varepsilon\sqrt{t}} \sqrt{ t}\msf P^{\fullspace}_{-t,-\varepsilon \sqrt{t}}(-x\sqrt{t},y) +  O(\varepsilon)\right) \times \sum_{z=-\varepsilon\sqrt{t}}^{\varepsilon\sqrt{t}} \msf P^{\fullspace}_{-\varepsilon \sqrt{t}, 0}(z,0).\label{eq:approxdecomposition}
\end{align}
 Going from \eqref{eq:exactdecomposition} to \eqref{eq:approxdecomposition}, we have replaced the first factor by its average value over $y\in [ -\varepsilon\sqrt{t},\varepsilon\sqrt{t} ]$, modulo some $ O(\varepsilon)$ error. This approximation is based on the assumptions that (1) the process $y\mapsto \msf P^{\fullspace}_{-t,-\varepsilon \sqrt{t}}(-x\sqrt{t},y)$ decorrelates on the $O(1)$ scale, and (2) the process $y\mapsto \msf P^{\fullspace}_{-\varepsilon \sqrt{t}, 0}(y,0)$ is approximately  constant (for large $t$) when $y$ varies on scale $O(1)$. The latter was anticipated in \cite{thiery2016exact}, it could also be deduced from \cite[Theorem 2.1]{balazs2019large}, or using arguments similar to those in Section \ref{sec:arbitrarystratingpoint} where we prove an analogous statement in the half-space case. 
 
 We do not attempt in the present paper to  prove  that these assumptions are valid. Nevertheless, assuming \eqref{eq:approxdecomposition},  we may use the central limit theorem proved in \cite{rassoul2005almost} (see also \cite{boldrighini1999central, berard2004almost, bouchet2014quenched} for earlier references under slightly less general assumptions) to show that this average value converges to the average value of the Gaussian density. Hence, we obtain that 
\begin{equation}
\sqrt{t} \msf P^{\fullspace}_{-t,0}(-x\sqrt{t},0) \approx \left( \frac{1}{2\varepsilon} \int_{x-\varepsilon}^{x+\varepsilon} g(y)dy + O(\varepsilon)\right)\times   \lim_{L\to\infty} \sum_{z=-L}^L \msf P^{\fullspace}_{-L, 0}(z,0).
\label{eq:approxdecomposition2}
\end{equation}  
 Let us introduce the notation  $\mathsf Z_t^{\fullspace}(x)$ to denote the same functional as in the RHS of \eqref{eq:approxdecomposition2}, seen from the environment at time $t$ and space point $x$, that is 
\begin{equation}
    \mathsf Z_t^{\fullspace}(x)  := \lim_{L\to\infty} \sum_{z=x-L}^{x+L}  \msf P^{\fullspace}_{t-L, t}(z,x).
    \label{eq:defOfullspace}
\end{equation}  
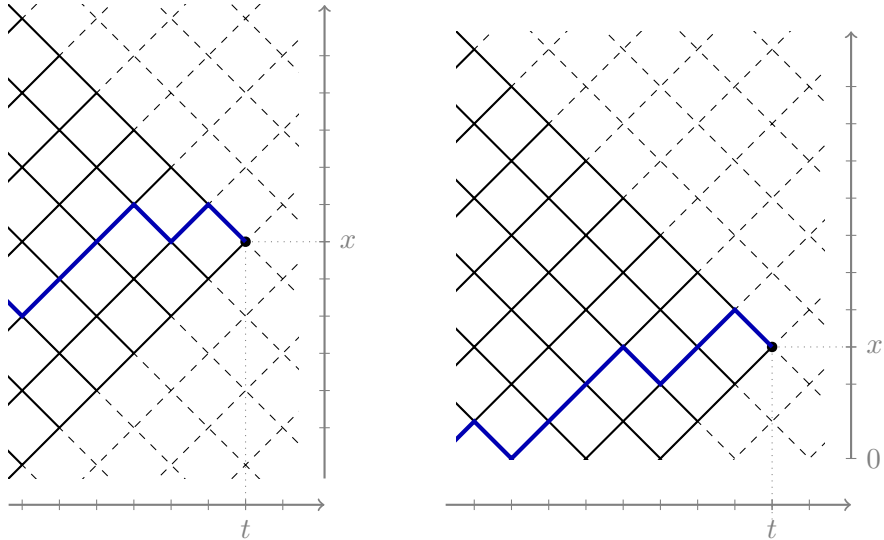
\begin{figure}
    \centering
    \begin{tikzpicture}
\begin{scope}[scale=0.7]
\draw[->, gray, thick] (-4.5,-5) -- (1.5,-5);
\draw[gray] (0,-4.9) -- (0,-5.1) node[anchor=north] {$t$};
\foreach \x in {-6,...,1}
\draw[gray] ({1/sqrt(2)*\x},-4.9) -- ({1/sqrt(2)*\x},-5.1);
\draw[gray, dotted] (0,-4.9) -- (0,0);
\draw[->, gray, thick] (1.5,-4.5) -- (1.5,4.5);
\draw[gray] (1.4,0) -- (1.6,0) node[anchor=west] {$x$};
\foreach \x in {-5,...,5}
\draw[gray] (1.4,{1/sqrt(2)*\x}) -- (1.6,{1/sqrt(2)*\x});
\draw[gray, dotted] (0,0) -- (1.4,0);
\clip (-4.5,-4.5) rectangle(1,4.5);
\begin{scope}[rotate=-45]
\draw[dashed] (-10,-10) grid(10,10);
\draw[thick] (-10,-10) grid(0,0);
\fill (0,0) circle(0.1);
\draw[ultra thick, blue!70!black] (-3,-6) -- (-3,-4) -- (-2,-4) -- (-2,-1)-- (-1,-1) -- (-1,0) -- (0,0); 
\end{scope}
\end{scope}

\begin{scope}[scale=0.7, xshift=10cm, yshift=-2cm]
\draw[->, gray, thick] (-6.2,-3) -- (1.5,-3) ;
\draw[gray] (0,-2.9) -- (0,-3.1) node[anchor=north] {$t$};
\foreach \x in {-8,...,1}
\draw[gray] ({1/sqrt(2)*\x},-2.9) -- ({1/sqrt(2)*\x},-3.1);
\draw[gray, dotted] (0,-3.4) -- (0,0);
\draw[->, gray, thick] (1.5,-2.12) -- (1.5,6);
\draw[gray] (1.4,0) -- (1.6,0) node[anchor=west] {$x$};
\draw[gray] (1.4,-2.12) -- (1.6,-2.12) node[anchor=west] {$0$};
\foreach \x in {-1,...,7}
\draw[gray] (1.4,{1/sqrt(2)*\x}) -- (1.6,{1/sqrt(2)*\x});
\draw[gray, dotted] (0,0) -- (1.4,0);
\clip (-6,-2.13) rectangle(1,6);
\begin{scope}[rotate=-45]
\draw[dashed] (-10,-10) grid(10,10);
\draw[thick] (-10,-10) grid(0,0);
\fill (0,0) circle(0.1);
\draw[ultra thick, blue!70!black] (-3,-6) -- (-3,-5) -- (-2,-5) -- (-2,-2)-- (-1,-2) -- (-1,0) -- (0,0); 
\end{scope}
\end{scope}
\end{tikzpicture}
    \caption{The random variable $\mathsf Z^{\fullspace}_t(x)$ corresponds to the sum of probabilities of all paths arriving at the point $(t,x)$, depicted on the left. One possible path is depicted in blue. The random variable $\mathsf Z_t(x)$ corresponds to the sum of probabilities of all paths confined in $\mathbb Z_{\ge 0}$ arriving at the point $(t,x)$, depicted on the right.}
    \label{fig:Z}
\end{figure}
In the context of directed polymers, this observable would be called a point-to-line partition function (see also Fig. \ref{fig:Z}), hence our use of the letter $\mathsf Z$. Doob's martingale convergence theorem applied to the martingale $L\mapsto \sum_{z=x-L}^{x+L}  \msf P^{\fullspace}_{t-L, t}(z,x)$ shows that the limit \eqref{eq:defOfullspace} exists almost surely for every $x\in \mathbb Z$. 
 Based on \eqref{eq:approxdecomposition2} and Proposition \ref{prop:localasymptoticsfull-space},  we conjecture the following. 
 \begin{conjecture} Fix $t\in \mathbb Z$. 
 For the full-space Beta RWRE,  the sequence $\left( \mathsf Z_t^{\fullspace}(x)\right)_{x\in \mathbb Z}$ is iid with distribution 
 $$\mathsf Z_t^{\fullspace}(x) \sim  \frac{1}{\alpha+\beta}\mathrm{Gamma}(\alpha+\beta).$$
 \label{conj:full-spacefunctional}
 \end{conjecture} 
The most striking fact in the statement of Conjecture \ref{conj:full-spacefunctional} is the spatial independence.  Conjecture \ref{conj:full-spacefunctional} is compatible with the recurrence relation (similar to  \eqref{eq:recurrencerelation}) satisfied by the heat kernel:  assuming Conjecture \ref{conj:full-spacefunctional}, we have  
 $$\mathsf Z^{\fullspace}_{t+1}(x) = \mathsf Z_t^{\fullspace}(x-1) W_{x-1,0}+\mathsf Z_t^{\fullspace}(x+1) (1-W_{x+1,0}) \overset{(d)}{=} \mathsf Z_t^{\fullspace}(x),$$
Indeed, if $G_1, G_2\sim\mathrm{Gamma(\alpha+\beta)}$, $B\sim \mathrm{Beta}(\alpha, \beta)$ and if $G_1,G_2,B$ are independent, then  $G_1B+G_2(1-B)\sim\mathrm{Gamma}(\alpha+\beta)$.  Moreover, if Conjecture \ref{conj:full-spacefunctional} holds at time $t$, then it may be checked that $\mathsf Z_{t+1}^{\fullspace}(x-1)$ and $\mathsf Z_{t+1}^{\fullspace}(x+1)$ are independent.

 \medskip 
 
In a half-space, the situation is more delicate as the environment is not iid and is not invariant with respect to spatial translations. More complicated  phenomena may occur, such as a trapping of the random walk at $0$ (for large $\eta$), resulting in the return probability at $0$ to be much larger than the return probability to a point $x\sqrt{t}$. Indeed, we see that \eqref{eq:returnproba0} may be very large for large $\eta$. In particular,  the expectation does not equal $1$, unlike the full-space case where $\mathbb E[\mathfrak O(W)]=1$.

 In the half-space case, we expect that  the diffusive limit of the half-space beta RWRE should be a reflected Brownian motion almost-surely, and then, the half-space analogue of the second part in \eqref{eq:approxdecomposition} is again an explicit functional.  Define the functional (see Fig. \ref{fig:Z})
 \begin{equation}
    \mathsf Z_t(x)  = \lim_{L\to\infty} \sum_{z=\max\lbrace 0,x-L\rbrace}^{x+L} \msf P_{t-L, t}(z,x).
    \label{eq:defOhalfspace}
\end{equation} 
For generic values of $\eta$, the process $L\mapsto \sum_{z=x-L}^{x+L}  \msf P_{t-L, t}(z,x)$ does not seem to be a martingale anymore. Nevertheless, we conjecture the following. 
 \begin{conjecture}
 Fix $t\in \mathbb Z$. 
  For the half-space Beta RWRE, the limit \eqref{eq:defOhalfspace} exists almost surely for every $x\in \mathbb Z_{\ge 0}$, and the variables $\left(\mathsf Z_t(x) \right)_{x}$ are independent with distribution 
 $$ \mathsf Z_t(x) \sim \begin{cases} \frac{1}{2\mu}\mathrm{Gamma}(2\mu) &\mbox{ if }x\ge 2,\\ 
 \frac{1}{2\mu}\mathrm{Gamma}(\mu+\eta) &\mbox{ if }x=1\\
 \frac{1}{2\mu}\mathrm{Gamma}(\eta) &\mbox{ if }x=0.\end{cases}$$
 \label{conj:half-spacefunctional}
 \end{conjecture}
 As for Conjecture \ref{conj:full-spacefunctional}, Conjecture \ref{conj:half-spacefunctional} is compatible with the recurrence relation \eqref{eq:recurrencerelation} satisfied by the half-space heat kernel. It is also compatible with the quenched local limit theorem from Corollary \ref{cor:returnproba} at the points  $x=0$ and $x=1$. Finally, it is also compatible with the idea that as $x$ goes to infinity, the processes $\mathsf Z_t(x)$ and $\mathsf Z^{\fullspace}_t(x)$ should be similar when $\alpha=\beta=\mu$.

\subsection{Asymptotics of the Mellin-Barnes type formula}
Now we consider the asymptotics of $\mathsf P_{0,t}(xt,1)$ for some fixed $x>0$. Going from height $xt$ to $1$ in time $t$ corresponds to a large deviation event for the random walk. Thus, we expect that this quantity decays exponentially. By analogy with the full-space situation studied in \cite{barraquand2017random, oviedo2021second, korotkikh2021hidden}, it is natural to expect that $\log \mathsf P_{0,t}(xt,1) $ has $t^{1/3}$ scale fluctuations. By KPZ universality, we further expect that the statistics should be essentially the same as for half-space last passage percolation, hence Tracy-Widom GUE distributed \cite[Theorem 1.4]{baik2018pfaffian}, with the same scaling constants as for the full-space beta RWRE. Those scaling constants were conjectured in \cite{barraquand2017random} for generic parameters, and proved in \cite{oviedo2021second}, for the asymptotics of  $\mathsf P^{\fullspace}_{0,t}(-xt,\mathbb Z_{\geq 0})$. The exact same limit theorem is expected to hold for point to point probabilities as well \cite{rassoul2014quenched, thiery2016exact}.

Proposition \ref{prop:mellinbarnes} above should, in principle, be the appropriate starting point to proving that the fluctuations of $\log \mathsf P_{0,t}(xt,1) $ are Tracy-Widom GUE distributed on the $t^{1/3}$ scale.  Indeed, a pretty straightforward -- but non-rigorous --  critical point asymptotic analysis of the formula from Proposition \ref{prop:mellinbarnes}, presented in Appendix \ref{sec:asymptotics},  yields the following. 
\begin{conjecture} \label{conj:tracywidom}
Let $\theta>\mu>0$ and $\eta>0$. Then, we have 
$$
\lim_{t \to \infty} \P \left( \frac{\log \mathsf P_{0,2t}(x_{\theta} t, 1)-a_{\theta} t}{b_{\theta} t^{1/3}}\leq y \right)=F_{\rm GUE}( y),
$$
where $F_{\rm GUE}$ is the Tracy-Widom GUE distribution function \cite{tracy1994level} (defined in  \eqref{eq:defGUE} below) and the constants $x_{\theta}, a_{\theta}, b_{\theta}$ depend on $\theta$ as 
$$x_{\theta} = \frac{2 \psi_1(\theta)-\psi_1(\theta+\mu)-\psi_1(\theta-\mu)}{\psi_1(\theta+\mu)-\psi_1(\theta-\mu)},\qquad 
a_{\theta}=-G'(\theta), \qquad  b_{\theta}=\left( \frac{G'''(\theta)}{2} \right)^{\frac{1}{3}},
$$
where $\psi_1$ is the trigamma function $\psi_1(z) = \partial_z^2\log \Gamma(z)$ and  the function $G$ is defined by 
\begin{equation}\label{eq:defG}G(z) = \log\left(\frac{\Gamma(z)^2}{\Gamma(z+\mu)\Gamma(z-\mu)}\right) +x(\theta) \log\left(\frac{\Gamma(z-\mu)}{\Gamma(z+\mu)}\right).\end{equation} 
\end{conjecture}
\begin{remark}
As $\theta$ varies from $\mu$ to $+\infty$, $x_{\theta}$ varies from $1$ to $0$. It would not make sense to consider values of $x_{\theta}$ above $1$, since, during a time interval of length $t$, a simple random walk cannot travel  a distance larger than $t$. However, we expect that the result of Conjecture \ref{conj:tracywidom} still holds when $\theta$ is scaled with $t$, as long as $b_{\theta} t^{1/3}$ goes to infinity  as $t$ goes to infinity. Estimates on polygamma functions show that this is the case as long as  $\theta\ll t^{1/4}$, that is when $x_{\theta}\gg t^{-1/4}$. This is consistent with the prediction from \cite{barraquand2020moderate} that for the full-space Beta RWRE,  the statistics of $\log \mathsf P_{0,2t}(x_{\theta} t, 1)$  are given by the KPZ equation when  $t x_{\theta}=O(t^{3/4})$, interpolating between Tracy-Widom GUE statistics when $t x_{\theta}\gg t^{3/4}$ and the Edwards-Wilkinson universality class when $t x_{\theta}\ll t^{3/4}$. 
\end{remark}
\begin{remark} 
The presence of the cross term 
$$\prod_{1 \leq i <j \leq \ell} \frac{\Gamma(z_i+w_j) \Gamma(w_i+z_j)}{\Gamma(w_i+w_j) \Gamma(z_i+z_j)}$$ 
in \eqref{eq:MellinBarnesFinal} makes it difficult to control the convergence of the series in $\ell$ uniformly as $t$ goes to infinity, since this factor may grow as $C^{\ell^2}$. This is the main reason why it is difficult to make completely rigorous the critical point asymptotics that we present in Appendix \ref{sec:asymptotics}. Similar issues arose in other papers in related contexts, see for instance \cite[Section 8]{barraquand2018half} about asymptotics of the half-space log-gamma polymer, which involves the same cross-term, or \cite{nguyen2016variants} about two-point asymptotics for the log-gamma polymer, which involves a similar cross-term. It is worth mentioning that in the context of two-point asymptotics for the stochastic six-vertex model, \cite{dimitrov2020two} succeeded in overcoming this issue, using a fine control of the problematic cross-term. Hence, we hope that it should be possible to make rigorous the asymptotics from Appendix \ref{sec:asymptotics}, perhaps using a variant of the ideas introduced in \cite{dimitrov2020two}.  
\end{remark}

\subsection{Further questions} 

A natural sequel to the present work would be to consider a continuous analogue of the half-space beta RWRE, as it was done in full-space in \cite{barraquand2020large}. One would need to rescale diffusively the random walk paths, $t=\varepsilon^{-2}\tilde t$, $x=\varepsilon^{-1}\tilde x$, and simultaneously scale $\mu=\varepsilon \tilde \mu$. This would lead to a model of sticky Brownian motions that stick to the boundary. We plan to study this model in future work. 
\medskip 

Regarding the half-space beta RWRE itself, it would be interesting to study the distribution of the random variable  $\msf P_{0,t}(0,x)$ instead of $\msf P_{0,t}(x,0)$ (see \eqref{eq:momentformula} and \eqref{eq:probato1and0}). Indeed, the knowledge of this random variable would provide some information on the distribution of the  maximum of $N$ half-space beta RWRE drawn independently in the same environment. Unfortunately, it is not clear to us yet  if reasonably simple formulas should exist for the distribution of $\msf P_{0,t}(x,y)$ with arbitrary $x,y$ (this is also unknown for other half-space solvable models). 
\medskip 

It would also be interesting to study asymptotic regimes where varying the boundary parameter $\eta$ induces a phase transition, as for other half-space models \cite{baik2001asymptotics, baik2018pfaffian, betea2018free, barraquand2021identity}. Such asymptotics could arise for the return probability at zero $\msf P_{0,t}(0,0)$, in a model where the random jump probabilities would be distributed as $W_{t,x}\sim\mathrm{Beta}(\mu_1,\mu_2)$ instead of the symmetric model that we consider where $W_{t,x}\sim\mathrm{Beta}(\mu,\mu)$. At the moment, it is not clear at all that such model would be integrable  in the half-space setting. 
\medskip 

Regarding the understanding of the integrability of the model, it would be useful to generalize the half-space beta RWRE by allowing the parameters $\mu, \eta$ to depend on $t$ and $x$. Clearly, the dependence cannot be completely arbitrary for the model to remain solvable. However, it is proved in \cite{korotkikh2021hidden} that a full-space beta RWRE model depending on three families of inhomogeneity parameters remains integrable (see also \cite{mucciconi2020spin, petrov2021parameter} for a version with only two families). Generally, integrable stochastic model depending on two families of inhomogeneity parameters (last passage percolation, log-gamma polymer) admit a half-space version depending on one family of inhomogeneity parameters. Thus, we may hope that some generalizations of the half-space beta RWRE exists, although currently, our proofs restrict our study to homogeneous parameters.
\medskip 

Finally, another direction which we hope to address in future work is to design a half-space variant of the (colored) $q$-Hahn vertex model (the model was originally defined in \cite{povolotsky2013integrability} in the uncolored case, though we refer the reader to \cite{korotkikh2021hidden} for the vertex model formulation), which should converge, in the $q\to 1$ limit, to the half-space beta RWRE. This problem is related to the discussion in Remark \ref{rem:YangBaxter}, and more generally, one may wonder about the existence of an integrable half-space variant of the stochastic higher-spin six vertex model \cite{borodin2017family, corwin2015stochastic, borodin2016higher}. 

 \subsection{Outline of rest of the paper} 
 In Section \ref{sec:betheansatz}, we prove Theorem \ref{theo:moments}. In an important part of the proof of this result, we need to show that a certain multivariate polynomial can be decomposed as sum of polynomials with certain properties. We prove this part in Section \ref{sec:boundarycondition}, by providing an effective characterization of the space of all polynomials which admits such decomposition. In Section \ref{sec:hankel}, we state and prove some properties of the Hankel transform, and prove Proposition \ref{prop:CauchytypeFredholm}, Proposition \ref{prop:mellinbarnes} and Corollary \ref{cor:returnproba}. Finally, we place in Appendix some useful, but perhaps less new, additional material. In Appendix \ref{sec:backgroundPfaffian}, we provide some background on Fredholm determinants and Pfaffians. In Appendix \ref{sec:localasymptoticsfullspace}, we prove Proposition \ref{prop:localasymptoticsfull-space} (which was already non-rigorously derived in \cite{thiery2016exact}). In Appendix \ref{sec:asymptotics}, we explain how a critical point asymptotic analysis of the formula from Proposition \ref{prop:mellinbarnes} leads to Conjecture \ref{conj:tracywidom}. 
 
 \subsection*{Acknowledgments} 
G.B. thanks Yu Gu and Alex Dunlap for a useful discussion about local central limit theorems, and in particular for drawing our attention to the reference \cite{boldrighini1999central}. 

This article is based upon work supported by the National Science Foundation under Grant No. DMS-1928930 while G.B.  participated in a program hosted by the Mathematical Sciences Research Institute in Berkeley, California, during
the Fall 2021 semester. M.R. was partially supported by the Fernholz Foundation and the NSF Grant No. DMS-1664650.

Data sharing not applicable to this article as no datasets were generated or analysed during the current study.
\section{Bethe ansatz} 
\label{sec:betheansatz}

In this section we examine the mixed moments of the quenched transition probabilities for the half-space beta RWRE. We first write down a difference equation satisfied by mixed moments (Equation \eqref{eq:recurrenceu}). Then, we use an argument inspired by the coordinate Bethe ansatz (as in \cite{povolotsky2013integrability, corwin2015q, barraquand2017random}) to rewrite this difference equation as a simpler difference equation subject to boundary conditions in Proposition \ref{prop:evolutionequation}.  We finally show that the unique solution to this difference equation and boundary conditions is given by the integral formula for the mixed moments appearing in Theorem \ref{theo:moments}, thus proving the theorem. 
\subsection{Evolution equation}
To write down difference equations satisfied by the mixed moments, we need to set up some notation. 
\begin{definition} Let $\Zhalf$ be the set of points of the form $(t,x)\in \Z \times \Z_{\geq 0}$ where $t+x$ is odd. 
Let us define operators $\tau^{(i)}_{\pm}$ acting on a function $f: \left(\Zhalf\right)^k \to \mathbb C$ by
\[\tau^{(i)}_{\pm} f[(t_1,x_1),\dots ,(t_k,x_k)]=f[(t_1,x_1),\dots , (t_i-1, x_i\pm 1),\dots,(t_k,x_k)].\]
When all coordinates $t_i$ are the same, we will use the simpler notation 
$$ f_t(x_1, \dots, x_k)=f[(t,x_1),\dots ,(t,x_k)] $$
\end{definition}
The goal of this section is to characterize the function  $u : \left(\Zhalf\right)^k \to \mathbb [0,1]$ defined by 
\[ u[(t_1,x_1),\dots ,(t_k,x_k)]:=\E[\msf{P}_{-t_1,0}(x_1,1)\dots \msf{P}_{-t_k,0}(x_k,1)].\]
We will encode the time evolution of $u_{t}(x_1,\dots ,x_k)$ by an operator written in terms of the $\tau^{(i)}_{\pm}$. In order to write this explicitly, let us first consider the case when $x_1=\dots =x_k$. 
Let $W \sim \mathrm{Beta}(\mu,\mu)$ and $\bar{W}=\mathrm{Beta}(\mu,\eta)$. When $x_1=\dots =x_k=y \geq 2$, the recurrence relation \eqref{eq:recurrencerelation} implies that
\begin{align} \label{eq:clusterat2} u_t(\vec{x})&= \sum_{j=0}^k \binom{k}{j} \E[ (1-W)^{j}W^{k-j}  (\msf P_{-t+1}(y-1,1)^{j} \msf P_{-t+1,0}(y+1,1)^{k-j}] \nonumber\\
&=\sum_{j=0}^k \binom{k}{j} \frac{(\mu)_j (\mu)_{k-j}}{(2 \mu)_{k}} u_{t-1}(\underbrace{y-1,\dots ,y-1}_{j},\underbrace{y+1,\dots ,y+1}_{k-j}) \nonumber\\
&=\sum_{j=0}^k \binom{k}{j} \frac{(\mu)_j (\mu)_{k-j}}{(2 \mu)_{k}} \tau^{(1)}_-\dots \tau_-^{(j)} \tau_+^{(j+1)}\dots \tau_+^{(k)} u_{t}(\vec{x}).
\end{align}
When $x_1=\dots x_k=1$, again,  the recurrence relation \eqref{eq:recurrencerelation} implies that
\begin{align} \label{eq:clusterat1} u_t(\vec{x})&= \sum_{j=0}^k \binom{k}{j} \E[  (1-\bar{W})^{j} \bar{W}^{k-j} \msf P_{-t+1}(0,1)^{j}\msf P_{-t+1,0}(2,1)^{k-j}] \nonumber\\
&=\sum_{j=0}^k \binom{k}{j} \frac{(\eta)_{j} (\mu)_{k-j} }{(\mu+\eta)_{k}} \E[ \msf P_{-t+1,0}(0,1)^{j}\msf P_{-t+1,0}(2,1)^{k-j}]\nonumber\\
&=\sum_{j=0}^k \binom{k}{j} \frac{(\eta)_j (\mu)_{k-j}}{(\mu+\eta)_{k}} \E[ \msf P_{-t+2,1}(1,1)^{j}\msf P_{-t+1,0}(2,1)^{k-j}], \nonumber\\
&=\sum_{j=0}^k \binom{k}{j} \frac{(\eta)_j (\mu)_{k-j}}{(\mu+\eta)_{k}} \tau^{(1)}_+\dots \tau^{(k)}_+ \tau^{(1)}_-\dots \tau^{(j)}_- u_{t}(\vec{x}).
\end{align}
The third equality above uses the fact that a random walker at zero must jump to one at the next time step, i.e. $\msf P_{-t+1,0}(0,1)=\msf P_{-t+2,0}(1,1)$. Finally, when $x_1=\dots =x_k=0$, we have 
\begin{align} \label{eq:clusterat0} u_t(0,\dots ,0)&= \E[\msf P_{-t+1,0}(1,1)^k]=u_{t-1}(1,\dots ,1)=\tau_+^{(1)}\dots \tau_+^{(k)} u_{t}(\vec{x}).
\end{align}

\medskip 
Now, we turn to the case where the $x_i$ are not necessarily equal. 
Divide the $x_i$ into clusters of equal coordinates as follows
\begin{align*}
x_1=\dots =x_{b_0}&= 0\\
x_{b_0+1}=\dots = x_{b_0+b_1}&=1\\
x_{b_0+b_1+1}=\dots =x_{b_0+b_1+c_1}&=y_1>1,\\  
 \dots &\\
 x_{b_0+b_1+c_1+\dots+c_{\ell-1}+1}=\dots =x_{b_0+b_1+c_1+\dots+c_{\ell}}&=y_{\ell}>y_{\ell-1}.\end{align*}
 Here $b_0$ is the number of coordinates at position $0$, $b_1$ is the number of coordinates at position $1$, and the $c_i$ are the size of each clusters of equal coordinates, except the clusters at $0$ and $1$. When considering the evolution of $u_t(\vec x)$ from time $t-1$ to $t$, the variables $W_{-t+1,x}$ involved are independent for different values of $x$. As a result, the recurrence relation for $u_t(\vec x)$ factorizes as a product over clusters of equal elements in the vector $\vec x$, 
\begin{multline} \label{eq:recurrenceu}
    u_{t}(\vec{x})= 
    \prod_{j=1}^{\ell} \left(\sum_{i=0}^{c_j}\binom{c_j}{i} \frac{(\mu)_i (\mu)_{c_j-i}}{(2 \mu)_{c_j}} \prod_{s=1}^{i} \tau^{(b_0+b_1 +c_1+\dots +c_j+s)}_{-} \prod_{r=i+1}^{c_j} \tau^{(b_0+b_1 +c_1+\dots +c_j+r)}_{+}\right) \\
    \times  \prod_{s=1}^{b_1} \tau_+^{(b_0+s)} \sum_{i=1}^{b_1} \binom{b_1}{i} \frac{(\mu)_i(\eta)_{b_1-i}}{(\mu+\eta)_{b_1}} \prod_{r=1}^{b_1-i} \tau_-^{(b_0+r)} \\
     \times (\tau^{(1)}_+\dots \tau^{(b_0)}_+)u_{t}(\vec{x}).
\end{multline} 
The first line corresponds to clusters of coordinates greater than $1$, the second corresponds to the cluster of coordinates equal to $1$ and the third corresponds to the cluster of coordinates equal to $0$. 
$u_t(\vec{x})$ satisfies the  initial condition
\begin{equation} \label{eq:momentinitialcondition}
u_0(\vec{x})=\prod_{i=1}^k\mathds{1}_{x_i=1}. \end{equation}

The function $u_t(\vec x)$ is, by definition, symmetric in the variables $x_i$. Hence, it is enough to determine it when $0 \leq x_1 \le \dots\le x_k$. The recurrence relations satisfied by $u_t$ are quite complicated to solve. However, we will see that we may find a function $\phi : \left( \Zhalf \right)^k \to \C $ such that  $u_t(\vec x)=\phi_t(\vec x)$ when $0 \leq x_1 \le \dots\le x_k$ and $\phi$ satisfies simpler recurrence relations which are well-suited to be solved via the coordinate Bethe ansatz. Note that $u_t$ does not solve those simpler recurrence relations. 
\begin{proposition} \label{prop:evolutionequation}
Assume that a function $\phi: \left( \Zhalf \right)^k \to \C$ satisfies the following properties: 
	\begin{enumerate}
		\item For any $1\le i\le k$, on the restriction of $\left(\Zhalf\right)^k$ to points $(t_1,x_1), \dots, (t_k, x_k)$ such that  $x_i \geq 2$, we have 
		\[\phi=\left(\frac{\tau_+^{(i)}+\tau_-^{(i)}}{2}\right) \phi.\]
		\item For any $1\le i\le k-1$, on the restriction of $\left(\Zhalf\right)^k$ to points $(t_1,x_1), \dots, (t_k, x_k)$ such that  $x_i=x_{i+1} \geq 2$ and $t_i=t_{i+1}$, we have 
		\[\frac{1}{2 \mu+1} \left( \tau^{(i)}_+ \tau^{(i+1)}_+  - (2 \mu+1) \tau^{(i)}_+ \tau^{(i+1)}_-+\tau^{(i)}_-\tau^{(i+1)}_-  
		+(2\mu -1)\tau^{(i)}_- \tau^{(i+1)}_{+} \right) \phi=0\]
		\item For any $1\le i\le k$, 
		\[(1-\tau^{(i)}_{+}) \phi[(t_1,x_1),\dots ,(t_k,x_k)]|_{x_i=0}=0\].
		\item For any odd $1\le x_1\le \dots \le x_k$ , we have the initial condition
		\[\phi_{0}(\vec{x})=\prod_{i=1}^k \mathds{1}_{x_i=1}=\mathds{1}_{x_k=1}\]
		\item For any $1\le r\le k$ and $t$ even, on the restriction of $\left(\Zhalf\right)^k$ to points $(t_1,x_1), \dots, (t_k, x_k)$ such that $(t_1,x_1)=\dots =(t_{r},x_{r})=(t,1)$, we have 
		\begin{equation}
		\left( 1-\left(\prod_{j=1}^r \tau^{(i)}_+\right)\sum_{i=0}^r \binom{r}{i} \frac{(\mu)_{i} (\eta)_{r-i}}{(\eta +\mu)_r} \prod_{j=1}^{r-i}\tau^{(j)}_- \right) \phi=0.
		\label{eq:boundary}
		\end{equation}
	\end{enumerate}
Then for any $0 \leq x_1 \le \dots\le x_k $ and $t\ge 0$ such that for $t+x_i$ is odd for all $1\le i\le k$, we have
\[\phi_t(\vec{x})=u_t(\vec{x}).\]
\end{proposition}

The proof of Proposition \ref{prop:evolutionequation} uses a binomial expansion formula for noncommutative variables, whose use in the Bethe ansatz context was pioneered in  \cite{povolotsky2013integrability} (see also \cite{rosengren2000non} for another proof of this formula). The main novelty is the property (5) which is specific to the half-space setting. 
\begin{proposition}[{\cite{povolotsky2013integrability}}] \label{th:noncommutativebinomial}
When $A$ and $B$ generate an associative algebra, such that
\[BA=\frac{1}{2 \mu+1} (AA+(2 \mu-1) AB +BB),\]
we have the noncommutative binomial formula

\[\left(\frac{A+B}{2}\right)^k =\sum_{i=0}^{k} \binom{k}{i} \frac{(\mu)_i (\mu)_{k-i}}{(2 \mu)_k} A^j B^{k-j}\]
\end{proposition}

\begin{proof}[Proof of Proposition \ref{prop:evolutionequation}]
To prove this proposition, we will assume $\phi_{t}(\vec{x})$ satisfies $(1)-(5)$ and show that it is the unique solution to \eqref{eq:recurrenceu} with the correct initial condition. 

Comparing property $(4)$ to \eqref{eq:momentinitialcondition} shows that the initial condition is correct. Now we consider the recursion. The solution to \eqref{eq:recurrenceu} is unique on $t>0$ given an initial condition $u_0(\vec{x})$. Indeed, equation \eqref{eq:recurrenceu} expresses $u_t(\vec{x})$ as a finite weighted sum of terms of the form
\[u[\underbrace{(1,t-2),\dots ,(1,t-2)}_{b_0 \text{ terms}}, (y_1, t-1),\dots ,(y_i, t-1)]=u_{t-1}(y_1,\dots ,y_i,0,\dots ,0),\]
so the uniqueness follows by induction on $t$. Since the time evolution of $u_t(\vec{x})$ is uniquely determined by \eqref{eq:recurrenceu}, we need to show that this same recurrence for $\phi_{t}(\vec{x})$ is implied by properties $(1)-(5)$. In particular the fact that \eqref{eq:clusterat0} is satisfied by $\phi_{t}(\vec{x})$ is exactly property $(3)$, and the fact that \eqref{eq:clusterat1} is satisfied is exactly property $(5)$. 
The fact that \eqref{eq:clusterat2} is satisfied will follow from $(1)$ and $(2)$ together. To see this use Proposition \ref{th:noncommutativebinomial} to see that when $x_1=\dots =x_c=y$, we have
\[\phi_t(\vec{x})=\prod_{j=1}^c \left(\frac{\tau^{(j)}_++ \tau^{(j)}_-}{2}\right) \phi_t(\vec{x}) =\sum_{j=0}^c \binom{c}{j} \frac{(\mu)_j (\mu)_{c-j}}{(2 \mu)_{c}} \prod_{r=1}^{j} \tau^{(r)}_{-} \prod_{s=j+1}^c \tau^{(s)}_{+}   \phi_t(\vec{x}).
\]
Since all three of these equations are satisfied, the full recurrence
\eqref{eq:recurrenceu} holds. 
This concludes the poof of Proposition \ref{prop:evolutionequation}. 
\end{proof}

 \subsection{Proof of Theorem \ref{theo:moments}}
 \label{sec:prooftheorem}
 In this section we will prove Theorem \ref{theo:moments}.  We consider the function defined by 
\begin{multline} \label{eq:integralmoment} \phi[(t_1,x_1), \dots, (t_k,x_k)] = \prod_{i=1}^k (-2)^k (\mu+\eta)_k \mathds{1}_{x_i+t_i \text{ is odd}}  \\
\times  \int_{\I \R} \frac{dz_1}{2 \pi \I}\dots \int_{\I \R} \frac{dz_k}{2 \pi \I} \prod_{1 \leq a< b \leq k} \frac{z_a-z_b}{z_a-z_b-1} \frac{z_a+z_b}{z_a+z_b+1} \prod_{i=1}^k \left( \frac{z_i^2}{z_i^2-\mu^2} \right)^{t_i/2+1} \left( \frac{z_i-\mu}{z_i+\mu} \right)^{(x_i-1)/2} \frac{1}{z_i(z_i+\eta)},
\end{multline}
when $x_i \geq 1$, and extended to $x_i = 0$ by imposing the condition 
\[(1-\tau_+^{(i)})\phi\Big|_{x_i=0}=0.\]
We will show that this function satisfies the conditions $(1)-(5)$ of Proposition \ref{prop:evolutionequation}. The main difficulty will be showing property $(5)$ is satisfied. This will require us to consider the polynomial 
 \begin{equation}
 P_k(z_1,\dots,z_k)=\prod_{i=1}^k z_i^2-\left(\prod_{i=1}^k (z_i-\mu) \right) \sum_{i=0}^k \binom{k}{i} \frac{(\mu)_i (\eta)_{k-i}}{(\eta+\mu)_k} \prod_{\ell=1}^{k-i} (z_{\ell} +\mu)\prod_{j=k-i+1}^k z_j  ,
 \label{eq:defPk}
 \end{equation}
 and show that it can be decomposed as follows. 
 
 \begin{proposition} \label{prop:decomposePk}
 The polynomial $P_k$ defined in \eqref{eq:defPk} admits a decomposition of the form
 \begin{equation} \label{eq:decomposePk} P_k(z_1,\dots ,z_k)=(z_1+\eta) F_1(z_1,\dots ,z_k)+\sum_{i=2}^{k} (z_{i-1}-z_i-1) F_{i}(z_1,\dots ,z_k), \end{equation}
 
 where \begin{itemize}[leftmargin=0.4cm]
	\item $F_1\in \C[z_1,\dots,z_k]$ is symmetric with respect to $z_1 \leftrightarrow -z_1$ (we say that a function $f(z_1,...,z_k)$ is symmetric with respect to $z_1 \xleftrightarrow[]{} -z_1$ if $f(z_1, z_2,...,z_k)=f(-z_1,z_2...,z_k)$).
	\item $F_i \in \C[z_1,\dots,z_k]$ is symmetric with respect to $z_i \leftrightarrow z_{i-1}$ (we say that a function $f(z_1,...,z_k)$ is symmetric  with respect to $z_i \xleftrightarrow[]{} z_{i-1}$ if $f(z_1,...,z_{i-1}, z_i,...,z_k)=f(z_1,...,z_i,z_{i-1},...,z_k)$).
\end{itemize}
 \end{proposition}
 The main technical challenge in the proof of Theorem \ref{theo:moments} is actually the proof of Proposition \ref{prop:decomposePk}, which we delay to  Section \ref{sec:boundarycondition}.  Let us assume for now that Proposition \ref{prop:decomposePk} is true, and proceed with the proof of Theorem \ref{theo:moments}.
 
\medskip 
It suffices to show that the function $\phi$ defined by \eqref{eq:integralmoment} satisfies properties $(1)-(5)$ from Proposition \ref{prop:evolutionequation}. This will imply that $\phi_{t}(\vec{x})=u_t(\vec{x})$ when $0 \le x_1\le \dots\le x_k $. 

\textbf{Property (1):} To check that $\phi$ satisfies property $(1)$ for $x_i \geq 2$ observe that applying $\frac{ \tau_+^{(i)}+\tau_-^{(i)}}{2}$ to our integral formula for $\phi$ brings a factor of 
\[\frac{1}{2}\left( \frac{z_i-\mu}{z_i}+\frac{z_i+\mu}{z_i} \right)=1\]
inside the integral. So applying $\frac{ \tau_+^{(i)}+\tau_-^{(i)}}{2}$ does not change $\phi$ provided $x_i \geq 2$. 

\textbf{Property (2):} To check that $\phi$ satisfies property $(2)$, observe that applying the operator
\[\frac{1}{2 \mu +1} \left( \tau^{(i)}_+ \tau^{(i+1)}_+ - (2 \mu+1)\tau^{(i)}_+ \tau^{(i+1)}_-+\tau^{(i)}_-\tau^{(i+1)}_- 
+(2 \mu -1)\tau^{(i)}_- \tau^{(i+1)}_{+} \right)\]
to $\phi$
brings a factor into the integrand which simplifies to
\[-\frac{ 4 \mu^2 (z_i-z_{i+1}-1)}{z_i z_{i+1} (2 \mu+1)}.\]
Notice that the factor $(z_i-z_{i+1}-1)$ cancels a similar factor in the denominator of the integrand. After this cancellation the integrand is antisymmetric with respect to $z_i \leftrightarrow z_{i+1}$ (because $x_i=x_{i+1}$), and the contour for these two variables is identical, so the integral is zero as desired. 

\textbf{Property (3):} We defined the extension of $\phi$ to $x_i=0$ in terms of the value of $\phi$ at $x_i=1$, so $\phi$ immediately satisfies property $(3)$. 

\textbf{Property (4):} To show property $(4)$ we begin by examining \eqref{eq:integralmoment} when $t_i=0$ for all $i$. Note first that due to the indicator function in front, we need only consider the case where all the $x_i$ are odd. If $x_k \geq 3$, then the integrand has no poles at $z_k=\mu$, and in fact the integrand has no poles in $z_i$ with nonnegative real part. Because the integrand also has quadratic decay near $\infty$ integrating over $z_1$ gives $0$. 

Thus, we only  need to compute the integral when $x_k=1$, and we will do so now. Due to the ordering on $x_i$'s and the indicator function in the definition of $\phi$, the only nontrivial case is when all $x_i$ are equal to $1$. It will be helpful to define

\[q(z_a,z_b):= \frac{z_a-z_b}{z_a-z_b-1} \frac{z_a+z_b}{z_a+z_b+1},\]
and
\[ g_i(z):=\frac{z}{(z-\mu-i)(z+\mu+i)(z+\eta)},\]
This gives the following expression (the first equality is a definition, the second is an evaluation)
\[I_k:=\frac{\phi_{0}(1,\dots ,1)}{(-2)^k (\mu+\eta)_k}= \oint_{\I \R} \frac{dz_1}{2 \pi \I}\dots \oint_{\I \R}\frac{dz_k}{2 \pi \I} \prod_{1 \leq a<b \leq k} q(z_a,z_b) \prod_{i=1}^k g_0(z_i).\]

We will proceed by induction. In the $k=1$ case we evaluate $I_1= \oint_{\I \mathbb{R}} \frac{dz_1}{2 \pi \I} g_0(z_1)$ by deforming the contour for $z_1$ to a small negatively oriented circle around $\mu$ and taking the residue at $z_1=\mu$. This gives
\[I_1=\frac{-1}{2(\mu+\eta)},\]
as desired.

We state the $k=2$ case explicitly as it contains the same idea as the general $k$ case. For $k=2$ we have
$$I_2=\oint_{\I \mathbb{R}} \frac{dz_1}{2 \pi \I} \oint_{\I \mathbb{R}} \frac{dz_2}{2 \pi \I} q(z_1,z_2) g_0(z_1) g_0(z_2).$$
Observe that deforming the contour for $z_2$ to a small negatively oriented circle around $\mu$ does not cross any pole of $q(z_1,z_2)$, so we can integrate out the $z_2$ variable by taking the residue at $z_2=\mu$. We obtain
\[I_2=I_1 \oint_{\I \R} \frac{d z_1}{2 \pi \I} q(z_1,\mu) g_0(z_1).\]
Simple cancellation gives
\[q(z_1,\mu) g_0(z_1)=g_1(z_1),\]
so
\[I_2=I_1 \oint_{\I \R} \frac{dz_1}{2 \pi \I} g_1(z_1).\]
We deform the $z_1$ contour to a negatively oriented circle around $\mu+1$ and evaluate the residue to obtain
\[I_2=I_1 \frac{-1}{2(\eta+\mu+1)}=\frac{1}{4(\eta+\mu)_2}\]
as desired. 

In the $k=2$ case we integrated by taking the residue $z_2=\mu$, then taking the residue $z_1=\mu+1$. For the induction hypothesis, we assume that in the in $k-1$ case, the equality of the lemma holds and that the integral can be evaluated by sequentially deforming each variable $z_{k-i}$ to a negatively oriented circle around $\mu+i$ and taking the residue starting with $z_k$. With this assumption we will prove the same for $k$. Begin with
\[I_k=\oint_{\I \mathbb{R}} \frac{dz_1}{2 \pi \I} ... \oint_{\I \R} \frac{dz_k}{2 \pi \I} \prod_{1 \leq a<b \leq k} q(z_a, z_b) \prod_{i=1}^k g_0(z_i).\]
Observe that no poles of $q(z_1, z_{k-i})$ are crossed when we deform the contour for $z_{k-i}$ to a negatively oriented circle around $\mu+i$. Thus by assumption we can integrate out the variables $z_k, z_{k-1},\dots  z_3, z_2$ in that order by deforming each variable $z_{k-i}$ to a negatively oriented circle around $\mu+i$ and taking the residue. By assumption if we ignore all terms containing $z_1$ this gives $I_{k-1}$. Thus
\begin{equation} \label{eq:683} I_k=I_{k-1} \oint_{\I \R} \frac{dz_1}{2 \pi \I} \prod_{i=0}^{k-2} q(z_1, \mu+i) g_0(z_1).
\end{equation}
Simple cancellation gives
\[q(z,\mu+i) g_{i}(z)=g_{i+1}(z),\]
 so \eqref{eq:683} becomes
$$I_k=I_{k-1}\oint_{\I \R} \frac{dz_1}{2 \pi \I} g_{k-1}(z_1).$$
The final integral is evaluated by deforming the contour for $z_k$ to a small negatively oriented circle around $\mu+k-1$ and taking the residue to obtain

\[I_k=\frac{-1}{2 (\eta+\mu+k-1)} I_{k-1}=\left( \frac{-1}{2} \right)^k \frac{1}{(\eta+\mu)_k},\]
which proves property $(4)$. 

\textbf{Property (5):} We can apply the operator 
	\[\left( 1-\left(\prod_{j=1}^r \tau^{(i)}_+\right)\sum_{i=0}^r \binom{r}{i} \frac{(\mu)_{i} (\eta)_{r-i}}{(\eta +\mu)_r} \prod_{j=1}^{r-i}\tau^{(j)}_- \right)
		\]
to our integral formula for $\phi_t(\vec{x})$ when $x_1=\dots =x_r=1$. We obtain
\begin{multline} \label{eq:polyintegratetozero} (-2)^k (\mu+\eta)_k \prod_{i=1}^k \mathds{1}_{x_i+t \text{ odd}}  \\
\times 
 \int_{\I \R} \frac{dz_1}{2 \pi \I}\dots \int_{\I \R} \frac{dz_k}{2 \pi \I} \prod_{1 \leq a< b \leq k} \frac{z_a-z_b}{z_a-z_b-1} \frac{z_a+z_b}{z_a+z_b+1} \prod_{i=r+1}^{k} \left( \frac{z_i^2}{z_i^2-\mu^2} \right)^{t/2+1} \left( \frac{z_i-\mu}{z_i+\mu} \right)^{(x_i-1)/2} \frac{1}{z_i(z_i+\eta)} \\
 \times 
 \prod_{i=1}^{r} \left(\frac{z_i^2}{z_i^2-\mu^2} \right)^{t/2}  \frac{1}{z_i(z_i^2-\mu^2)(z_i+\eta)} P_r(z_1,...,z_r).
\end{multline}
where $P_r$ was defined in \eqref{eq:defPk}. To show that  property $(5)$ is satisfied, it suffices to show that \eqref{eq:polyintegratetozero} is equal to $0$. This is where  Proposition \ref{prop:decomposePk} is crucial. 
We will show  that the decomposition in Proposition \ref{prop:decomposePk} allows us to write \eqref{eq:polyintegratetozero} as a sum of $r$ terms, each of which is equal to zero. 

First consider replacing $P_r(z_1,\dots ,z_r)$ by $(z_1+\eta) F_1(z_1,\dots ,z_r)$ in \eqref{eq:polyintegratetozero}, where $F_1(z_1,\dots ,z_r)$ is symmetric with respect to $z_1 \leftrightarrow -z_1$. Then, in \eqref{eq:polyintegratetozero}, the integrand (with $P_r$ replaced by $(z_1+\eta) F_1$) is antisymmetric with respect to $z_1 \leftrightarrow -z_1$, and the contour of integration is unchanged by this transformation. Thus, the integral over the variable $z_1$ equals $0$. 

Now consider replacing $P_r(z_1,\dots ,z_r)$ by $(z_{i-1}-z_i-1)F_i(z_1,\dots ,z_r),$
in \eqref{eq:polyintegratetozero}, where $F_i(z_1,\dots ,z_r)$ is symmetric with respect to $z_i \leftrightarrow z_{i-1}$. Then, in \eqref{eq:polyintegratetozero}, this replacement cancels a factor of the form $\frac{1}{z_{i-1}-z_{i}-1}$, and the resulting integrand is antisymmetric in $z_{i} \leftrightarrow z_{i-1}$. Because the contour is the same for both variables, the integral over variables $z_{i}$ and $z_{i-1}$ equals  $0$. 

These observations, along with the decomposition from Proposition \ref{prop:decomposePk} imply that \eqref{eq:polyintegratetozero} is equal to zero. This proves that $\phi_t(\vec{x})$ satisfies property $(5)$. 

\begin{remark}
The proof of Theorem \ref{theo:moments} shows that a sufficient condition for the boundary condition \eqref{eq:boundary} to be satisfied is that the polynomial \eqref{eq:defPk} satisfies the decomposition \eqref{eq:decomposePk}. A similar property is also satisfied by the moment formula for the half-space log-gamma polymer \cite[Eq. (4)]{barraquand2018half}, and this is the guiding principle that led us to the definition of the half-space beta RWRE (Definition \ref{def:halfspacebetaRWRE}). Initially, we considered the integral formula \eqref{eq:momentformula}, inspired by moment formulas satisfied by several other models (the half-space log-gamma polymer \cite{barraquand2018half}, the half-line KPZ equation \cite{borodin2016directed}, and also by a Brownian limit of the full-space beta RWRE, studied in \cite{barraquand2020large}, which also admits a half-space version, to be studied in future work). Then, to determine to which discrete model the formula \eqref{eq:momentformula} could correspond to, we first postulated that, away from the boundary, the model should be similar as in full-space, that is $W_{t,x}\sim \mathrm{Beta}(\mu,\mu)$ for $x>1$. It remained to understand the distribution of boundary weights $W_{t,1}$. We wrote an analogue of the boundary condition \eqref{eq:boundary}, involving the moments of the random variables $W_{t,1}$. 
As in \eqref{eq:polyintegratetozero}, it can be rewritten in terms of a polynomial $P_k$ similar as \eqref{eq:defPk}, involving the values of $\mathbb E[(W_{t,1})^k]$, and we expected that this polynomial satisfies the decomposition \eqref{eq:decomposePk}, by analogy with the half-space log-gamma polymer moment formulas. As a consequence,  this polynomial $P_k$ must satisfy $P_k(-\eta, -\eta-1, \dots, -\eta-k+1)=0$,  which yields some recurrence relation for the moments $\mathbb E[W_{t,1}^k]$. The recurrence is readily solved, leading to $W_{t,1}\sim \mathrm{Beta}(\mu, \eta)$ as in Definition \ref{def:halfspacebetaRWRE}. 
\label{rem:origin}
\end{remark}

\section{Polynomial decomposition}  \label{sec:boundarycondition}

The goal of this section is to prove Proposition \ref{prop:decomposePk}, which we do by providing a general characterization of polynomials admitting a decomposition similar to the one arising in Proposition \ref{prop:decomposePk}, as defined in the following definition. 
\begin{definition} Let $\C[z_1,\dots,z_k]_{\leq 2,\dots,2}$ be the subspace of polynomials in the variables $z_1,\dots,z_k$ over $\C$ spanned by monomials $z_1^{i_1}\dots z_k^{i_k}$ where $i_j \leq 2$ for all $j$. We say that a polynomial $p\in \C[z_1,\dots,z_k]_{\leq 2,\dots,2}$ admits a boundary  decomposition if we can write 

 \begin{equation} \label{eq:boundarydecomposition} p(z_1,\dots ,z_k)=(z_1+\eta) F_1(z_1,\dots ,z_k)+\sum_{i=2}^{k} (z_{i-1}-z_i-1) F_{i}(z_1,\dots ,z_k), \end{equation}
 
 where \begin{itemize}[leftmargin=0.4cm]
	\item $F_1\in \C[z_1,\dots,z_k]$ is symmetric with respect to $z_1 \leftrightarrow -z_1$ with $ (z_1+\eta) F_1(\vec{z}) \in \C[z_1,\dots,z_k]_{\leq 2,\dots,2},$
	\item $F_i \in \C[z_1,\dots,z_k]$ is symmetric with respect to $z_i \leftrightarrow z_{i-1}$ with $(z_{i-1}-z_i-1) F_i(\vec{z}) \in \C[z_1,\dots,z_k]_{\leq 2,\dots,2}$.
\end{itemize}

  The set of polynomials admitting a boundary decomposition is a subspace of $\C[z_1,\dots,z_k]_{\leq 2,\dots ,2}$  which we will denote by $\BD_k$.
\label{def:decomposition}
\end{definition}

It is clear that $P_k(\vec{z}) \in \mathbb{C}[z_1,\dots,z_k]_{\leq 2,\dots ,2}$ and our goal is to show that it lies in $\BD_k$. For a given polynomial $p\in \BD_k$, the decomposition \eqref{eq:boundarydecomposition} is, in general, far from unique. Indeed, the dimension of  $\BD_k$ is clearly less than $3^k$, while the space of all $k$-tuples of polynomials $F_i$ satisfying the conditions of Definition \ref{def:decomposition} has  dimension at least $k3^{k-2}$. Partly for this reason, writing an explicit decomposition of $P_k$ for all $k$ ends up being quite difficult. Instead we will find a characterization of $\BD_k$ that allows us to check whether $P_k$ belongs to $\BD_k$.

To build some intuition we can observe that any polynomial $p \in \BD_k$ satisfies 
\begin{equation}
    p(-\eta, -\eta-1,\dots, -\eta-k+1)=0.
    \label{eq:pevaluatedinstrings}
\end{equation}
In addition for $k=1$ a polynomial $p(z_1) \in \mathbb{C}[z_1]_{\leq 2}$ is in $\BD_1$ if and only if $p(-\eta)=0$ and the coefficient of $z_1^2$ in $p(z_1)$ is $0$.  To describe the characterization of $\BD_k$ for general $k$ we need to introduce some notation. 

\begin{definition}
We call a $(r, k-r)$ shuffle any permutation $\sigma \in \mathcal S_k$ which satisfies $\sigma(1)<\dots<\sigma(r)$, and $\sigma(r+1)<\dots<\sigma(k)$. We denote the set of $(r, k-r)$ shuffles by  $\mathrm{Sh}_{r, k-r}$. Define the linear map $ L_{r,k}$ by 
	\begin{align*}
	L_{r,k} \;\;: \;\; \C[z_1,\dots,z_k]_{\leq 2,\dots,2} &\rightarrow \;\;\;\C[x_1, \dots, x_r]_{\leq 2,\dots,2}\\
	 p\;\;\;\; & \mapsto \sum_{\sigma \in \mathrm{Sh}_{r, k-r}} p(\sigma(-\eta, -\eta-1, \dots, -\eta-k+r+1, x_1,\dots,x_{r})
	\end{align*}
Let $\mathbb{SH}_k$ be the space of polynomials in $\C[z_1,\dots,z_k]_{\leq 2,\dots,2}$ satisfying the condition that the coefficient of $x_1^2\dots x_r^2$ in $L_{r,k}(p)$, denoted $[x_1^2  \dots x_r^2] L_{r,k}(p)$ below,  equals $0$ for all $0\leqslant r\leqslant k$.
\label{def:shuffles}
\end{definition}

Note that when $r=0$, the condition $[x_1^2  \dots x_r^2] L_{r,k}(p)=0$  becomes \eqref{eq:pevaluatedinstrings}, and when $r=k$, the condition says that the coefficient of $z_1^2\dots z_k^2$ in $p(z_1,\dots,z_k)$ is $0$. Both these properties of $\mathbb{SH}_k$ are also satisfied by $\mathbb{BD}_k$ as a direct consequence of  Definition \ref{def:decomposition}.

\begin{theorem} For all $k\geq 1$, we have 
$$\mathbb{BD}_k=\mathbb{SH}_k,$$
where $\BD_k$ is defined in Definition \ref{def:decomposition} and $\mathbb{SH}_k$ is defined in Definition \ref{def:shuffles}.  
\label{thm:characterizationdecomposition}
\end{theorem}

Before we delve into the proof of Theorem \ref{thm:characterizationdecomposition}, we show in Section \ref{sec:coefficientsPk} that Theorem \ref{thm:characterizationdecomposition} implies Proposition \ref{prop:decomposePk}. Then we prove Theorem \ref{thm:characterizationdecomposition} in three steps. In Section \ref{sec:easy}, we show that $\BD_k \subset \mathbb{SH}_k$, so all that remains is to show $ \dim \BD_k \geq  \dim \mathbb{SH}_k$. In Section \ref{sec:linearalgebra}, we reduce the problem to checking that a certain list of monomials can be decomposed as a polynomial in $\BD_k$ plus a polynomial in a certain subspace. In Section \ref{sec:B} and  Section \ref{sec:C} we show that these monomials can indeed be decomposed in the desired way.

\subsection{Proof of Proposition \ref{prop:decomposePk}}
\label{sec:coefficientsPk}
We prove this proposition in two steps. First we compute the coefficient of $[x_1^2\dots x_r^2] L_{r,k} (P_k)$, then we use combinatorial identities to show that it is equal to $0$ for all $0 \leq r \leq k$. By Theorem \ref{thm:characterizationdecomposition} these two steps will be enough to prove the proposition. 

\textbf{Step 1:} In this step we show that  $[x_1^2\dots x_r^2] L_{r,k}(P_k)$ is equal to $C_{r,k}$ with
$$C_{r,k}=\binom{k}{r} ( \eta)_{k-r}^2 - (\eta+\mu)_{k-r} \sum_{i=0}^k \binom{k}{i} \frac{(\mu)_i (\eta)_{k-i}}{(\eta+\mu)_k} \sum_{\ell=0}^{k-r} \binom{i}{\ell} \binom{k-i}{k-r-\ell} (\eta+k-r-\ell)_{\ell} (\eta-\mu)_{k-r-\ell}.$$
Most of the evaluation of the $x_1^2\dots x_r^2$ coefficient of $L_{r,k}(P_k)$ is direct, and we will only describe in detail how the term,	
	$$\sum_{\ell=0}^{k-r} \binom{i}{\ell} \binom{k-i}{k-r-\ell} (\eta+k-r-\ell)_{\ell} (\eta-\mu)_{k-r-\ell},$$
	arises.
	This term comes from specializing
	
	$$ \prod_{\ell=1}^{k-i} (z_{\ell}+\mu) \prod_{j=k-i+1}^k z_j=(z_1+\mu)...(z_{k-i}+\mu) z_{k-i+1}...z_k
	$$
	along all $(r, k-r)$ shuffles of the composition $ (-\eta, -\eta-1,\dots, -\eta-k+r+1, x_1,...,x_r)$. Such a shuffle will always preserve the ordering of the constants and the variables, and to determine a shuffle it suffices to determine the $k-r$ coordinates in $\{1,\dots,k\}$ at which we have a constant. 
	
	If $\ell$ constant terms appear in the coordinates $\{k-i+1,\dots,k\}$, then those terms are $-\eta-k+r+1,\dots -\eta-k+r+\ell$, and the highest degree term of $z_{k-i+1}\dots z_k$ is $(-\eta-k+r+1)\dots(-\eta-k+r+\ell)=(-1)^{\ell}(\eta+k-r-\ell)_{\ell}$. The remaining constant terms $-\eta-k+r +\ell+1,\dots,-\eta$ must appear in the coordinates $\{1,\dots,k-i\}$, so the highest degree term of $(z_{1}+\mu)\dots(z_{k-i}+\mu)$ is $(-\eta+\mu)\dots (-\eta+ \mu -k+r+\ell+1)=(-1)^{k-r-\ell}(\eta-\mu)_{k-r-\ell}.$
	
	So altogether every way of placing the coordinates so that $\ell$ of them lie in $\{k-i+1,\dots,i\}$ gives a term
	$$(-1)^{k-r}(\eta+k-r-\ell)_{\ell} (\eta-\mu)_{k-r-\ell}.$$
	There are $\binom{i}{\ell}$ ways to choose the coordinates in $\{k-i+1,\dots,k\}$ and $\binom{k-i}{k-r-\ell}$ ways to choose the coordinates in $\{1,\dots,k-i\}$, so there are $\binom{i}{\ell} \binom{k-i}{k-r-\ell}$ shuffles which give this term.
	
	Finally we note that the factor $(-1)^{k-r}$ ends up cancelling with an identical term that appears when we specialize $\prod_{i=1}^k (z_i-\mu)$ to get $(-1)^{k-r} (\eta+\mu)_{k-r}$, and this proves our formula for $C_{r,k}$. 
	
	\textbf{Step 2:} In light of Theorem \ref{thm:characterizationdecomposition}, we need to show that for all $1\leqslant r\leqslant k$, $C_{r,k}=0$. This could be proved directly using properties of hypergeometric functions. We instead provide a probabilistic proof which bypasses some of the hypergeometric series transformations. 
	
	Notice that for $\ell \leqslant i \leqslant k$ with $i\leqslant r+\ell$, 
	$$ \binom{k}{i} \binom{i}{l} \binom{k-i}{k-r-\ell} = \binom{k}{r}\binom{r}{i-\ell} \binom{k-r}{\ell}. $$
	Thus, we may use the change of variables $i=\ell+i'$ to write, 
		\begin{equation*}
	C_{r,k} \frac{1}{\binom{k}{r}(\alpha)_{k-r}(\eta+\mu)_{k-r}} = \frac{(\eta)_{k-r}}{(\eta+\mu)_{k-r}} - \sum_{\ell=0}^{k-r} \binom{k-r}{\ell} \sum_{i'=0}^{r} \binom{r}{i'} \frac{(\eta-\mu)_{k-r-l}}{(\eta)_{k-r-\ell}} \frac{(\mu)_{i'+\ell}(\eta)_{k-i'-\ell}}{(\eta+\mu)_k}.
	\end{equation*}
	Let us assume for the moment that $\eta>\mu$ and let $X,Y$ be two independent random variables such that $X\sim \mathrm{Beta}(\mu,  \eta)$, $Y\sim \mathrm{Beta}(\eta-\mu,  \mu)$. 
	Then, we may write 
		\begin{align*}
	C_{r,k} \frac{1}{\binom{k}{r}(\eta)_{k-r}(\eta+\mu)_{k-r}} &= \frac{(\eta)_{k-r}}{(\eta+\mu)_{k-r}} - \mathbb E\left[\sum_{\ell=0}^{k-r} \binom{k-r}{\ell} Y^{k-r-\ell} \sum_{i'=0}^{r} \binom{r}{i'}   X^{i'+\ell}(1-X)^{k-\ell -i'}\right].\\
	&=\frac{(\eta)_{k-r}}{(\eta+\mu)_{k-r}}-\E\left[ \sum_{\ell=0}^{k-r} \binom{k-r}{\ell} X^{\ell} [Y(1-X)]^{k-r-\ell} \right]\\
	&= \frac{(\eta)_{k-r}}{(\eta+\mu)_{k-r}} - \mathbb E\left[\left(  1-(1-X)(1-Y)\right)^{k-r}\right].
	\end{align*}
	
	Notice that $(1-X)(1-Y)\sim \mathrm{Beta}(\mu, \eta)$. This is a well-known property of beta random variables, which can be proved by computing the moments
	\[\E[[(1-X)(1-Y)]^k]=\E[(1-X)^k] \E[(1-Y)^k]=\frac{(\eta)_k}{(\eta+\mu)_k} \frac{(\mu)_k}{(\eta)_k}=\frac{(\mu)_k}{(\eta+\mu)_k}.\]
	Thus, 
	\[ \E[(1-(1-X)(1-Y))^{k-r}]=\frac{(\eta)_{k-r}}{(\eta+\mu)_{k-r}}, \]
	and finally $C_{r,k}=0$ for any $\eta>\mu>0$. This implies that the equality $C_{r,k}=0$ holds in the field of rational functions $\mathbb C(\eta, \mu)$, thus in particular it also holds for any $\eta, \mu >0$.

\subsection{The simpler direction} \label{sec:easy}

\begin{proposition} \label{prop:easier} For all $k\geq 1$, the vector spaces $\mathbb{BD}_k$ and $\mathbb{SH}_k$, defined in Definitions \ref{def:decomposition} and \ref{def:shuffles}, satisfy 
$$\mathbb{BD}_k \subset \mathbb{SH}_k.$$  
\end{proposition}
 \begin{proof} Let $p\in \BD_k$. Using the decomposition from Definition \ref{def:decomposition}, we may write 
\begin{equation} L_{r,k}(p)= L_{r,k}\left( (z_1+\eta)F_1 \right)+\sum_{i=2}^{k} L_{r,k}  \left( (z_{i-1}-z_{i}-1) F_i \right), 
\label{eq:decompositionspecialized}
\end{equation}
where the $F_i$ satisfy the hypotheses in Definition \ref{def:decomposition}. We will show that the coefficient of $x_1^2\dots x_r^2$ in each term of \eqref{eq:decompositionspecialized} equals $0$. 

\medskip 

By Definition \ref{def:decomposition}, $(z_1+\eta)F_k(z_1,\dots,z_k)$ has $z_1$ degree $\leq 2$ and $F_1(z_1,\dots,z_k)$ is symmetric with respect to $z_1 \leftrightarrow -z_1$ so that  $F_1(z_1,\dots,z_k)$ has $z_1$ degree $0$. Thus for any $(r, k-r)$ shuffle $\sigma$,  $p(\sigma(-\eta,-\eta-1,\dots,-\eta-k+r+1, x_1,\dots,x_{r})$, the term $(z_1+\eta)F_1(z_1,\dots,z_k)$ does not contribute to the coefficient of $x_1^2\dots x_{r}^2$. Indeed,  either $z_1=\eta$ in which case this term is zero, or $z_1=x_{1}$ in which case $(z_1+\eta)F_k(z_1,\dots,z_k)$ has degree $\leq 1$ in $x_{1}$. 

\medskip 

Fix some $i\in \{2,\dots,k\}$. By Definition \ref{def:decomposition}, $(z_{i-1}-z_i-1) F_i(z_1,\dots,z_k)$ has $z_i$ degree $\leq 2$ and $z_{i-1}$ degree $\leq 2$, so that  $F_i(z_1,\dots,z_k)$ has $z_i$ degree $\leq 1$ and $z_{i-1}$ degree $\leq 1$. This means that in the polynomial $(z_{i-1}-z_i-1) F_i(z_1,\dots,z_k)$ there is no monomial which includes $z_i^2 z_{i-1}^2$ as a factor and has nonzero coefficient. 

\medskip 
There are three types of shuffles to consider. 
\begin{itemize}[leftmargin=0.4cm]
	\item Consider  $(r, k-r)$ shuffles  $\sigma$ such that $\sigma^{-1}(i),\sigma^{-1}(i-1) > k-r$. This means that after the substitution  $\vec z\to \sigma(-\eta,-\eta-1,\dots, -\eta-k+r+1, x_1,\dots,x_{r})$, the  variables $z_i, z_{i-1}$ are evaluated into the $x_i's$. Then using the observation above, the coefficient of $x_1^2\dots x_{r}^2$ in  $(z_{i-1}-z_i-1) F_i(z_1,\dots,z_k)$, after the substitution, equals $0$.  
	\item Consider  $(r, k-r)$ shuffles  $\sigma$ such that $\sigma^{-1}(i),\sigma^{-1}(i-1) \leq k-r$. It means that after the substitution  $\vec z\to \sigma(-\eta,-\eta-1,\dots, -\eta-k+r+1, x_1,\dots,x_{r})$, variables $z_i, z_{i-1}$ are evaluated into constants of the form $z_i=-\eta-j$, $z_{i-1}=-\eta-j+1$ for some $j$ (because $\sigma$ is a shuffle). This implies that $(z_{i-1}-z_i-1) F_i(z_1,\dots,z_k)$ vanishes after the substitution. 
	\item Consider now $(r, k-r)$ shuffles $\sigma$ such that either $\sigma^{-1}(i) > k-r$ and $\sigma^{-1}(i-1) \leq k-r$, or $\sigma^{-1}(i) \leqslant k-r$ and $\sigma^{-1}(i-1) > k-r$ for which we employ the same argument.  Assuming we are in the first case,  this means that after the substitution  $\vec z\to \sigma(-\eta,-\eta-1,\dots, -\eta-k+r+1, x_1,\dots,x_{r})$, $z_i$ is evaluated into some $x_j$ and $z_{i-1}$ is evaluated into some constant. Then, the coefficient of $x_1^2\dots x_{r}^2$ in $(z_{i-1}-z_{i}-1)F_i(z_1,\dots,z_k)$ comes entirely from the term $-z_i F_i(z_1,\dots,z_k) \to -x_j F_i(\sigma(-\eta,-\eta-1,\dots, -\eta-k+r+1, x_1,\dots,x_{r}))$. 
	Consider now the shuffle $\bar\sigma=\sigma\circ(i,i-1)$. 
	After the substitution  $\vec z\to \bar\sigma(-\eta,-\eta-1,\dots, -\eta-k+r+1, x_1,\dots,x_{r})$, the coefficient of $x_1^2\dots x_{r}^2$ in $(z_{i-1}-z_i-1)F_i(z_1,\dots,z_k)$ comes from the term \begin{align*}z_{i-1} F_i(z_1,\dots,z_k)&\to x_jF_i(\bar \sigma(-\eta,-\eta-1,\dots, -\eta-k+r+1, x_1,\dots,x_{r}))\\ &=x_j F_i(\sigma(-\eta,-\eta-1,\dots, -\eta-k+r+1, x_1,\dots,x_{r})),\end{align*} where the last equality follows from the symmetry of $F_i$. Hence, the contributions of terms corresponding to $\sigma$ and $\bar\sigma$ cancel each other. 
\end{itemize}
We have shown that the coefficient $x_1^2\dots x_{r}^2$ in $L_{r,k}  \left( (z_{i-1}-z_{i}-1) F_i \right)$ is $0$. We conclude that  $p\in \mathbb{SH}_k$. 

\end{proof}

\subsection{Dimension counting} \label{sec:linearalgebra}

Our aim now is to show that $\dim \mathbb{BD}_k \geq \dim \mathbb{SH}_k$. Together with the inclusion from Proposition \ref{prop:easier}, this  will prove Theorem \ref{thm:characterizationdecomposition}. We first show that $\dim \mathbb{SH}_k=3^k-k-1$.
\begin{lemma} \label{lem:SHdim}The vector space $\mathbb{SH}_k$ has dimension
$$\dim \mathbb{SH}_k=3^k-k-1.$$
\end{lemma}
\begin{proof}
$\mathbb{SH}_k$ is a subspace of the vector space $\mathbb{C}[z_1,\dots,z_k]_{\leq 2,\dots ,2}$ which has dimension $3^k$. The $k+1$ conditions used to define $\mathbb{SH}_k$ are linearly independent when viewed as linear equations in the coordinates $c_{m_1,\dots,m_k}$ of the polynomial 
$$p(z_1,\dots,z_k)=\sum_{(m_1,\dots,m_k) \in C_k} c_{m_1,\dots,m_k} \vec{z}^{\vec{m}}.$$
Indeed, if we list the conditions from $r=k$ to $r=0$, then the $r$th condition is the first time where the coefficient $c_{ \underbrace{0,\dots,0}_{k-r},\underbrace{2,\dots,2}_r}$ makes a nontrivial appearance. We conclude that \[\dim \mathbb{SH}_k=3^k-k-1.\]
\end{proof}

Now, we need to show that $\dim \mathbb{BD}_k \geq 3^k-k-1$. The main result of this section is Proposition \ref{prop:dimensionconditionsimple}, which shows that $\dim \mathbb{BD}_k \geq 3^k-k-1$ provided that $\mathbb{BD}_k$ satisfies the conditions $\IIcond_k$ and $\IIIcond_k$ below. Verifying that these conditions are satisfied will be the subject of the following sections.

\begin{definition} \label{def:cube} Define the cube $C_k=\{0,1,2\}^k$, and let $\vec m \to \vec{z}^{\vec{m}}:=z_1^{m_1}\dots z_k^{m_k}$ be the map from $\vec{m} \in C_k$ to the monomial basis of $\C[z_1,\dots,z_k]_{\leq 2,\dots,2}$ with inner product $\langle \vec{z}^{\vec{m}}, \vec{z}^{\vec{n}} \rangle=\mathds{1}_{\vec{n}=\vec{m}}.$ We will often use the same notation to refer to  a subset of $C_k$ and the corresponding set of monomials in $\mathbb{C}[z_1,\dots, z_k]_{\leq 2, \dots, 2}$ under the bijection $\vec m \to \vec{z}^{\vec{m}}$. For $i \in \{0,1,2\}$, we also define $L_k^i$ to be the set of monomials in $\C[z_1,\dots,z_k]_{\leq 2,\dots,2}$ with exponent of $z_k$ equal to $i$. Each monomial in $C_k$ is in exactly one of $L_k^0$, $L_k^1,$ or $L_k^2$. In the correspondence between monomials and lattice points, $L_k^i \cap C_k$ corresponds to $C_{k-1} \times \{i\}.$
\end{definition}

 In order to  show that $\dim(\mathbb{BD}_k) \geq 3^k-k-1$, we introduce a well-chosen  set of monomials $V_k\subset C_k$ with  $\dim(\vspan(V_k))=3^k-k-1$ and show that $\dim(\mathbb{BD}_k)\geqslant \dim(\vspan(V_k))$. 
\begin{definition}
Let $W_k \subset C_k$ be the set consisting of the point $(2,\dots ,2)$ along with all its nearest neighbors in $C_k$ ( i.e. all points with a $2$ in $k-1$ coordinates and a $1$ in the remaining coordinate). Let $V_k=C_k \setminus W_k$. We have  that $|W_k|=k+1$, so $|V_k|=3^k-k-1$. 
\end{definition}
Observe that if 

\begin{equation*}   
\text{Condition }\hypertarget{Icond}{(\mathrm I)_k} : \quad \vspan(V_k)\subset \BD_k+\vspan(W_k),
\end{equation*}  
is satisfied, then $\dim(\BD_k)\geq 3^k-k-1$. Indeed, since $\vspan(V_k)$ and $\vspan(W_k)$ are orthogonal, the condition $\Icond_k$ implies that $\vspan(V_k)\subset \mathrm{proj}_{\vspan(V_k)}(\BD_k)$, which implies that $\dim \BD_k \geqslant \dim \vspan(V_k)=3^k-k-1$.

In principle, we could prove $\Icond_k$ by checking that each monomial in $V_k$ can be decomposed as an element of $\BD_k$ plus an element of $\vspan(W_k)$. We will see that it is simpler to use a recurrence argument, so that  we only need to work explicitly with a much smaller set of monomials. To this aim, it is useful to remark that, by the definition of $\BD_k$, we have the inclusions 
\begin{equation} \label{eq:inductioninclusionBD} \BD_{k-1} \subset \BD_k, \qquad z_k \BD_{k-1} \subset \BD_k, \qquad z_k^2 \BD_{k-1} \subset \BD_k. \end{equation}
We will decompose the set of monomials $V_k$ using the inclusion 
\begin{equation}
    V_k \subset V_{k-1}\times \lbrace 0,1,2\rbrace \cup W_{k-1} \times \{1\} \cup W_{k-1} \times \{0\}.
    \label{eq:decompositionspaces}
\end{equation}
This motivates the definition of the following two conditions: 
\begin{align*}
  \text{Condition }\hypertarget{IIcond}{(\mathrm{II})_k} :  &\quad  W_{k-1} \times \{1\}\subset \BD_k+\vspan(L^2_k\cup W_k) \\ 
 \text{Condition }\hypertarget{IIIcond}{(\mathrm{III})_k} :   &\quad W_{k-1} \times \{0\}\subset \BD_k+\vspan(L^1_k\cup L^2_k).
\end{align*}

\begin{proposition} \label{prop:dimensionconditionsimple}
If conditions $\IIcond_k$ and $\IIIcond_k$ are satisfied for every $k \geq 1$, then $\dim \mathbb{BD}_k \geq 3^k-k-1$. 
\end{proposition}
\begin{proof}
Assume that $\IIcond_k$ and $\IIIcond_k$ are satisfied for every $k \geq 1$. As we have seen, it suffices to show that $\Icond_k$ is satisfied for all $k$. 

First, observe that in the case $k=1$, then $V_1$ contains the monomial $1$, and $W_1$ contains the monomials $z_1, z_1^2$, and it is easy to see that $\BD_1$ contains the monomial $\frac{1}{\eta}(z_1+\eta)$, so that $\Icond_1$ is satisfied. 

Now we will show that for all $k>1$, $\Icond_{k-1}, \IIcond_k, \IIIcond_k$ imply $\Icond_k$. By recurrence, this will imply that $\Icond_k$ is satisfied for all $k$. Thus, let us fix some $k>1$, assume that $\Icond_{k-1}, \IIcond_k, \IIIcond_k$ hold, and consider the decomposition \eqref{eq:decompositionspaces}. Using the condition $\Icond_{k-1}$ and the inclusion \eqref{eq:inductioninclusionBD}, we have that 
\begin{equation}
    \vspan(V_{k-1}\times \lbrace 0,1,2\rbrace)\subset \BD_k+ \vspan(W_{k-1}\times \lbrace 0\rbrace )+ \vspan(W_{k-1}\times \lbrace 1\rbrace )+ \vspan(W_{k-1}\times \lbrace 2\rbrace )
    \label{eq:firstinclusion}
\end{equation}
Applying conditions $\IIcond_k$ and $\IIIcond_k$ in \eqref{eq:firstinclusion}, and using the facts that $W_{k-1}\times \lbrace 2\rbrace \subset W_k$ and  $W_k\subset L_k^1\cup L_k^2$ we deduce that 
\begin{equation}
    \vspan(V_{k-1}\times \lbrace 0,1,2\rbrace)\subset \BD_k+ \vspan(L_k^1\cup L_k^2 ),
    \label{eq:secondinclusion}
\end{equation}
Now, using the decomposition \eqref{eq:decompositionspaces}, the inclusion \eqref{eq:secondinclusion} and conditions $\IIcond_k, \IIIcond_k$, we also have that 
\begin{equation*}
    \vspan(V_{k})\subset \BD_k+ \vspan(L_k^1\cup L_k^2 ).
\end{equation*}
Observe now that $L_k^1=(V_{k-1}+W_{k-1})\times \lbrace 1\rbrace$ and $L_k^2=(V_{k-1}+W_{k-1})\times \lbrace 2\rbrace$, so that using again the inclusion \eqref{eq:inductioninclusionBD} and condition $\Icond_{k-1}$, 
\begin{equation*}
    \vspan(V_{k}) \subset \BD_k+ \vspan(W_{k-1}\times\lbrace 1\rbrace ) + \vspan(W_{k-1}\times\lbrace 2\rbrace )\subset \BD_k+ \vspan(W_k)+ \vspan(W_{k-1}\times\lbrace 1\rbrace ). 
\end{equation*}
Using again condition $\IIcond_k$, we obtain that 
\begin{equation*}
     \vspan(V_{k}) \subset \BD_k+ \vspan(W_k) + \vspan(L_k^2 ),
\end{equation*}
where we may again use that $L_k^2=(V_{k-1}+W_{k-1})\times \lbrace 2\rbrace$ and $W_{k-1}\times \lbrace 2\rbrace \subset W_k$, so that 
\begin{equation*}
    \vspan(V_{k})  \subset \BD_k+ \vspan(W_k).
\end{equation*}
Hence, $\Icond_k$ is satisfied, and this concludes the proof. 
\end{proof}

\subsection{Condition $\text{II}$} \label{sec:B}
In this section we show that $\IIcond_k$ is satisfied for all $k\ge 1$. The set $ W_{k-1} \times \{1\} $ contains the monomial $z_1^2 z_2^2\dots z_{k-1}^2 z_k$, as well as monomials $z_1^2 \dots z_{i-1}^2 z_i z_{i+1}^2\dots z_{k-1}^2z_k$ for $1 \leq i \leq k-1$. We will show that all these monomials belong to $\BD_k+\vspan(L_k^2 \cup W_k)$ (Proposition \ref{prop:step4conclusion}).
We will use a series of lemmas to build up more complex polynomials eventually culminating in Proposition \ref{prop:step4conclusion}.

\begin{lemma} \label{lem:basictelescope} For all $1 \leq r \leq k-1$, we have
\begin{equation} \label{eq:simplifier}  z_rz_{r+1}\dots z_k (z_r-(k-r))  \prod_{j=1}^{r-1} z_j^{i_j} \in \mathbb{BD}_k+ \vspan(L_k^2), \end{equation}
where we may choose $i_j \in \{0,1,2\}$ arbitrarily for all $1\le j\le r-1$. 
\end{lemma}
\begin{proof}
By the definition of $\BD_k$, $z_i z_{i-1} (z_{i-1}-z_{i}-1) \prod_{j \neq i,i-1} z_j^{i_j}\in \BD_k $ for any $i_j \in \{0,1,2\}$. In particular for $i-1 \geq r$,  $z_r\dots z_k (z_{i-1}-z_{i}-1)\prod_{j =1}^{r-1} z_j^{i_j} \in \BD_k$. We may sum terms of this form to get
\begin{equation}
    z_r\dots z_k \sum_{j=r+1}^{k} (z_{j-1}-z_{j}-1) \left( \prod_{j=1}^{r-1} z_j^{i_j} \right) = z_r\dots z_k \left(z_r-z_k-(k-r) \right) \left( \prod_{j=1}^{r-1} z_j^{i_j} \right).
    \label{eq:simpletelescoping}
\end{equation}
By linearity, $\eqref{eq:simpletelescoping}\in \BD_k $, and we see that the right hand-side of \eqref{eq:simpletelescoping} is equal to the left hand side of \eqref{eq:simplifier}, up to some element of $\vspan(L_k^2)$. This concludes the proof. 
\end{proof}

\begin{lemma} \label{lem:foildown} For all $1 \leq i \leq k-1$, and any $\{i_j\}_{j=1}^{k-i}$ with $i_j \in \{0,1,2\}$ for all $1\le j\le k-i$, we have
\begin{equation}  \label{eq:foildownsimplify}  \prod_{j=1}^{i-1} z_j^{i_j} \left( (z_{i} z_{i+1}^2\dots z_{k-1}^2z_k-(k-i-1)! z_{i} z_{i+1} \dots  z_k) \right)\in \BD_k+\vspan(L_k^2)
\end{equation}

\end{lemma}

\begin{proof} Let us define the polynomial 
\begin{align} \label{eq:foildown1} q_{k}^i(\vec{z}):&=\left( \prod_{j=1}^{i-1} z_j^{i_j} \right) \sum_{r=i+1}^{k-1} (k-r-1)! (z_{r}-(k-r)) z_{i} z_{i+1}^2\dots z_{r-1}^2 z_r\dots z_k \\
 &= \label{eq:foildown2} \left( \prod_{j=1}^{i-1} z_j^{i_j} \right) \sum_{r=i+1}^{k-1} (k-r-1)! z_{i} z_{i+1}^2\dots z_{r}^2 z_{r+1}\dots z_k -(k-r)! z_i z_{i+1}^2\dots z_{r-1}^2 z_{r}\dots z_k.
\end{align}
 On the one hand, each summand in the definition of $q_k^i$ in \eqref{eq:foildown1} lies in  $\BD_k+\vspan(L_k^2)$ by Lemma \ref{lem:basictelescope}, so $q_k^i \in  \BD_k+\vspan(L_k^2)$.  On the other hand, the sum \eqref{eq:foildown2} telescopes to give the left hand side of \eqref{eq:foildownsimplify}.
\end{proof}

\begin{lemma} \label{lem:foilup} For  all $1 \leq i \leq k$, and any $\{i_j\}_{j=1}^{i-1}$ with $i_j \in \{0,1,2\}$ for all $j$, we have
 \begin{equation}  \label{eq:foilupsimplify} \left( \prod_{j=1}^{i-1} z_j^{i_j} \right) \left( z_i\dots z_k-\frac{1}{(k-i)!} z_i^2 \dots z_{k-1}^2 z_k\right) \in \BD_k+\vspan(L_k^2).
 \end{equation}
 \end{lemma}
 \begin{proof}
 
Let us define the polynomial 
\begin{align} \label{eq:foilup1} r^i_k(\vec{z}):&=-\left( \prod_{j=1}^{i-1} z_j^{i_j} \right) \sum_{r=i}^{k-1}  (z_r-(k-r))  \frac{z_i^2 \dots z_{r-1}^2 z_r \dots z_k}{\prod_{s=k-r}^{k-i} s}\\
&=\label{eq:foilup2} -\left( \prod_{j=1}^{i-1} z_j^{i_j} \right) \sum_{r=i}^{k-1} \frac{z_i^2 \dots z_{r}^2 z_{r+1} \dots z_k}{(k-r)(k-r+1) \dots (k-i)}-\frac{z_i^2 \dots z_{r-1}^2 z_r \dots z_k}{(k-r+1)(k-r+2) \dots (k-i)}
\end{align}
On the one hand, each summand in the definition of $r_k^i$ in \eqref{eq:foilup1} lies in $\BD_k +\vspan(L_k^2)$ by Lemma \ref{lem:basictelescope}, so $r_k^i \in \BD_k+\vspan(L_k^2)$. On the other hand, the sum \eqref{eq:foilup2} telescopes to give the left hand side of \eqref{eq:foilupsimplify}. 
\end{proof}

Using Lemmas \ref{lem:foildown} and \ref{lem:foilup},  we can now prove the following proposition which is a restatement of condition $\IIcond_k$.

\begin{proposition} \label{prop:step4conclusion}
We have 
\begin{equation} 
z_1^2 \dots z_{k-1}^2 z_k\in \BD_k+\vspan(L_k^2 \cup W_k),
\label{eq:secondstatementconditionII}
\end{equation}
and for all $1 \leq i \leq k-1$, we have
\begin{equation}
    z_1^2\dots z_{i-1}^2 z_i z_{i+1}^2\dots z_{k-1}^2 z_k\in \BD_k+ \vspan(L_k^2 \cup W_k).
    \label{eq:firststatementconditionII}
\end{equation}
\end{proposition}

\begin{proof}
By definition of $W_k$,  $z_1^2\dots z_{k-1}^2 z_k \in W_k$, so that \eqref{eq:secondstatementconditionII} clearly holds. Let us turn to  \eqref{eq:firststatementconditionII}.  Fix $1 \leq i \leq k-1$. By Lemma \ref{lem:foildown},  choosing  
$i_{1}=\dots =i_{i-1}=2$, we have that 
\begin{equation} z_1^2 \dots z_{i-1}^2 z_i z_{i+1}^2\dots z_{k-1}^2 z_k-(k-i-1)! z_1^2\dots z_{i-1}^2 z_{i}\dots z_k\in \BD_k+\vspan(L_k^2).
\label{eq:defv1}
\end{equation}
By Lemma \ref{lem:foilup},  again choosing $i_{1}=\dots =i_{i-1}=2$, we have that 
\begin{equation}z_1^2\dots z_{i-1}^2 z_{i}\dots z_k-\frac{1}{(k-i)!} z_1^2 \dots z_{k-1}^2 z_k\in \BD_k+ \vspan(L_k^2).
\label{eq:defv2}
\end{equation}
Observe now that  $\eqref{eq:defv1}+(k-i+1)!\times \eqref{eq:defv2}$ simplifies to give
$$z_1^2...z_{i-1}^2 z_i z_{i+1}^2...z_{k-1}^2 z_k  - z_1^2 \dots z_{k-1}^2 z_k \in \BD_k+\vspan(L_k^2),$$
so that using \eqref{eq:secondstatementconditionII}, \eqref{eq:firststatementconditionII} holds. 
\end{proof}

\subsection{Condition $\text{III}$} \label{sec:C} 

In this section we show that condition $\IIIcond_k$ is satisfied for all $k\ge 1$. The set $ W_{k-1} \times \{0\}$ contains the monomial $z_1^2\dots z_{k-1}^2$ as well as the monomials $z_1^2\dots z_{i-1}^2 z_i z_{i+1}^2\dots z_{k-1}^2$ for all $1 \leq i \leq k-1$. We will show that all these monomials are contained in $\mathbb{BD}_k+\vspan(L_k^1 \cup L_k^2)$ (Proposition \ref{prop:finalform}). We use a series of lemmas to build up more complex polynomials eventually culminating in Proposition \ref{prop:finalform}.

\begin{lemma} \label{lem:mostbasictelescope} 
For all $1 \leq r \leq k-1$, the polynomial
\begin{equation} \label{eq:mostbasicsimplifier}\left( z_{r}-(k-r)\right)\prod_{j=1}^{r-1} z_j^{i_j}\in \BD_k+\vspan(L_k^1),
\end{equation}
where we may choose  $i_j \in \{0,1,2\}$ arbitrarily for all $1\le j\le r-1$.
\end{lemma}

\begin{proof}
By definition of $\BD_k$, $ (z_{i-1}-z_{i}-1) \prod_{j \neq i,i-1} z_j^{i_j}\in \BD_k $ for any $i_j \in \{0,1,2\}$. In particular, for $i-1 \geq r$,  $ (z_{i-1}-z_{i}-1)\prod_{j =1}^{r-1} z_j^{i_j} \in \BD_k$. We may sum terms of this form to get
\begin{equation}
    \sum_{j=r+1}^{k} (z_{j-1}-z_{j}-1) \left( \prod_{j=1}^{r-1} z_j^{i_j} \right) =  \left(z_r-z_k-(k-r) \right) \left( \prod_{j=1}^{r-1} z_j^{i_j} \right).
    \label{eq:mostbasictelescoping}
\end{equation}
By linearity, $\eqref{eq:mostbasictelescoping}\in \BD_k $, and we see that the right hand-side of \eqref{eq:mostbasictelescoping} is equal to the left hand side of \eqref{eq:mostbasicsimplifier}, up to some element of $\vspan(L_k^1)$. 
\end{proof}

\begin{lemma} \label{lem:decompose1} We have
$$1 \in \BD_k+\vspan(L_k^1)$$
\end{lemma}
\begin{proof}
$(z_1+\eta)$ is in $\BD_k$ by definition, and $(z_1-(k-1))$ is in $ \BD_k + \vspan(L_k^1)$ by Lemma \ref{lem:mostbasictelescope}, so 
\begin{equation*} \label{eq:decompose1}
    \frac{1}{\eta+k-1} \bigg( (z_1+\eta)-(z_1-(k-1) \bigg)=1 \in \BD_k + \vspan(L_k^1).
\end{equation*}
\end{proof}

\begin{lemma} \label{lem:basicfoil} We have
\begin{equation} \label{eq:foilupfullsimplify3}
 \left( z_{1}\dots z_{k-1}-(k-1)!\right)  \in \BD_k+\vspan(L_k^1 \cup L_k^2).
 \end{equation}
\end{lemma}

\begin{proof}Define the polynomial \begin{align}m_k(\vec{z}):&= \label{eq:basicfoil1}\sum_{r=1}^{k-1}(z_r-(k-r)) z_1 \dots z_{r-1} \prod_{j=r+1}^{k-1} (k-j)\\
&= \label{eq:basicfoil2} \sum_{r=1}^{k-1}  (k-r-1)! z_1 \dots z_r-(k-r)! z_1 \dots z_{r-1}
\end{align}
On the one hand, each summand in the definition of $m_k$ in \eqref{eq:basicfoil1} lies in $\BD_k +\vspan(L_k^1)$ by Lemma \ref{lem:mostbasictelescope}, so $m_k \in \BD_k+\vspan(L_k^1)$. On the other hand, the sum \eqref{eq:basicfoil2} telescopes to give the left hand side of \eqref{eq:foilupfullsimplify3}. 
\end{proof}

\begin{lemma} \label{lem:basictelescope2}
For all $1 \leq r \leq k-1$, we have
\begin{equation} \label{eq:basicsimplifier}  z_rz_{r+1}\dots z_{k-1} (z_r-(k-r))  \prod_{j=1}^{r-1} z_j^{i_j} \in \mathbb{BD}_k+ \vspan(L_k^1 \cup L_k^2), \end{equation}
where we may choose $i_j \in \{0,1,2\}$ arbitrarily for all $1\le j\le r-1$. 
\end{lemma} 
\begin{proof}
By definition of $\BD_k$, $ z_{i-1} z_i (z_{i-1}-z_{i}-1) \prod_{j \neq i,i-1} z_j^{i_j}\in \BD_k $ for any $i_j \in \{0,1,2\}$, and $(z_{k-1}+z_k) (z_{k-1}-z_k-1) \prod_{j=1}^{k-2} z_j^{i_j}$. In particular, for $ k > i >r$,  $ z_r \dots z_{k-1}(z_{i-1}-z_{i}-1)\prod_{j =1}^{r-1} z_j^{i_j} \in \BD_k$, and $z_r\dots z_{k-1}(z_{k-1}+z_k) (z_{k-1}-z_k-1) \prod_{j=1}^{r-1} z_j^{i_j}.$ We may sum terms of this form to get
\begin{align} \label{eq:basictelescoping}
   &\left( z_r\dots z_{k-2} (z_{k-1}+z_k) (z_{k-1}-z_k-1)+
   z_r \dots z_{k-1} \sum_{j=r+1}^{k-1} (z_{j-1}-z_j-1) \right)\prod_{j=1}^{r-1} z_j^{i_j}\\
   &= \label{eq:basictelescoping2}\left([ z_r \dots z_{k-1}(z_r-(k-r))+z_r \dots z_{k-2} z_k (z_{k-1}-z_k-1)-z_r \dots z_k \right) \prod_{j=1}^{r-1} z_j^{i_j}.
\end{align}
By linearity \eqref{eq:basictelescoping} is in $\BD_k$, and we see that \eqref{eq:basictelescoping2} is equal to the left hand side of \eqref{eq:basicsimplifier} up to some element of $\vspan(L_k^1 \cup L_k^2)$. 
\end{proof}

\begin{lemma} \label{lem:foilupfull}
For any $1 \leq i \leq k-1$, we have
\begin{equation} \label{eq:foilupfullsimplify}
 \left( z_1^2\dots z_{i-1}^2 z_{i}\dots z_{k-1} -\frac{(k-1)!}{(k-i)!} z_1\dots z_{k-1} \right) \in \BD_k+\vspan(L_k^1+L_k^2).
\end{equation}
\end{lemma}

\begin{proof} Define the polynomial 
\begin{align} \label{eq:foilupfull}
w_k^i(\vec{z}):&=\sum_{r=1}^{i-1} (z_r-(k-r)) z_1^2\dots z_{r-1}^2 z_r \dots z_{k-1} \prod_{j=r+1}^{i-1} (k-j) \\
&=  \sum_{r=1}^{i-1} \frac{(k-r-1)!}{(k-i)!}  z_1^2 \dots z_r^2 z_{r+1} \dots z_{k-1}  - \frac{(k-r)!}{(k-i)!} z_1^2 \dots z_{r-1}^2 z_r \dots z_{k-1}.\label{eq:foilupfull2} 
\end{align}
On the one hand, each summand in the definition of $w_k^i$ in \eqref{eq:foilupfull} lies in $\BD_k+\vspan(L_k^1+L_k^2)$ by Lemma \ref{lem:basictelescope2}, so $w_k^i \in \BD_k + \vspan(L_k^1+L_k^2).$ On the other hand, the sum \eqref{eq:foilupfull2} telescopes to give the left hand side of \eqref{eq:foilupfullsimplify}. 
\end{proof}

\begin{lemma} \label{lem:fullfoildownsimplify} For all $1 \leq i \leq k-1$, and for any  $\{i_j\}_{j=1}^{i-1}$ with $i_j \in \{0,1,2\}$, we have
\begin{equation} \label{eq:fullfoildownsimplify}
\left( \prod_{j=1}^{i-1} z_j^{i_j} \right) \left( z_i z_{i+1}^2\dots z_{k-1}^2  -(k-i-1)! z_{i}\dots z_{k-1} \right) \in \BD_k + \vspan(L_k^1 \cup L_k^2).
\end{equation}
\end{lemma}
\begin{proof}
Let us define the polynomial
\begin{align} \label{eq:fullfoildown1} v_k^i(\vec{z}):&=
\left( \prod_{j=1}^{i-1} z_j^{i_j} \right)  \sum_{r=i+1}^{k-1} (k-r-1)!  (z_r-(k-r)) z_i z_{i+1}^2 \dots z_{r-1}^2 z_r \dots z_{k-1} \\
&= \label{eq:fullfoildown2} \left( \prod_{j=1}^{i-1} z_j^{i_j} \right) 
\sum_{r=i+1}^{k-1} (k-r-1)! z_i z_{i+1}^2 \dots z_r^2 z_{r+1} \dots z_{k-1}-(k-r)! z_i z_{i+1}^2 \dots z_{r-1}^2 z_r \dots z_{k-1}.
 \end{align}
 Each summand in the definition of $v_k^1$ in \eqref{eq:fullfoildown1} lies in $\BD_k+\vspan(L_k^1+L_k^2)$ by Lemma \ref{lem:basictelescope2}, so $v_k^1 \in \BD_k+\vspan(L_k^1+L_k^2).$ The sum \eqref{eq:fullfoildown2} telescopes to give the left hand side of \eqref{eq:fullfoildownsimplify}. This completes the proof.
\end{proof}
Now we combine  Lemmas \ref{lem:basicfoil}, \ref{lem:foilupfull}, and  \ref{lem:fullfoildownsimplify} to prove the following proposition, which is a restatement of condition $\IIIcond_k$. 
\begin{proposition} \label{prop:finalform}
We have
\begin{equation} \label{eq:finalform}
z_1^2\dots z_{k-1}^2\in \BD_k+\vspan(L_k^1 \cup L_k^2),
\end{equation}
and for all $1 \leq i \leq k-1$, we have
\begin{equation} \label{eq:finalformi}
z_1^2\dots z_{i-1}^2 z_i z_{i+1}^2\dots z_{k-1}^2 \in \BD_k+ \vspan(L_k^1 \cup L_k^2).
\end{equation}
\end{proposition}

\begin{proof}
By Lemma \ref{lem:decompose1}, 
\begin{equation*}
1 \in \BD_k + \vspan(L_k^1).
\end{equation*}
By Lemma \ref{lem:basicfoil}, 
\begin{equation*}
   u= z_1 \dots z_{k-1}-(k-1)! \in \BD_k+\vspan(L_k^1 \cup L_k^2)
\end{equation*}
By Lemma \ref{lem:foilupfull},  we have that for $1 \leq i \leq k-1$, 
\begin{equation*}
   v_i= z_1^2\dots z_{i-1}^2 z_{i}\dots z_{k-1} -\frac{(k-1)!}{(k-i)!} z_1\dots z_{k-1} \in \BD_k+\vspan(L_k^1 \cup L_k^2).
\end{equation*}
By Lemma \ref{lem:fullfoildownsimplify} (choosing  $i_1=\dots=i_{i-1}=2$),  we have that for $1 \leq i \leq k-1$, 
\begin{equation*} w_i= z_1^2\dots z_{i-1}^2z_i z_{i+1}^2\dots z_{k-1}^2-(k-i-1)! z_1^2\dots z_{i-1}^2 z_{i}\dots z_{k-1} \in \BD_k+\vspan(L_k^1\cup L_k^2).
\end{equation*}
Observe that 
\begin{equation*}v_{k-1}+u \frac{(k-1)!}{(k-i)!}+\frac{(k-1)!}{(k-i)!} \in \BD_k+\vspan(L_k^1 \cup L_k^2)
\end{equation*}
simplifies to \eqref{eq:finalform}.
Similarly 
\begin{equation*}
w_i+(k-i-1)!v_i+\frac{(k-1)!}{k-i} u+\frac{[(k-1)!]^2}{k-i} \in \BD_k +\vspan(L_k^1 \cup L_k^2)
\end{equation*}
simplifies to \eqref{eq:finalformi}. 
\end{proof}

Proposition \ref{prop:step4conclusion} and Proposition \ref{prop:finalform} imply conditions $\IIcond_k$ and $\IIIcond_k$. Together with Proposition \ref{prop:dimensionconditionsimple},  this completes the proof of Theorem \ref{thm:characterizationdecomposition}. The proof of Proposition \ref{prop:decomposePk}, assuming Theorem \ref{thm:characterizationdecomposition} is true, was given just after the statement of Theorem \ref{thm:characterizationdecomposition}.

\section{Hankel transforms and asymptotics} 
\label{sec:hankel}
\subsection{Definition and properties of Hankel transforms} \label{sec:hankelgeneral}
In this preliminary section we prove some useful statements about the  Hankel transform, see Definition \ref{def:Hankeltransform}. 
Recall that  
\[F_{\nu}(x):=\sum_{k=0}^{\infty} \frac{x^k}{(\nu)_k k!}. \]
It is useful to observe that 
\begin{equation} \label{eq:Fbessel} F_{\nu}(x)=\frac{\Gamma(\nu)}{\left( \sqrt{-x} \right)^{\nu-1}} J_{\nu-1}\left(2 \sqrt{-x} \right),\end{equation}
where 
\[J_{\alpha}(x):=\sum_{m=0}^{\infty} \frac{(-1)^m}{m! \Gamma(m+\alpha+1)} \left(\frac{x}{2} \right)^{2m+\alpha}\]
is the Bessel function of the first kind. The two square roots appearing on the right hand side of \eqref{eq:Fbessel} must have the same branch cut, but besides that the equality holds for any choice of branch cut. 
\begin{remark}
The Hankel transform of a function $f$ is more commonly defined as 
$$ k\mapsto \int_0^{+\infty} f(r)J_{\nu}(kr) r \mathrm d r,$$
but in view of \eqref{eq:Fbessel}, our definition is essentially equivalent. Orthogonality properties of Bessel functions give an explicit inversion formula for the Hankel transform, but we will see below that this inversion formula is not needed for our purposes. 
\end{remark}
The following results in this section show that, for nonnegative random variables, and assuming $\nu>1/2$,  the Hankel transform acts very similarly to the Fourier transform. In particular, the Hankel transform determines the distribution of a nonnegative random variable, and we will show an equivalent of the L\'evy continuity theorem. We start with a useful estimate. 
\begin{lemma} \label{lem:Fbounded} For $\nu>1/2$, there exists a constant $C>0$, so that $F_{\nu}(x)<C$ for all $x \leq 0$. 
\end{lemma}

\begin{proof} We will combine \eqref{eq:Fbessel} with asymptotics for the Bessel function $J_{\alpha}(z)$. Let $r \geq 0$, then by \eqref{eq:Fbessel} we have
\[F_{\nu}(-r)= \frac{\Gamma(\nu)}{\left(\sqrt{r} \right)^{\nu-1}} J_{\nu-1} \left( 2 \sqrt{r} \right),
\]
where we choose the branch cut of the square root so that $\sqrt{r}>0$. The asymptotics of the Bessel function $J_{\alpha}(s)$ as $s \in \R_+$ goes to $+\infty$ are given by 
\[
J_{\alpha}(s)=\sqrt{\frac{2}{\pi s}} \left( \cos\left( s-\frac{\alpha \pi}{2}-\frac{\pi}{4} \right)+O \left( \frac{1}{s} \right) \right).
\]

Setting $\alpha=\nu-1$ and $s=2 \sqrt{r}>0$, we can see that 
\begin{align*} \label{eq:Fbound}|F_{\nu}(-r)| &= \left|\frac{ \Gamma(\nu)}{(\sqrt{r})^{\nu-1}} \sqrt{\frac{1}{\pi  \sqrt{r}}} \left( \cos\left(s-\frac{\alpha \pi}{2}-\frac{\pi}{4}\right) + O\left(\frac{1}{\sqrt{r}}\right) \right) \right|\\
& \leq \frac{1}{(\sqrt{r})^{\nu-\frac{1}{2}}} \frac{\Gamma(\nu)}{\sqrt{\pi}} \left( 1+ O \left( \frac{1}{\sqrt{r}} \right) \right). 
\end{align*}
The right hand side approaches $0$ as $r \to \infty$ because $\nu>1/2$.  Thus for any fixed $\nu$, $F_{\nu}(-r)$ is bounded on a finite interval by continuity and is bounded as $r \to +\infty$ by the asymptotics \eqref{eq:Fbessel}.
\end{proof}
The Hankel transform is closely related to the Laplace transform, in the following sense. 
\begin{lemma} \label{lem:hankeltolaplace}
Let $X$ be  any  nonnegative random variable and let $Z$ be  an independent gamma distributed random variable with parameter $\nu$. For any $t \geq 0$ and $\nu>1/2$,
$$\E[F_{\nu}(-t Z X)]=\E[e^{-t X}], $$
and both sides of this equation are finite. 
\end{lemma}
\begin{proof}
Let $h(-t)=\E[F_{\nu}(-t X)]$, $F_{\nu}(x)$ is bounded for $x \leq 0$ by Lemma \ref{lem:Fbounded}, so this expectation exist for $t\geq 0$. We compute
$$\E_{Z}[h(-t Z)]=\E_Z \E_{X}[F_{\nu}(-t Z X)].$$
Again because $F_{\nu}(x)$ is bounded for $x<0$, we can switch the expectations to obtain
$$\E_{Z}[h(-t Z)]=\E_{X} \E_{Z}[F_{\nu}(-t Z X)]=\E_{X}\E_Z \left[\sum_{k=0}^{\infty} \frac{(- t X)^k}{k!} \frac{Z^k}{(\nu)_k} \right].$$
The expectation over $Z$ can be switched with the summation over $k$ for any fixed value $X=x$, so 
$$\E_Z[h(-t Z)]=\E_{X} \left[\sum_{k=0}^{\infty} \frac{(-t X)^k}{k!} \frac{\E[Z^k]}{(\nu)_k} \right]=\E[e^{-t X}],$$
where in the last equality we use that the moments of a gamma random variable are $\E[Z^k]=(\nu)_k$. 
\end{proof}
Now we show that the Hankel transform of a nonnegative random variable uniquely determines its distribution. 
\begin{lemma} \label{lem:hankelcharacterize}
Let $\nu>1/2$, and let $X$ and $Y$ be nonnegative random variables, so that for $t \geq 0$, 
$$\E[F_{\nu}(-t X)]=\E[F_{\nu}(-t Y)],$$
then $X$ and $Y$ have the same distribution.
\end{lemma}
\begin{proof}
Apply Lemma \ref{lem:hankeltolaplace} to obtain $\E[e^{-t X}]=\E[e^{-t Y}]$ for all $t \geq 0$. Now $X$ and $Y$ are nonnegative random variables whose Laplace transforms agree on the negative real line, so they are equal by \cite[p.430 Theorem 1]{feller2008introduction}.
\end{proof}
The next proposition is the equivalent of L\'evy's continuity theorem. It says we can characterize weak limits by pointwise convergence of Hankel Transforms.
\begin{proposition}
Let $\nu>1/2$ and let $X_1, X_2,\dots $ be nonnegative random variables. For all $t\leq 0$, Let 
$$h_n(t)=\E[F_{\nu}(t X_n)].$$
 If  $h_n(t) \to h(t)$ pointwise for $t \leq 0$, with $h$ left continuous at $0$, then $X_n \to X$ weakly, where $X$ is the unique random variable satisfying
$$h(t)=\E[F_{\nu}(t X)],$$
for all $t \geq 0$. 
\label{prop:Levycontinuity}
\end{proposition}
\begin{proof}
Let $Z$ be a gamma distributed random variable with parameter $\nu$ that is independent of $X_1, X_2,\dots $. By Lemma \ref{lem:hankeltolaplace}, 
$$\E_Z[h_n(-t Z)]=\E[e^{-t X_n}].$$
Because $F_{\nu}$ is bounded for negative arguments, so is $h_n$, thus $\E_{Z}[h_n(-t Z)] \to \E_{Z}[h(-t Z)]$ pointwise for $t\leq 0$ by dominated convergence. 

To recap, we have nonnegative random variables $X_n$ whose Laplace transforms $\E[e^{-t X_n}]$ converge to a function $\E_Z[h(-t Z)]$ pointwise for $t>0$. Now $h(0)=F_{\nu}(0)=1$. Left continuity and boundedness of $h$ allow us to use dominated convergence to see $\lim_{t \to 0} \E_Z[h(-t Z)]=1$. Now the continuity theorem \cite[p. 431 Theorem 2]{feller2008introduction} gives that the $X_n$ converge weakly to the unique random variable $X$ whose Laplace transform is equal to $\E_Z[h(-t Z)]$ for $t>0$. $F_{\nu}(x)$ is a bounded continuous function for $x \leq 0$, so for $t\leq 0$, 
$$h(t)=\lim_{n \to \infty} \E[F_{\nu}(t X_n)]=\E[F_{\nu}(t X)].$$
By Lemma \ref{lem:hankelcharacterize},  $X$ is the unique random variable satisfying $\E[F_{\nu}(t X)]=h(t).$  
\end{proof}

To conclude this preliminary section about the Hankel transform, we remark that in certain asymptotic regimes, the limiting distribution of a random variable can be recovered directly from its Hankel transform, without using any explicit inversion formula. This is an analogue of a similar result about the Laplace transform that was used in \cite{borodin2014macdonald, borodin2014duality} and many subsequent articles, see  \cite[Lemma 4.1.39]{borodin2014macdonald}.
\begin{lemma} \label{lem:hankellimit}
Let $\nu>1/2$, let $X_t$ be a nonnegative random variable for each $t>0$, let $h_t(s)=\E[F_{\nu}(s X_t)]$, and let $\zeta_t(y)=e^{t a -t^{1/3} b y}$. If for all $c>0$, 
$$\lim_{t \to \infty} h_t(-\zeta_t(y-t^{-1/3} c))=F(y)$$ for some continuous cumulative density function $F$, then we have weak convergence
$$\frac{\log(X_t) -a t}{b t^{1/3}} \xRightarrow[t \to \infty]{} X$$
where $X$ is the random variable defined by 
$\mathbb P(X \leq x)=F(x)$. 
\end{lemma}
Heuristically, this comes from the fact that for $\nu>1/2$, 
$$F_{\nu}(-e^{t^{1/3} x}) \xrightarrow{t \to \infty} 
\begin{cases}
0 & x>0,\\
1 & x<0.\\
\end{cases}$$
so that 
$$\lim_{t \to \infty}\E[F_{\nu}(-e^{t a -t^{1/3} b y} X_t)]   = \lim_{t \to \infty} \P \left( \frac{\log X_t+t a}{b t^{1/3}}<y \right).$$
\begin{proof}
Let $Z$ be a Gamma distributed random variable with parameter $\nu$. Apply Lemma \ref{lem:hankeltolaplace} to see that $\E[e^{-\zeta_t(y) X_t}]=\E_{Z}[h_t(-Z \zeta_t(y))]$. Use this to compute
$$\lim_{t \to \infty} \E[e^{-\zeta_t(y) X_t}]=\lim_{t \to \infty} \int_0^{\infty} h_t(-z \zeta_t(y)) \frac{z^{\nu-1} e^{-z}}{\Gamma(\nu)} dz=\lim_{t \to \infty} \int_0^{\infty} h_t(-\zeta_t(y-t^{-1/3} b^{-1} \log(z))) \frac{z^{\nu-1} e^{-z}}{\Gamma(\nu)} dz $$
The functions $h_t$ are bounded uniformly in $t$ for negative real arguments, so we can apply dominated convergence to obtain
$$\lim_{t \to \infty} \E[e^{-\zeta_t(y) X_t}]=\int_0^{\infty} \frac{ z^{\nu-1} e^{-z}}{\Gamma(\nu)} \lim_{t \to \infty} h_t(- \zeta_t(y-t^{-1/3} b^{-1} \log(z)) dz=F(y).$$
Hence, we have reduced the problem to a statement about the Laplace transform and the application of \cite[Lemma 4.1.39]{borodin2014macdonald} completes the proof. 
\end{proof}

\begin{remark}
The proofs of Lemmas \ref{lem:hankelcharacterize}  and Proposition \ref{prop:Levycontinuity} require very little about the specific form of $F_{\nu}$. If we replace $F_{\nu}$ by any entire function, bounded on the negative reals which has the form $\sum_{k=0}^{\infty} \frac{ x^k}{k! \E[Z^k]}$ for some nonnegative random variable $Z$, these statements will continue to hold by the same proofs. 
\end{remark}

\subsection{Simple Pfaffian formula} 
\label{sec:Cauchyformula}
In this section we prove Proposition \ref{prop:CauchytypeFredholm}. By the definition of $F_{\nu}$ in \eqref{eq:defFnu}, 
\begin{equation*}
    \mathbb E\left[F_{\mu+\eta}(-\zeta  \msf{P}_{0, 2t}(x,1)) \right]  = \mathbb E\left[ \sum_{k=0}^{\infty} \frac{(-\zeta)^k}{ k! (\mu+\eta)_k} \msf{P}_{0, 2t}(x,1)^k \right].
\end{equation*}
 Since $\mathbb E\left[ \msf{P}_{0, 2t}(x,1)^k\right]<1$, we may exchange the expectation with the sum and obtain that 
 \begin{equation}
 \label{eq:generatingseries}
    \mathbb E\left[F_{\mu+\eta}(-\zeta  \msf{P}_{0, 2t}(x,1)) \right]  = \sum_{k=0}^{+\infty} \frac{(-\zeta)^k}{ k! (\mu+\eta)_k} \mathbb E\left[  \msf{P}_{0, 2t}(x,1)^k \right],
 \end{equation}
 where we recall that  the expectation is given by the integral formula in  \eqref{eq:momentformula}. 
The product over $a<b$ in \eqref{eq:momentformula} can be written as 
\begin{multline}
\prod_{a<b} \frac{z_a - z_b}{z_a-z_b-1} \frac{z_a + z_b}{z_a+z_b+1} = \prod_{a<b} \frac{z_a - z_b}{z_a-z_b-1} \frac{-z_a - z_b}{-z_a-z_b-1} \frac{-z_a + z_b}{-z_a+z_b-1} \frac{z_a + z_b}{z_a+z_b-1} \\ 
\prod_{a<b} \frac{z_b-z_a-1}{z_b-z_a} \frac{z_b+z_a-1}{z_b+z_a}. 
\label{eq:decompositionproduct}
\end{multline}
Observe that the first line is BC-symmetric (i.e invariant with respect to permutations of variables and transformations $z_j\to  -z_j$ for all $1\le j\le k$). Further, it can be written as a Pfaffian (see Appendix \ref{sec:backgroundPfaffian} for background on Pfaffians). Indeed,
\begin{equation}
\prod_{a<b} \frac{z_a - z_b}{z_a-z_b-1} \frac{z_a + z_b}{z_a+z_b-1} \frac{z_a - z_b}{-z_a+z_b-1} \frac{z_a + z_b}{-z_a-z_b-1} = \Pf\left[  \frac{u_i-u_j}{u_i+u_j}\right]_{i,j=1}^{2k} \prod_{i=1}^k \frac{1}{2z_i},
\label{eq:factorizationPfaffian}
\end{equation}
where $(u_1, \dots , u_{2k}) = (z_1-1/2, -z_1-1/2, z_2-1/2, \dots, -z_k-1/2)$. Since the contours are all the same in \eqref{eq:momentformula}, and they are invariant by the transformation $z_j\to  -z_j$ for all $1\le j\le k$, we may symmetrize the integrand (with respect to the action of signed permutations) and write 
\begin{equation}
\frac{1}{2^k (\mu+\eta)_k} \mathbb E\left[  \msf{P}_{0, 2t}(x,1)^k  \right] =  \int_{\I \R} \frac{dz_1}{2\I\pi} \cdots \int_{\I \R} \frac{dz_k}{2\I\pi}  \Pf\left[  \frac{u_i-u_j}{u_i+u_j}\right]_{i,j=1}^{2k} \prod_{i=1}^k \frac{1}{2z_i} E(z_1, \dots, z_k),
\end{equation}
where 
\begin{multline*}
E(z_1, \dots, z_k) =  \\
\frac{1}{2^k k!} \sum_{\sigma\in BC_k} \sigma\left( \prod_{a<b} \frac{z_b-z_a-1}{z_b-z_a} \frac{z_a+z_b-1}{z_a+z_b} \prod_{i=1}^k \left( \frac{z_i^2}{z_i^2-\mu^2} \right)^{t} \left(\frac{z_i-\mu}{z_i+\mu}\right)^{(x-1)/2} \frac{z_i}{(z_i^2-\mu^2)(z_i+\eta)} \right),
\end{multline*}
where for any function $f$ and signed permutation $\sigma$, the notation $\sigma(f(z_1, \dots, z_ k))$ means $f(z_{\sigma(1)}, \dots, z_{\sigma(k)})$, where we follow the convention that $z_{-j}=-z_j$ for all $1\leq j\leq k$. 
In general, the function $E$ does not admit a simple factorized expression. However, when $x=1$, it can be simplified. Indeed, for any parameter $A\in \mathbb C$ and complex variables $(z_i)_{1\leqslant i\leqslant n}$, we have 
\begin{equation}
\sum_{\sigma\in BC_k} \sigma\left(  \prod_{a<b} \frac{z_b-z_a+1}{z_b-z_a} \frac{z_b+z_a+1}{z_b+z_a}\prod_{j=1}^k \frac{z_j+A}{z_j} \right)= 2^k k!.\label{eq:BCsymetrization}
\end{equation}
A generalization of it was proved in \cite[Theorem 2.6]{venkateswaran2015symmetric} in the context of BC analogues of Hall-Littlewood polynomials. We refer to \cite[Eq. (54)]{borodin2016directed} for more details about how to degenerate Venkateswaran's symmetrization identity \eqref{eq:BCsymetrization}. 
Thus, in the case $x=1$, using \eqref{eq:BCsymetrization} with $A=-\eta$ we obtain that 
\begin{equation}
E(z_1, \dots, z_k) =\prod_{i=1}^k \left( \frac{z_i^2}{z_i^2-\mu^2} \right)^{t}  \frac{z_i^2}{(z_i^2-\mu^2)(z_i^2-\eta^2)},
\end{equation}
and we obtain that
\begin{multline}
\frac{1}{2^k (\mu+\eta)_k} \mathbb E\left[  \msf{P}_{0, 2t}(1,1)^k  \right] =\\ \int_{\I \R} \frac{dz_1}{2\I\pi} \cdots \int_{\I \R} \frac{dz_k}{2\I\pi}  \Pf\left[  \frac{u_i-u_j}{u_i+u_j}\right]_{i,j=1}^{2k} \prod_{i=1}^k  \left( \frac{z_i^2}{z_i^2-\mu^2} \right)^{t}  \frac{z_i}{2(z_i^2-\mu^2)(z_i^2-\eta^2)}.
\label{eq:Pfaffiansummand}
\end{multline}
We may simplify the factors of $2$ in \eqref{eq:Pfaffiansummand}, form the generating series \eqref{eq:generatingseries}, and we find that 
\begin{equation} 
   \mathbb E\left[F_{\mu+\eta}(-\zeta \msf{P}_{0, 2t}(1,1)) \right] = \Pf (J-\zeta K)_{\mathbb L^2(\I\R\times \lbrace 1,2\rbrace)},
\label{eq:simpleFredholm}
\end{equation}
where $K$ is the $2\times 2$ matrix kernel given by 
\begin{subequations}
	\begin{align}
	K_{11}(z,w) &= \frac{z-w}{z+w+1}\left( \frac{z^2}{z^2-\mu^2} \right)^t  \frac{z}{(z^2-\mu^2)(z^2-\eta^2)} ,\\
	K_{12}(z,w) &= -K_{2,1}(w,z) = \frac{z+w}{z-w+1}\left( \frac{z^2}{z^2-\mu^2} \right)^{t}  \frac{z}{(z^2-\mu^2)(z^2-\eta^2)},\\
	K_{22}(z,w) &= \frac{w-z}{-z-w+1}\left( \frac{w^2}{w^2-\mu^2}\right)^{t}\frac{w}{(w^2-\mu^2)(w^2-\eta^2)}.
	\end{align}
	\label{eq:defKproof}
\end{subequations}
The Fredholm Pfaffian in \eqref{eq:simpleFredholm} is well-defined (see Lemma \ref{lem:welldefinedPfaffian}) since there exist a constant $C>0$ such that the following bounds hold for any $z,w\in \I\R$, 
\begin{align*}
\vert	K_{11}(z,w) \vert &\leqslant \frac{C}{\vert z^2-\mu^2\vert} ,\\
\vert	K_{12}(z,w) \vert &\leqslant \frac{C}{\vert z^2-\mu^2\vert} ,\\
\vert	K_{22}(z,w) \vert  &\leqslant \frac{C}{\vert w^2-\mu^2\vert},
\end{align*}
and these bounds are clearly in $L^2(\I\mathbb R)$.
\subsection{Mellin-Barnes type formula} 
\label{sec:MellinBarnesformula}
 In this section we prove Proposition \ref{prop:mellinbarnes}. We will show that 
for  $\mu, \eta>0$, for  $x,t$ being positive integers so that $x+t$ is odd, and  for any $\zeta \in \mathbb{C}\setminus \mathbb R_{<0}$, we have 
\begin{multline}
\label{eq:MellinBarnesintro}
\E\left[F_{\mu+\eta} \left( -\zeta \msf P_{0,t}(x,1) \right) \right]=\\
1+\sum_{\ell=1}^{\infty} \frac{1}{\ell!} \int_{\mathcal C_{1/2}^{\pi/3}} \frac{d \lambda_1}{2 \pi \I} \dots \int_{\mathcal C_{1/2}^{\pi/3}} \frac{d \lambda_{\ell}}{2 \pi \I}  \oint_{\gamma} \frac{dw_1}{2 \pi \I} \dots \oint_{\gamma} \frac{d w_{\ell}}{2 \pi \I} \prod_{i=1}^{\ell} \frac{\pi}{\sin(-\pi \lambda_i)} I_{\ell}(\lambda, w;\zeta,x,t),
\end{multline}
where $\gamma$ is a positively oriented circle around $\mu$ with radius smaller than $1/4$, and
\begin{multline}\label{eq:defIell}
I_{\ell}(\lambda,w;\zeta,x,t):= \det  \left[\frac{1}{w_i+\lambda_i-w_j} \right]_{i,j=1}^{\ell}
\prod_{1 \leq i <j \leq \ell(\lambda)} \frac{\Gamma(w_i+w_j+\lambda_i) \Gamma(w_i+w_j+\lambda_j)}{\Gamma(w_i+w_j) \Gamma(w_i+w_j+\lambda_i+\lambda_j)} \\
\times \prod_{i=1}^{\ell} \zeta^{\lambda_i}  \frac{ \Gamma(2 w_i +\lambda_i)}{ \Gamma(2 w_i )} 
\frac{ \Gamma(\eta+w_i) \Gamma(\mu+w_i) \Gamma(w_i-\mu)}{ \Gamma(\eta+w_i+\lambda_i) \Gamma(\mu+w_i+\lambda_i) \Gamma(w_i-\mu+\lambda_i)}. \\
\times \left( \frac{\Gamma(w_i-\mu+\lambda_i) \Gamma(\mu+w_i)}{\Gamma(w_i-\mu)\Gamma(\mu+w_i+\lambda_i)} \right)^{(x-1)/2} 
 \left[ \left( \frac{\Gamma(w_i+\lambda_i)}{\Gamma(w_i)} \right)^2 \frac{\Gamma(w_i-\mu) \Gamma(w_i+\mu)}{\Gamma(w_i-\mu+\lambda_i) \Gamma(w_i+\mu+\lambda_i)} \right]^{t/2}. 
\end{multline}
The statement of Proposition \ref{prop:mellinbarnes} is then obtained by the change of variables $z_i=w_i+\lambda_i$. After the change of variables, the contours need to be chosen so that  the contour for variables $w_i$ lies on the left of the contour for variables $z_i$, while $w_i+1$ lies on the right of the contour for variables $z_i$. The contour for variables $w_i$ can remain $\gamma$, a circle around $\mu$ with radius smaller than $1/4$ (not containing $-\mu$ not $-\eta$), and the contour for the variables $z_i$ becomes $\mathcal C_{\mu+1/2}^{\pi/3}$, as in the statement of Proposition \ref{prop:mellinbarnes}. 
\medskip 

Our starting point is the integral formula for mixed moments  given in  \eqref{eq:momentformula}. We may deform the contours of integration in order, beginning with the $z_k$ contour and ending with the $z_1$ contour, to obtain the formula
\begin{multline}\label{eq:momentformulanested}
\frac{1}{2^k (\mu+\eta)_k} \mathbb E\left[  \prod_{i=1}^k\msf{P}_{0, t}(x_i,1) \right] =  \int_{\gamma_1} \frac{dz_1}{2 \pi \I}\dots \int_{\gamma_k} \frac{dz_k}{2 \pi \I} \prod_{1 \leq a< b \leq k} \frac{z_a-z_b}{z_a-z_b-1} \frac{z_a+z_b}{z_a+z_b+1} \\ \times\prod_{i=1}^k \left( \frac{z_i^2}{z_i^2-\mu^2} \right)^{t/2} \left( \frac{z_i-\mu}{z_i+\mu} \right)^{(x_i-1)/2} \frac{z_i}{(z_i^2-\mu^2)(z_i+\eta)},
\end{multline}
where the contours all enclose $\mu$, and for $a<b$, the contour for $z_a$ (i.e. $\gamma_a$) encloses $z_b+1$. In order to obtain such contours from \eqref{eq:momentformula}, we start deforming the  contour $\gamma_k$ to be  a small  circle  around $\mu$, without crossing any singularities, and change the orientation so that the contour becomes positively oriented (this is why we have cancelled a factor $(-1)^k$ present in \eqref{eq:momentformula}). We then proceed to deform  $\gamma_{k-1}$ to be  a positively oriented circle containing $\mu$ and $\gamma_k+1$, and iteratively, the other contours are  nested into each other so that the above condition is satisfied, with $\gamma_1$ being the outermost contour.  

We will now further shrink all contours to the small contour $\gamma:=\gamma_k$. This contour deformation now crosses many singularities, and we need to keep track of all residues. It turns our that the contribution of all residues can be nicely rearranged as a sum of integrals of determinants, this was originally discovered in \cite[Proposition 3.2.1]{borodin2014macdonald}. 
\begin{proposition} \label{prop:momentsamecontour} Let $\mu, \eta>0$, let $t$ and the $x_i$ be  positive integers such that  $1\leq x_1 \leq \dots  \leq x_k$ with $t+x_i$ odd. We have 
\begin{multline} \label{eq:nonqresidue} \frac{1}{2^k (\mu+\eta)_k}\mathbb E\left[  \prod_{i=1}^k\msf{P}_{0, t}(x_i,1) \right]= \sum_{\underset{\lambda=1^{m_1}2^{m_2}\dots }{\lambda \vdash k}} \frac{1}{m_1!m_2!\dots }   \oint_{\gamma} \frac{dw_1}{2 \pi \I} \dots  \oint_{\gamma} \frac{dw_{\ell(\lambda)}}{2 \pi \I} \mathrm{det}\left[\frac{1}{w_i+\lambda_i-w_j} \right]_{i,j=1}^{\ell(\lambda)} \\
\times \overline{E}(w_1, w_1+1,\dots ,w_1+\lambda_1-1,\dots ,w_{\ell(\lambda)}, w_{\ell(\lambda)}+1,\dots ,w_{\ell(\lambda)}+\lambda_{\ell(\lambda)}-1),
\end{multline}
where $\gamma$ is a small positively oriented circle around $\mu$ containing no other singularity, and 
\begin{multline}\label{eq:defoverlineE} \overline{E}(z_1,\dots ,z_k)=\prod_{1 \leq a<b \leq k} \frac{z_a+z_b}{z_a+z_b+1} \prod_{i=1}^k \left( \frac{z_i^2}{z_i^2-\mu^2} \right)^{t/2} \frac{z_i}{(z_i^2-\mu^2)(z_i+\eta)}  \\
\times \sum_{\sigma \in S_k}  \sigma \left( \prod_{1 \leq B<A \leq k} \frac{z_{A}-z_{B}-1}{z_{A}-z_{B}}
 \prod_{i=1}^k \left( \frac{z_i-\mu}{z_i+\mu} \right)^{(x_i-1)/2} \right), 
\end{multline}
where the permutation $\sigma$ acts on the expression inside parentheses by permuting variables.
\end{proposition}
\begin{proof}
Going from \eqref{eq:momentformulanested} to \eqref{eq:nonqresidue} is almost a direct application of \cite[Proposition 5.1]{borodin2016directed}. The only difference is that \cite[Proposition 5.1]{borodin2016directed} deals with infinite contours, while we are working with finite contours. The structure of residues that are encountered during the contour deformations $\gamma_i\to \gamma$ are, however, exactly the same as in the proof of \cite[Proposition 5.1]{borodin2016directed}. Alternatively, one may use \cite[Proposition 7.2]{borodin2016directed}, which deals with finite contours, set $y_i=q^{z_i}$ there, and let $q$ goes to $1$: this yields an analogue of  \cite[Proposition 5.1]{borodin2016directed} for finite contours. 
\end{proof}
At this point, we notice that  when the $x_i$ are all equal,  the function $\overline E$ in \eqref{eq:defoverlineE} simplifies thanks to  the symmetrization formula \cite[Chap. III, (1.3)]{macdonald1995symmetric} (see also  \cite[Equation (53)]{borodin2016directed}) 
\begin{equation} \label{eq:symmetricBethesum}\sum_{\sigma \in S_k} \prod_{1 \leq b<a \leq k} \frac{z_{\sigma(a)}-z_{\sigma(b)}-c}{z_{\sigma(a)}-z_{\sigma(b)}}=k! \end{equation}
Forming the generating series \eqref{eq:generatingseries}, we obtain the following. 
\begin{proposition} \label{prop:premellinbarnes} Let $\mu, \eta>0$ and let $x,t$ be positive integers so that $x+t$ is odd. For any $\zeta\in \mathbb C\setminus \mathbb R_{>0}$, we have 
\begin{equation}
\label{eq:goodsummationformula}
\mathbb E\left[F_{\mu+\eta}\left( -\zeta \msf{P}_{0,t}(x,1) \right) \right]=1+\sum_{\ell=1}^{\infty} \frac{1}{\ell!}\sum^{\infty}_{\lambda_1,...,\lambda_{\ell}=1} \oint_{\gamma} \frac{dw_1}{2 \pi \I} \dots \oint_{\gamma} \frac{d w_{\ell}}{2 \pi \I} I_{\ell}(\lambda,w; -\zeta,x,t),
\end{equation}
where the integrand $I_{\ell}(\lambda, w; x,t)$ was defined in \eqref{eq:defIell}. 
\end{proposition}
\begin{proof}
For $x_1=\dots=x_k=x$, we find that 
\[\overline{E}(z_1,\dots,z_k)= k! \prod_{1 \leq a<b \leq k} \frac{z_a+z_b}{z_a+z_b+1}  \prod_{i=1}^k \left( \frac{z_i^2}{z_i^2-\mu^2} \right)^{t/2} \left( \frac{z_i-\mu}{z_i+\mu} \right)^{(x-1)/2} \frac{z_i}{(z_i^2-\mu^2)(z_i+\eta)}.
\]
This function can be explicitly evaluated into $(w_1,\dots ,w_1+\lambda_1-1,\dots ,w_{\ell(\lambda)},\dots ,w_{\ell(\lambda)}+\lambda_{\ell(\lambda)}-1)$. Combining Proposition \ref{prop:momentsamecontour} with equation \eqref{eq:generatingseries}, we obtain that 
\begin{equation} \label{eq:defg} 
\mathbb{E} \left[F_{\mu+\eta}\left( -\zeta \msf{P}_{0, 2t}(x,1) \right) \right]=\sum_{k=0}^{\infty} \sum_{\lambda \vdash k} \frac{1}{m_1! m_2!\dots} f(k, \lambda),
\end{equation}
with
\begin{equation}\label{eq:fklambda} f(k,\lambda)=  \oint_{\gamma} \frac{dw_1}{2 \pi \I} \dots \oint_{\gamma} \frac{d w_{\ell(\lambda)}}{2 \pi \I} I_{\ell}(\lambda,w;-\zeta,x,t)
\end{equation}
We now reindex the sum in  \eqref{eq:defg} to be over the variables $\ell=\ell(\lambda)$  and $\lambda_1,\dots,\lambda_{\ell}$. We may further relax the condition $\lambda_1 \geq \lambda_2 \geq \dots \geq \lambda_{\ell}$ and use the symmetry of $f$ to obtain the sum
\begin{equation} \label{eq:reindexedsum2} \sum_{\ell=1}^{\infty} \frac{1}{\ell!} \sum_{\lambda_1=1,\dots,\lambda_{\ell}=1}^{\infty} f \left( \sum_{i=1}^{\ell} \lambda_i, \lambda \right), \end{equation}
which is exactly the statement of Proposition \ref{prop:premellinbarnes}. 

All that remains is to show that this reordering  does not change the value of the infinite sum. To do this we will bound $f(\sum_{i} \lambda_i,\lambda)$ to show that \eqref{eq:reindexedsum2} converges absolutely. We will consider the factors appearing in \eqref{eq:fklambda} one by one. Fix some arbitrary  $\mu>0$, and assume that $\gamma$ is a circle around $\mu$ with radius less than $\min\lbrace 1/4, \mu\rbrace$. Further assume that $|\zeta|=r<1$. The determinant can be bounded using Hadamard's inequality (see \eqref{eq:hadamard}), which yields  
\begin{equation} \label{eq:hadamard}
\frac{1}{\ell!} \prod_{i=1}^{\ell} \zeta^{\lambda_i} \left| \det \left[ \frac{1}{w_i+\lambda_i-w_j} \right]_{i,j=1}^{\ell} \right| \leq  \varepsilon^{-\ell} \frac{\ell^{\ell/2}}{\ell!} \prod_{i=1}^{\ell}  (2r)^{\lambda_i}.
\end{equation}
Observe that the right hand side of \eqref{eq:hadamard} has superexponential decay in $\ell$. Further, if some complex numbers $z_a, z_b$ have positive real part, then 
$$ \left\vert\frac{z_a+z_b}{z_a+z_b+1} \right\vert <1,$$ 
which implies the simple bound 
$$\left\vert  \frac{\Gamma(w_i+w_j+\lambda_i) \Gamma(w_i+w_j+\lambda_j)}{\Gamma(w_i+w_j) \Gamma(w_i+w_j+\lambda_i+\lambda_j)}  \right\vert <1.$$
Next, consider the factor 
$$\left| \frac{\Gamma(w_i-\frac{1}{2}+\lambda_i) \Gamma(2 w_i +\lambda_i)}{\Gamma(\frac{1}{2}+w_i) \Gamma(2 w_i +2 \lambda_i-1)} 
\frac{\Gamma(w_i+\lambda_i) \Gamma(\eta+w_i) \Gamma(\mu+w_i) \Gamma(w_i-\mu)}{\Gamma(w_i) \Gamma(\eta+w_i+\lambda_i) \Gamma(\mu+w_i+\lambda_i) \Gamma(w_i-\mu+\lambda_i)} \right|.$$
This term as superexponential decay in $\lambda_i$ as $\lambda_i \to \infty$. All remaining factors  in \eqref{eq:fklambda} take the form $\Gamma(a+\lambda_i)/\Gamma(b+\lambda_i)$ with $a$ and $b$ bounded (possibly raised to a fixed power $t$ or $x$). We have the asymptotics of Gamma function ratios  \cite{tricomi1951asymptotic} 
\begin{equation}\label{eq:Gammaratios}\frac{\Gamma(z+a)}{\Gamma(z+b)}=z^{a-b}+o(z^{a-b}), \quad \text{ as } |z| \to \infty\end{equation}
provided that $z$ avoids the poles of $\Gamma(z+\beta)$ and $\Gamma(z+\alpha)$, and avoids the branch cut used to define $z^{\alpha-\beta}$. This 
implies  that all these factors have at most polynomial growth in $\lambda_i$ (possibly raised to a fixed power $t$ or $x$.) All these bound together imply that $|f(\sum_{i=1}^{\ell} \lambda_i, \lambda)|$ has superexponential decay in both $\lambda_i$ and  $\ell$. Thus \eqref{eq:reindexedsum2} is an absolutely convergent sum,   which concludes the proof. 
\end{proof}
In order to prove the Mellin-Barnes type formula stated as  Proposition \ref{prop:mellinbarnes}, we will now rewrite the discrete sums over the $\lambda_i$ as integrals. Before doing this, we need some preliminary results. We recall the definition of the polygamma function 
\begin{equation} \label{eq:defpolygamma}
\psi_n(z)=\left( \frac{d}{dx} \right)^{n+1} \Gamma(z)=(-1)^{n+1} n! \sum_{k=0}^{\infty} \frac{1}{(z+k)^{n+1}}.
\end{equation} 
Known asymptotics of polygamma functions imply the following bound. 
\begin{lemma} \label{lem:boundcrossterm}
Let $z=r_1 e^{\I \tau_1}$ $w=r_2 e^{\I \tau_2}$, with $r_1,r_2 \in (0,\infty)$, $\tau_1,\tau_2 \in [-\frac{\pi}{3},\frac{\pi}{3}]$ and let $s$ be a complex number with positive real part in some compact set $S$. There exists a constant $C$, independent of $r_1,r_2,\tau_1,\tau_2,s$, such that 
\begin{equation} \label{eq:crossterm}
\left\vert  \frac{\Gamma(z+s) \Gamma(w+s)}{\Gamma(s) \Gamma(z+w+s)} \right\vert <C.\end{equation}
\end{lemma}
\begin{proof}
Taking the log of \eqref{eq:crossterm}, we may write 
\begin{multline} \label{eq:logcrossterm}
 \log \Gamma(z+s)-\log \Gamma(s)-(\log \Gamma(z+w+s)
-\log \Gamma(w+s))
=\int_0^{z} \psi_0(s+x) dx-\int_0^{z} \psi_0(w+s+x) dx \\
= -\int_0^{z} dx \int_0^{w} dy \psi_1(s+x+y) 
=-\int_0^{r_i} d t_1 \int_0^{r_j} dt_2 \psi_1(s+t_1 e^{\I \tau_1}+t_2 e^{\I \tau_2}) e^{\I (\tau_1+\tau_2)}.
\end{multline} 
where the integrals $\int_0^{z}$ should be interpreted as integrals along the straight line in the complex plane connecting $0$ and $z$. 
Our goal is to upper bound the real part of \eqref{eq:logcrossterm} as $r_1,r_2,\tau_1,\tau_2$ and $s$ vary. First note that if we set $u=t_1 e^{\I \tau_1}$ and $v=t_2 e^{\I \tau_2}$, the last integrand in \eqref{eq:logcrossterm} has the form $\psi_1(s+u+v) \frac{u v}{|uv|}$. By Stirling approximation, 
\[\psi_1(s+u+v)=\frac{1}{s+u+v}+O\left(\frac{1}{s+u+v}\right)^2
=\frac{1}{u+v}+O\left(\frac{s}{(u+v)^2}\right)+ O\left(\frac{1}{s+u+v}\right)^2.\]
This estimate yields 
\[\psi_1(s+u+v) \frac{u v}{|uv|}= \frac{uv}{|uv|} \left[ \frac{1}{u+v} +O\left(\frac{s}{(u+v)^2}\right)+ O\left(\frac{1}{s+u+v}\right)^2 \right].
\]
Recall that 
\begin{equation} \label{eq:argumentequality} \arg\left(\frac{uv}{u+v}\right)=\arg(u)+\arg(v)-\arg(u+v).
\end{equation}
Without loss of generality, we may  assume that $\arg(u)<\arg(v)$ and we recall that, by assumption, both arguments are in the interval $[-\frac{\pi}{3}, \frac{\pi}{3}]$. Then we have 
\[ \arg(u) \leq \arg(u+v) \leq arg(v).\]
Combining this with \eqref{eq:argumentequality} gives
\begin{equation} \label{eq:argbounds}\frac{-\pi}{3} \leq \arg(u) \leq \arg\left(\frac{uv}{u+v}\right) \leq \arg(v) \leq \frac{\pi}{3}. 
\end{equation}
In particular,  $ \left( \frac{uv}{|uv|(u+v)} \right)$ has positive real part. We may  bound the error terms $\frac{s}{(u+v)^2}$ and $\left(\frac{1}{s+u+v}\right)^2$ purely in terms of $r_1$ $r_2$ and the set $S$, so that for $r_1, r_2$ outside some fixed compact set independent of all variables, we have 
\[\Re \left[\psi_1(s+u+v) \frac{u v}{|uv|} \right] \geq 0.\]
Now \eqref{eq:logcrossterm} is the integral of a continuous function whose real part is negative outside a fixed compact set, and thus, it is bounded above.  
\end{proof}

In order to deduce  \eqref{eq:MellinBarnesintro} from  Proposition \ref{prop:premellinbarnes} all that remains is to use Mellin-Barnes type summation formula. This summation trick was first used in a similar context in  \cite[Lemma 3.2.13]{borodin2014macdonald}, and has been used many times in subsequent works.  
\begin{lemma}[Mellin-Barnes summation formula] \label{lem:mellinbarnes}
If $g$ is a meromorphic function with no poles that have real part larger than $a<1$, $\theta \in [-\frac{\pi}{2}, \frac{\pi}{2}]$  and for all $\tau \in [-\theta, \theta]$, $\lim_{r \to +\infty} r |g(r e^{\I \tau})|=0$ uniformly in $\tau$, then
$$\sum_{n = 1}^{\infty} (-\zeta)^n g(n)
=\int_{\mathcal C_{a}^{\theta}} \frac{ds}{2 \pi \I}  \frac{\pi}{\sin(-\pi s)} \zeta^s g(s),
$$
where we recall that the contour $\mathcal C_{a}^{\theta}$ is defined in the statement of Proposition \ref{prop:mellinbarnes}. 
\end{lemma}
\begin{proof} 
This relies on the fact that the residue of $\frac{\pi}{\sin(-\pi s)}$ at $s=n$, where $n$ is an integer,  is $(-1)^{n+1}$. The statement of Lemma \ref{lem:mellinbarnes} is essentially the same as \cite[Lemma 3.2.13]{borodin2014macdonald}, with slightly more explicit assumptions.  
\end{proof}

Now we may prove 
Proposition \ref{prop:mellinbarnes}.  It suffices to show that for each $\ell$ the summands in \eqref{eq:MellinBarnesintro} and   \eqref{eq:goodsummationformula} are equal. More explicitly, we want to show that for any $\ell \in \mathbb{Z}_{\geq 1}$, 
\begin{multline} \label{eq:ellequality}
   \sum^{\infty}_{\lambda_1,\dots,\lambda_{\ell}=1} \oint_{\gamma} \frac{dw_1}{2 \pi \I} \dots \oint_{\gamma} \frac{d w_{\ell}}{2 \pi \I} I_{\ell}(\lambda,w; -\zeta,x,t)=\\
    \int_{\mathcal C_{1/2}^{\pi/3}} \frac{d \lambda_1}{2 \pi \I} \dots  \int_{\mathcal C_{1/2}^{\pi/3}} \frac{d \lambda_{\ell}}{2 \pi \I}  \oint_{\gamma} \frac{dw_1}{2 \pi \I} \dots  \oint_{\gamma} \frac{d w_{\ell}}{2 \pi \I} \prod_{i=1}^{\ell} \frac{\pi}{\sin(-\pi \lambda_i)} I_{\ell}(\lambda, w;\zeta,x,t). 
\end{multline}
This identity can be obtained  by applying Lemma \ref{lem:mellinbarnes} to each sum over $\lambda_i$. Hence, we need to show that \begin{enumerate}
    \item When each $w_i$ is restricted to $\gamma$,  the integrand $I_{\ell}(\lambda,w;\zeta,x,t)$ has no poles in $\lambda_i$ that have real part larger than $1/2$.
    \item When each $w_i$ is restricted to $\gamma$, setting $\lambda_i=r_i e^{\I \tau_i}$ for any $\tau_1,\dots,\tau_\ell \in[-\frac{\pi}{3},\frac{\pi}{3}]$ and any $r_1,\dots,r_{\ell}>0$, the integrand $I_{\ell}(\lambda, w;\zeta,x,t)$ is bounded by 
\begin{equation} \label{eq:Iboundsimple} |I_{\ell}(\lambda, w; \zeta,x,t)| \leq \frac{C}{r_1 r_2\dots r_{\ell}},
\end{equation}
for some constant $C$.
\end{enumerate}
We begin by examining the poles of $I_{\ell}(\lambda,w;\zeta,x,t)$.  The determinant is explicitly given by Cauchy determinant identity 
\[\det \left[ \frac{1}{w_i+\lambda_i-w_j} \right]_{i,j=1}^{\ell}=\frac{\prod_{1\leq i<j \leq \ell} (w_j-w_i+\lambda_j-\lambda_i)(w_i-w_j)}{\prod_{i,j=1}^{\ell} (w_i+\lambda_i-w_j)}.
\]
Recall that $\Gamma(z)$ has poles at $z \in \mathbb{Z}_{\leq 0}$, and  $\frac{1}{\Gamma(z)}$ has no poles. Examining all possible poles and taking into account some cancellations, we see that, for fixed $w_i \in \gamma$ (a circle around $\mu$ with radius smaller than $1/4$),   $I_{\ell}$ has no poles in $\lambda_i$ with real part larger or equal than $1/2$. 

Now we only need to prove equation \eqref{eq:Iboundsimple}. We will actually prove the much stronger statement that $I_{\ell}$ has superexponential decay in every one of the $r_i$.  We will bound the most important factors of $I_{\ell}$ first. We use Stirling approximation, along with the fact that for $r>0$, $\theta \in (-\pi/2, \pi/2,)$, to obtain
\begin{equation} \label{eq:complexexponential} \left| (r e^{\I \theta})^{r e^{\I \theta}} \right|=e^{r \log(r) \cos(\theta)-r \theta \sin(\theta)} \sim e^{r \log(r) \cos(\theta)},
\end{equation}
so by Stirling
\begin{equation} \label{eq:bigdecay} \left| \frac{1}{\Gamma(s)} \right| = \left| \sqrt{\frac{s}{2 \pi}} \left( \frac{e}{s} \right)^s \right| \sim e^{-r \log(r) \cos(\theta)+r}  \to 0,\end{equation}
for $s=\lambda_i+\mu+w_i$ and $s=\lambda_i+\eta+w_i$, with $\lambda_i=r e^{\I \theta}$.
We see that this term has superexponential decay in $r_i=|\lambda_i|$. Now we only need to show that we can split the rest of $I_{\ell}$ into other factors each of which have at most exponential growth in any of the $r_i$. All $\lambda_i$ have argument with absolute value less than or equal to $\frac{\pi}{3}$. Thus we can uniformly bound each factor of the cross-term
\[\prod_{1 \leq i <j \leq \ell(\lambda)} \frac{\Gamma(w_i+w_j+\lambda_i) \Gamma(w_i+w_j+\lambda_j)}{\Gamma(w_i+w_j) \Gamma(w_i+w_j+\lambda_i+\lambda_j)},\]
 by applying Lemma \ref{lem:boundcrossterm} with $z=\lambda_i$, $w=\lambda_j$, and $s=w_i+w_j$. All remaining factors are controlled using  \eqref{eq:Gammaratios} and Stirling approximation, and can be easily shown to have at most exponential growth. This establishes \eqref{eq:MellinBarnesintro} and concludes the proof of Proposition \ref{prop:mellinbarnes}. 

\subsection{Return probability asymptotics}
\label{sec:returnprobaproof}
In this section we prove Corollary \ref{cor:returnproba}. We start from \eqref{eq:CauchytypeFredholm}. Let us examine the limit of this formula for large $t$. It is natural to expect (and we will prove)  that, as for usual random walks, $\msf P_{0, t}(1,1)$ should be of order $1/\sqrt{t}$. Hence we scale $\zeta$ as $\zeta =  \tilde \zeta \sqrt{2t}$. In the integrals \eqref{eq:defK} defining the kernel $K$, we will also scale variables as 
\begin{equation}w=\sqrt{t}/\tilde w, \;\;\;z=\sqrt{t}/\tilde z
\label{eq:changevariableskernel}\end{equation} 
Let us first consider the pointwise convergence. Under this change of variables, one immediately finds that for fixed $z\neq w$, 
$$ \sqrt{2t}K_{11}(z,w)\mathrm d z= O\left(1/\sqrt{t}\right), \;\;\sqrt{2t}K_{12}(z,w)\mathrm d z= O\left(1/\sqrt{t}\right), \;\;\sqrt{2t}K_{22}(z,w)\mathrm d z=O\left(1/\sqrt{t}\right).$$
However, when $z=w$, $K_{11}(z,z)=K_{22}(z,z)=0$ and 
$$ \sqrt{2t}K_{12}(z,z)\mathrm d z\xrightarrow[t\to\infty]{} \frac{1}{2\I\pi}\int_{\I\R} 2\sqrt{2}e^{\tilde z^2\mu^2} \mathrm d \tilde z.$$
The integral can be computed as 
$$ \frac{1}{2\I\pi}\int_{\I\R} 2\sqrt{2}e^{\tilde z^2\mu^2}\mathrm d \tilde z = \frac{1}{\mu}\frac{2}{\sqrt{2\pi}}.$$
Hence, we have the pointwise convergence \note{probably some tildes are missing, and a $d\tilde z$}
\begin{equation} \sqrt{2t}K(z,w)\mathrm d z \xrightarrow[t\to\infty]{} K^{\infty}(\tilde z,\tilde w)\mathrm d\tilde z:=\mathds{1}_{\tilde z=\tilde w}\frac{d\tilde z}{\mu}\frac{2}{\sqrt{2\pi}}\begin{pmatrix} 0 & 1\\-1&0\end{pmatrix}.
\label{eq:limitkernel}
\end{equation}
At this point, we claim that we may apply dominated convergence to conclude about the convergence of the Fredholm Pfaffian. Indeed, under the change of variables \eqref{eq:changevariableskernel}, and writing $\tilde z=\I y$ where $y$ is a real number,  
$$ \left( \frac{z^2}{z^2-\mu^2} \right)^t  = \exp\left(-t \log\left(1+\frac{\mu^2y^2}{t}\right)\right).$$
Using $t\log(1+x/t)\geqslant \log(1+x) $ for $x>0$ and $t>1$, we have that 
$$ \left\vert \left( \frac{z^2}{z^2-\mu^2} \right)^t \right\vert\leqslant \frac{1}{1+\mu^2y^2}.$$
Moreover, on the contour $\I\R$, we may bound the denominators in the kernel as  $\vert \pm z \pm w+ 1 \vert \geqslant 1$.  The remaining rational fractions have no singularity on the contour and go to $1$ as $t$ goes to infinity, after taking into account the Jacobian of the change of variables. Hence,  there exist a constant $c>0$ so that 
\begin{equation}\label{eq:half-spacebounds} \left\vert K_{11}(z,w)\frac{\mathrm d z}{\mathrm d \tilde z}\right\vert \leqslant  \left\vert \frac{c}{1-\mu^2\tilde z^2}\right\vert, \;\;\left\vert K_{12}(z,w)\frac{\mathrm d z}{\mathrm d \tilde z}\right\vert \leqslant  \left\vert \frac{c}{1-\mu^2\tilde z^2}\right\vert, \;\;\left\vert K_{22}(z,w)\frac{\mathrm d w}{\mathrm d \tilde w}\right\vert \leqslant \left\vert \frac{c}{ 1-\mu^2\tilde w^2}\right\vert.
\end{equation}
These bounds are clearly in $L^2(\I\R)$. Hence, by dominated convergence on each term of the Fredholm Pfaffian $\Pf (J+\zeta K)_{\mathbb L^2(\I\R\times \lbrace 1,2\rbrace)}$, and using Hadamard's bound to control the series (see Appendix \ref{sec:backgroundPfaffian} for details, in particular Lemma \ref{lem:welldefinedPfaffian}), the Fredholm Pfaffian  converges to the Fredholm Pfaffian of the limiting kernel $K^{\infty}$ given in the RHS of \eqref{eq:limitkernel}. Using the fact that for distinct variables $z_1, \dots, z_j$, 
$$ \mathrm{Pf}\left( \mathds{1}_{z_i=z_j} C \begin{pmatrix} 0 & 1\\-1&0\end{pmatrix} \right)_{i,j=1}^k = C^k,$$
we find that
$$\mathrm{Pf}(J+\tilde \zeta K^{\infty} ) \xrightarrow[t\to\infty]{}  \exp\left( \frac{\tilde \zeta}{\mu}\frac{2}{\sqrt{2\pi}} \right).$$
Finally, using Proposition \ref{prop:CauchytypeFredholm}, we obtain the following limit: uniformly for $\tilde \zeta$ in a compact subset of $\R$, we have
\begin{equation}
\mathbb E\left[F_{\eta+\mu}(-2\tilde \zeta \sqrt{2t} \msf P_{0, 2t}(1,1)) \right] \xrightarrow[t\to\infty]{} \exp\left( -\frac{\tilde \zeta}{\mu}\frac{2}{\sqrt{2\pi}} \right). 
\label{eq:convergenceFredholm}
\end{equation}
On the right hand side, we recognize the $F_{\eta+\mu}$-transform of a gamma random variable (up to a multiplicative factor). More precisely, if $G\sim \mathrm{Gamma}(\eta+\mu)$, then for any $z<0$, 
$$ \mathbb E\left[ F_{\eta+\mu}(zG)\right] =e^{z}.$$
Using Proposition \ref{prop:Levycontinuity}, we conclude that  for $\eta+\mu>1/2$, we have  the weak convergence 
\begin{equation}  
\sqrt{2\pi t} \msf P_{0,t}(1,1)\xRightarrow[t\to\infty]{} \frac{1}{\mu} \mathrm{Gamma}(\mu+\eta),
\label{eq:weakconvergenceGamma}
\end{equation} 
which proves Corollary \ref{cor:returnproba}. 

\subsection{Arbitrary starting point} 
\label{sec:arbitrarystratingpoint}
In this section, we explain why $\msf P_{0,t}(x,1)$ and $\msf P_{0,t}(x,0)$ asymptotically do not depend on $x$. In particular, for $\eta+\mu>1/2$ and for any fixed $x>0$, we have that 
\begin{equation}
\sqrt{t} \msf P_{0,t}(x,1) \xRightarrow[t\to\infty]{} \frac{2}{\sqrt{2\pi}}\frac{\mathrm{Gamma}(\eta+\mu)}{2\mu}
\label{eq:returnprobax1again}
\end{equation} 
and 
\begin{equation} \sqrt{ t} \msf P_{0,t}(x,0) \xRightarrow[t\to\infty]{} \frac{2}{\sqrt{2\pi}}\frac{\mathrm{Gamma}(\eta)}{2\mu}.
\label{eq:returnprobax0again}
\end{equation}
By the definition of the half-space beta RWRE,   \eqref{eq:returnprobax1again} implies \eqref{eq:returnprobax0again} (see in particular \eqref{eq:probato1and0}). Since we know by Corollary \ref{cor:returnproba} that \eqref{eq:returnprobax1again} holds when $x=1$,  it suffices to show that  $\sqrt{t}\left(\msf P_{0,t}(1,1)-\msf P_{0,t}(x,1)\right)$ goes to zero in distribution, and for that, it suffices to show that 
\begin{equation}
\label{eq:convergenceL1}
    \mathbb E\left[ \sqrt{t} \vert \msf P_{0,t}(1,1)-\msf P_{0,t}(x,1)  \vert \right] \xrightarrow[t\to\infty]{} 0.
\end{equation}
Let us fix $x>1$ and for simplicity, let us assume that $x$ is odd. Otherwise it would be more convenient to compare $P_{0,t}(0,1)$ and $\msf P_{0,t}(x,1)$, using a very similar argument. Let us consider two coalescing half-space beta random walks $X_x(t)$ starting from $x$ at time $0$  and $X_1(t)$ starting from $1$ at time $0$, sampled in the same environment. The marginal distribution of each random walk is the half-space beta RWRE. Denoting by $\mathsf E$ the expectation with respect to the measure on these coalescing random walks, we have 
\begin{equation}
\label{eq:differenceproba}
    \msf P_{0,t}(1,1) - \msf P_{0,t}(x,1) = \mathsf E\left[ \mathds{1}_{X_1(t)=1} - \mathds{1}_{X_x(t)=1} \right].
\end{equation}
If the random walks intersect before time $t$, the difference of indicator functions in \eqref{eq:differenceproba} will be zero since the random walks are coalescing. The coalescing random walks are weakly ordered for all time, so if $X_x(t)=1$, then $X_1(t)=1$. Hence, \eqref{eq:differenceproba} is simply the probability that $X_x$ and $X_1$ never intersect between times $0$ and $t$, and $X_1(t)=1$. Averaging over the environment, we obtain that 
\begin{align*} 
    \mathbb E\left[  \left\vert \msf P_{0,t}(1,1)-\msf P_{0,t}(x,1)  \right\vert \right] &= \mathbb E \left[ \mathsf P\left( X_x \text{ and }X_0 \text{ do not intersect on }[0,t] \text{ and }   X_1(t)=1 \right)\right]\\
    &=\mathbb P \left( \overline X_x \text{ and }\overline X_1 \text{ do not intersect on }[0,t] \text{ and }   \overline X_1(t)=1\right), 
\end{align*}
where $\overline X_x$ and $\overline X_1$ denote the coalescing random walks averaged over the environment. We may assume that $\overline X_1$ is a simple random walk on $\mathbb N$ as described in Remark \ref{rem:averagerandomwalk}, and $\overline X_x$ is a simple random walk absorbed when it hits the trajectory of $\overline X_1$. Hence,  we may write that 
\begin{equation*}
    \mathbb E\left[ \sqrt{t} \left\vert \msf P_{0,t}(1,1)-\msf P_{0,t}(x,1)  \right\vert \right] \leq \sqrt{t}\mathbb P(\overline X_1(t)=1)\times \mathbb P\left( \overline X_x(s)>1 \text{ for all } 0\leqslant s\leqslant t\right),
\end{equation*}
where $\overline X_x$ is now a simple random walk absorbed at $1$. We know that $\sqrt{t}\mathbb P(\overline X_1(t)=1)= \sqrt{t} \mathbb E \msf P_{0,t}(1,1)$ converges to $\frac{\eta+\mu}{\mu\sqrt{2\pi}}$. The probability that a simple random walk absorbed at $1$ and starting from $x>1$ never reaches $1$ between times $0$ and $t$ goes to zero as $t$ goes to infinity (this is the gambler's ruin problem \cite[chap. XIV]{feller2008introduction}). This implies \eqref{eq:convergenceL1}, and concludes the proof of Remark \ref{rem:arbitrarystartingpoint}. 

\appendix
\section{Pfaffians and Fredholm Pfaffians} 
\label{sec:backgroundPfaffian}
In this section we provide some background on the notions of Pfaffians and Fredholm Pfaffians, and give conditions for them to be well-defined. 
The Pfaffian of a skew-symmetric $2k\times 2k$ matrix $A=(a_{i,j})_{i,j=1}^{2k}$ is defined by 
\begin{equation}
    \label{eq:defPfaffian}
    \mathrm{Pf}(A) = \frac{1}{2^k k!} \sum_{\sigma\in S_{2k}} \prod_{i=1}^k a_{\sigma(2k-1)}a_{\sigma(2k)},
\end{equation}
and has the property that $\det(A)=\mathrm{Pf}(A)^2$. The notion of Fredholm Pfaffian was introduced in \cite{rains2000correlation}, as an analogue of the Fredholm determinant. Recall that given a measure space $(\mathbb X, \mu)$ and a kernel  $K:\mathbb X\times \mathbb X\to \mathbb R$, we define the Fredholm determinant $\det(I+K)_{L^2(\mathbb X,d\mu)}$ (we will omit the reference measure $d\mu$ when considering the Lebesgue measure on a subset of $\R^n$ or when considering the measure $\frac{1}{2\I\pi}dz$ on a contour of the complex plane, which is the case in this paper) by the series expansion 
\begin{equation}\label{eq:defFredholm}
    \det(I+K)_{L^2(\mathbb X,\mu)} = 1+\sum_{k=1}^{+\infty} \frac{1}{k!} \int_{\mathbb X}d\mu(x_1) \dots \int_{\mathbb X}d\mu(x_k)  \det\left( K(x_i,x_j) \right)_{i,j=1}^k,
\end{equation}
provided the right-hand-side is a convergent series (see \cite{simon2005trace} for a more general definition). When proving that such expansions are bounded, one often uses  Hadamard's bound: for a $k\times k$ matrix $M$ such that $\vert m_{i,j}\vert \le a_ib_j$ for all $1\le i,j\le k$, 
\begin{equation} \vert \det(M)\vert \le k^{k/2} \prod_{i=1}^ka_ib_i.
\label{eq:Hadamard}\end{equation} 

Consider now a skew-symmetric, matrix-valued kernel 
$$ K(x,y) = \begin{pmatrix} K_{11}(x,y) & K_{12}(x,y) \\ K_{21}(x,y ) & K_{22}(x,y) \end{pmatrix},\;\; x,y\in \mathbb X.$$
The Fredholm Pfaffian of $K$, denoted $\mathrm{Pf}(J+K)_{L^2(\mathbb X\times \lbrace 1,2\rbrace, d\mu)}$ (again, we will often omit the reference measure $d\mu$),  is defined by the series expansion 
\begin{equation}
    \mathrm{Pf}(J+K)_{L^2(\mathbb X\times \lbrace 1,2\rbrace, d\mu)} =1+ \sum_{k=1}^{+\infty} \frac{1}{k!} \int_{\mathbb X}d\mu(x_1) \dots \int_{\mathbb X}d\mu(x_k)  \mathrm{Pf}\left( K(x_i,x_j) \right)_{i,j=1}^k,
    \label{eq:defFRedholmPfaffian}
\end{equation}
provided the series converges. 

\begin{lemma}
Let $K:\mathbb X\times X\to \R$ be a $2\times 2$ matrix valued skew symmetric kernel. Assume that there exist  functions $f_{11}, f_{12}, f_{22}\in L^2(\mathbb X, d\mu)$ such that for all $x,y\in \mathbb X$, $$\vert K_{11}(x,y) \vert \le f_{11}(x), \; \vert K_{12}(x,y)\vert \le f_{12}(x), \; \vert K_{22}(x,y)\vert \le f_{22}(y).$$ 
Then, the Fredholm Pfaffian expansion of $\mathrm{Pf}(J+K)_{L^2(\mathbb X\times \lbrace 1,2\rbrace, d\mu)}$ in \eqref{eq:defFRedholmPfaffian} is convergent. 

Moreover, if a sequence of  $2\times 2$ matrix valued skew symmetric kernels $K^{(n)}$ converges pointwise to $K$ as $n$ goes to infinity, and for all $x,y\in \mathbb X$ we have the bounds (uniformly in $n$)
$$\vert K^{(n)}_{11}(x,y) \vert \le f_{11}(x),\; \vert K^{(n)}_{12}(x,y)\vert \le f_{12}(x),\; \vert K^{(n)}_{22}(x,y)\vert \le f_{22}(y),$$ 
then 
\begin{equation}
   \lim_{n\to+\infty} \mathrm{Pf}(J+K^{(n)})_{L^2(\mathbb X\times \lbrace 1,2\rbrace, d\mu)}  =  \mathrm{Pf}(J+K)_{L^2(\mathbb X\times \lbrace 1,2\rbrace, d\mu)}.
    \label{eq:convergencePfaffians}
\end{equation}
\label{lem:welldefinedPfaffian}
\end{lemma}
\begin{proof}
In the series expansion \eqref{eq:defFRedholmPfaffian}, we use the fact that $$\left\vert \mathrm{Pf}\left( K(x_i,x_j) \right)_{i,j=1}^k \right\vert = \sqrt{\left\vert \det\left( K(x_i,x_j) \right)_{i,j=1}^k \right\vert},$$
which, using Hadamard's bound \eqref{eq:Hadamard}, is bounded by 
$$ (2k)^{k/2}  \prod_{i=1}^k f_{11}(x_i)f_{12}(x_i)f_{12}(y_i)f_{22}(y_i).$$
Using Cauchy-Schwarz inequality, and the fact that $f_{11}, f_{12}, f_{13}\in L^2(\mathbb X, d\mu)$, yields that the series expansion in \eqref{eq:defFRedholmPfaffian} is bounded by 
$$\sum_{k=0}^{+\infty} \frac{1}{k!} \left( \Vert f_{11}\Vert_2 \Vert f_{12}\Vert_2^2 \Vert f_{22} \Vert_2 \right)^{k/4} (2k)^{k/2},$$
which is finite. The convergence \eqref{eq:convergencePfaffians} then follows from applying the dominated convergence theorem in the integrals and the series.  
\end{proof}

\section{Local limit theorem for the full-space beta RWRE}
\label{sec:localasymptoticsfullspace}
Let us start by recalling a formula for the moments of $\mathsf P^{\fullspace}(x,0)$, defined in \eqref{eq:defPfullspace}. For $x_1\ge \dots\ge x_k$ with $t+x_i$ even, we have 
\begin{multline}
\label{eq:momentformulafullspace}
\E\left[ \msf{P}^{\fullspace}_{0,t}(x_1,0)\dots \msf{P}^{\fullspace}_{0,t}(x_k,0)\right] =  (\alpha+\beta)_k  \\ \times  \int_{\gamma_1} \frac{dz_1}{2 \pi \I}\dots \int_{\gamma_k} \frac{dz_k}{2 \pi \I} \prod_{1 \leq a< b \leq k} \frac{z_a-z_b}{z_a-z_b-1} \prod_{i=1}^k \left( \frac{(\alpha+z_i)^2}{z_i(z_i+\alpha+\beta)} \right)^{t/2} \left( \frac{z_i+\alpha+\beta}{z_i} \right)^{\frac{x_i}{2}+1} \frac{1}{(z_i+\alpha+\beta)^2}, 
\end{multline}
where the contours $\gamma_1, \dots, \gamma_k$ all contain $0$, exclude $-\alpha-\beta$, and are nested in such a way that $\gamma_i$ contains $\gamma_j+1$ for all $1\le i<j\le k$. This formula is a variant of the moment formula from \cite[Proposition 1.11]{barraquand2017random}, given in \cite{thiery2016exact}, see also \cite[(3.1)]{corwin2017kardar}. We may deform the contours to become  a vertical line  $\eta+\I\mathbb R$, with $-\alpha-\beta<\eta<0$. Once the contours are all the same,  
using the symmetrization identity 
\begin{equation}
\sum_{\sigma\in S_k} \sigma\left(  \prod_{a<b} \frac{z_a-z_b-1}{z_a-z_b}  \right)= k!\label{eq:Snsymetrization},
\end{equation}
we obtain that 
\begin{equation}
    \E\left[ \msf{P}^{\fullspace}_{0,t}(x,0)^k\right] =  (\alpha+\beta)_k  \int_{\eta+\I\mathbb R} \frac{dz_1}{2 \pi \I}\dots \int_{\eta+\I\mathbb R} \frac{dz_k}{2 \pi \I} \det\left( \frac{1}{z_i-z_j-1} \right)_{i,j=1}^k \prod_{i=1}^k f_{t,x}(z_i),
\end{equation}
where 
$$ f_{t,x}(z) =  \left( \frac{(\alpha+z)^2}{z(z+\alpha+\beta)} \right)^{t/2} \left( \frac{z+\alpha+\beta}{z} \right)^{\frac{x}{2}+1} \frac{1}{(z+\alpha+\beta)^2}.$$
We refer to \cite[Proposition 3.2.2]{borodin2014macdonald} for similar type of derivations of determinantal formulas. We may now take the generating series to compute the Hankel transform, and obtain 
\begin{equation}
\label{eq:fullspaceFredholm}
    \mathbb E\left[ F_{\alpha+\beta}(\zeta \msf{P}^{\fullspace}_{0,t}(x,0))\right] = \det(I+\zeta L)_{L^2(\eta+\I\mathbb R)},
\end{equation}
 where the kernel $L$ is given by
 $$ L(z,z') = \frac{1}{z-z'-1}f_{t,x}(z,z').$$ 
Let us consider the scalings and changes of variables 
\begin{equation}\label{eq:scalingsfullspace}
  \zeta=\tilde \zeta\sqrt{t}, \:\; x=\frac{\beta-\alpha}{\beta+\alpha}t+\tilde x\sqrt{t},\;\;  z=\frac{\sqrt{2t \alpha\beta}}{\tilde z},\;\; z'=\frac{\sqrt{2t \alpha\beta}}{\tilde z'}.
\end{equation} 
Under these scalings, we see -- the computations are similar as in Section \ref{sec:returnprobaproof} -- that we have the pointwise convergence of the kernel
$$ \sqrt{t}L(z,z') \xrightarrow[t\to\infty]{} L^{\infty}(\tilde z, \tilde z') := \mathds 1_{\tilde z=\tilde z'} \frac{1}{\sqrt{2\alpha\beta}} \exp\left(\frac{\tilde z^2}{2}+\frac{x\tilde z}{\sigma}\right),$$
where $\sigma$ is as in Proposition \ref{prop:localasymptoticsfull-space}. Furthermore, the kernel can be bounded in a way similar as in  \eqref{eq:half-spacebounds}, that is, there exist constants $c, c'>0$ such that
$$\left\vert L(z,z')\frac{\mathrm d z}{\mathrm d \tilde z}\right\vert \leqslant  \left\vert \frac{c}{c'-\tilde z^2}\right\vert,$$
where we recall that $g_{\sigma}$ is the density function of the Gaussian distribution with variance $\sigma^2= \frac{4\alpha\beta}{(\alpha+\beta)^2}$, as in the statement of Proposition \ref{prop:localasymptoticsfull-space}. 
Thus, we may apply dominated convergence in the Fredholm determinant expansion. We find that 
$$\det(I+\zeta L)_{L^2(\eta+\I\mathbb R)} \xrightarrow[t\to\infty]{}  \det(I+\tilde \zeta L^{\infty})_{L^2(\I\mathbb R)} =\exp\left( \tilde \zeta \int_{\I\mathbb R} \frac{dz}{2\I\pi} L^{\infty}(z,z)\right) = \exp\left( \frac{\tilde \zeta g_{\sigma}(\tilde x) }{\alpha+\beta}\right).$$
Using \eqref{eq:fullspaceFredholm} and Proposition \ref{prop:Levycontinuity}, this shows that under the scalings \eqref{eq:scalingsfullspace}
$$ \sqrt{t}  \msf{P}^{\fullspace}_{0,t}(x,0) \xRightarrow[t\to\infty]{} g_{\sigma}(\tilde x) \frac{\mathrm{Gamma}(\alpha+\beta)}{\alpha+\beta},$$
which
concludes the proof of Proposition \ref{prop:localasymptoticsfull-space}. 

\section{Critical point asymptotics} 
\label{sec:asymptotics}

In this Section, we show that a critical point (non-rigorous) asymptotic analysis of the formula from Proposition \ref{prop:mellinbarnes} yields the statement from Conjecture \ref{conj:tracywidom}. 

All constants arising in the asymptotic analysis take the simplest possible form when they are written as functions of a real number $\theta>\mu$, which corresponds to the location of the critical point in the asymptotic analysis below. We start with  setting $\zeta_t(y)=e^{t a_{\theta}  -t^{1/3} b_{\theta} y}$, as in Lemma \ref{lem:hankellimit}, where the constants $a_{\theta} , b_{\theta}$ will be the same as in Conjecture \ref{conj:tracywidom}, and we explain below how these constants were determined. We may restate Proposition \ref{prop:mellinbarnes}, gathering all factors depending on $t$, as 
\begin{multline}\label{eq:startingasymptotics}
\E\left[F_{\mu+\eta} \left( -\zeta_t(y) \mathsf P_{0,2t}(x_{\theta}t,1)/4 \right)  \right]=
1+\sum_{\ell=1}^{\infty} \frac{1}{\ell!} \int_{\frac{1}{2}+e^{-\I \pi/3} \infty}^{\mu+\frac{1}{2} +e^{\I \pi/3} \infty} \frac{d z_1}{2 \pi \I} \dots \int_{\mu+\frac{1}{2}+e^{- \I \pi/3} \infty}^{\mu+\frac{1}{2}+e^{\I \pi/3} \infty} \frac{d z_{\ell}}{2 \pi \I} \\  \oint_{\gamma} \frac{dw_1}{2 \pi \I} \dots \oint_{\gamma} \frac{d w_{\ell}}{2 \pi \I}  
\prod_{i=1}^{\ell} \frac{\pi}{\sin(\pi (w_i-z_i))} \det  \left[\frac{1}{z_i-w_j} \right]_{i,j=1}^{\ell}
\prod_{1 \leq i <j \leq \ell} \frac{\Gamma(z_i+w_j) \Gamma(w_i+z_j)}{\Gamma(w_i+w_j) \Gamma(z_i+z_j)} \\
\times \prod_{i=1}^{\ell}\frac{ \Gamma(z_i+w_i)}{ \Gamma(2w_i)} 
\frac{ \Gamma(\eta+w_i) \Gamma(\mu+w_i) \Gamma(w_i-\mu)}{ \Gamma(\eta+z_i) \Gamma(z_i+\mu) \Gamma(z_i-\mu)}e^{t(F(z_i) -F(w_i))- t^{1/3} by(z_i-w_i) }. 
\end{multline} 
where $F(z)=G(z)+za_{\theta}$ and  the function $G$ is defined as in \eqref{eq:defG}. Choosing 
 $$x_{\theta}=\frac{2 \psi_1(\theta)-\psi_1(\theta+\mu)-\psi_1(\theta-\mu)}{\psi_1(\theta+\mu)-\psi_1(\theta-\mu)}, \qquad a_{\theta}=-G'(\theta),$$
 as in Conjecture \ref{conj:tracywidom} implies that $\theta$ is is a double critical point of the function $F$, i.e. 
 $F'(\theta)=F''(\theta)=0$, so that by Taylor expansion, 
\begin{equation}\label{eq:Taylor} F(z) = b_{\theta}^3 \frac{(z-\theta)^3}{3}+O\left((z-\theta)^4\right),\end{equation}
where $b_{\theta}$ is defined in Conjecture \ref{conj:tracywidom}. 
 Fix some $\theta>\mu$, and let us assume that we may deform the contours for the variables $w_i$ (resp. for the variables $z_i$) to some contour $\mathcal C$ (resp. some contour $\mathcal D$) going through $\theta$, in such a way that $z\mapsto \Real[G(z)]$ attains its minimum (resp. maximum) along $\mathcal C$ (resp. $\mathcal D$) precisely at $z=\theta$. We further assume that $z\mapsto \Real[G(z)]$ grows (resp. decays) sufficiently fast as $\vert z-\theta\vert$ increases. Then, we may approximate the integrals in \eqref{eq:startingasymptotics} by their contribution in a neighborhood of $\theta$. This is the principle of Laplace's method/steep descent analysis. Given the Taylor expansion \eqref{eq:Taylor}, it is natural to rescale variables as 
 $$z_i=\theta+b_{\theta}^{-1} t^{-1/3}\tilde z_i, \qquad w_i=\theta+b_{\theta}^{-1} t^{-1/3} \tilde w_i,$$
 and we will call $\widetilde{\mathcal C}, \widetilde{\mathcal D}$ the rescaled contours. Under this change of variables, we have 
 $$ tF(z_i) -t^{1/3} b_{\theta} yz_i =\frac{\tilde z_i^3}{3} -\tilde z_i y + O(t^{-1/3}).$$
 The other factors in \eqref{eq:startingasymptotics} can be approximated as well, and taking into account some simplifications, one arrives at  
 \begin{multline} \label{eq:limitexpansion}
     \lim_{t\to\infty} \E\left[F_{\mu+\eta} \left( -\zeta_t(y) \mathsf P_{0,2t}(x_{\theta}t,1)/4 \right)  \right]= \\ 
     1+\sum_{\ell=1}^{\infty} \int_{\widetilde{\mathcal C}}\frac{d\tilde w_1}{2\I\pi} \dots \int_{\widetilde{\mathcal C}}\frac{d\tilde w_{\ell}}{2\I\pi}   \int_{\widetilde{\mathcal D}}\frac{d\tilde z_1}{2\I\pi} \dots \int_{\widetilde{\mathcal D}}\frac{d\tilde z_{\ell}}{2\I\pi} 
     \det \left[ \frac{1}{\tilde z_i-\tilde w_j}  \right] \prod_{i=1}^{\ell} \frac{1}{\tilde w_i-\tilde z_i} \frac{e^{\frac{\tilde z_i^3}{3}-y \tilde z_i}}{e^{\frac{\tilde w_i^3}{3}-y \tilde w_i}}. 
 \end{multline}
 It is plausible that the contour $\widetilde{\mathcal C}$ (resp. $\widetilde{\mathcal D}$) can be deformed to become the union of two semi infinite rays going to infinity in directions $e^{\pm 2\I\pi/3} $ (resp. $e^{\pm \I\pi/3} $), in such a way that $\widetilde{\mathcal C}$ and  $\widetilde{\mathcal D}$ do not intersect. Then, the right hand side of \eqref{eq:limitexpansion} would be exactly  the Fredholm determinant expansion (see \eqref{eq:defFredholm}) of the Tracy-Widom GUE distribution 
 \begin{equation}\label{eq:defGUE}
     F_{\rm GUE}(y) := \det(I-K_{\Ai})_{L^2(y,+\infty)},
 \end{equation}
 where the Airy kernel is defined by  
 $$ K_{\Ai}(x,y) = \int_{\widetilde{\mathcal C}}\frac{d w}{2\I\pi}\int_{\widetilde{\mathcal D}}\frac{d z}{2\I\pi} \frac{1}{ z- w} \frac{e^{\frac{ z^3}{3}-x  z}}{e^{\frac{w^3}{3}-y w}}. $$
 Using Lemma \ref{lem:hankellimit}, this explains why we expect that 
$$
\lim_{t \to \infty} \P \left( \frac{\log \msf P_{0,2t}(x_{\theta} t, 1)-a_{\theta} t}{b_{\theta} t^{1/3}}\leq y \right)=F_{\rm GUE}( y).
$$

\renewcommand{\emph}[1]{\textit{#1}}
\bibliography{halfspacebeta.bib}
\bibliographystyle{goodbibtexstyle} 
\end{document}